\documentclass[reqno, 11pt]{smfart}
\usepackage[bbgreekl]{mathbbol}
\usepackage{amsmath,amssymb,amsthm,amsfonts, amscd, url, enumerate, leftidx}
\usepackage{appendix}
\usepackage[numeric]{amsrefs}
\usepackage[mathscr]{eucal}

\input xy
\xyoption{all}

\usepackage[colorlinks=true, pdfstartview=FitV, citecolor=PineGreen, urlcolor=blue]{hyperref}	

\makeatletter
\def\@tocline#1#2#3#4#5#6#7{\relax
      \ifnum #1>\c@tocdepth 
       \else
       \par \addpenalty\@secpenalty\addvspace{#2}%
       \begingroup \hyphenpenalty\@M
	       \@ifempty{#4}{%
		               \@tempdima\csname r@tocindent\number#1\endcsname\relax
		       }{%
		               \@tempdima#4\relax
		       }%
	      \parindent\z@ \leftskip#3\relax \advance\leftskip\@tempdima\relax
       \rightskip\@pnumwidth plus4em \parfillskip-\@pnumwidth
	       #5\leavevmode\hskip-\@tempdima
	       \ifcase #1
	       \or\or \hskip 1em \or \hskip 2em \else \hskip 3em \fi%
	       #6\nobreak\relax
	       \dotfill\hbox to\@pnumwidth{\@tocpagenum{#7}}\par
	       \nobreak
       \endgroup
	       \fi}
\makeatother

\usepackage[dvipsnames]{xcolor}
\usepackage{graphics}
\usepackage{graphicx}

\definecolor{cardinal}{rgb}{0.77, 0.12, 0.23}
\newcounter{Todo}

\usepackage{mathptmx} 
\usepackage[scaled=0.90]{helvet} 
\usepackage{courier} 
\normalfont
\usepackage[T1]{fontenc}
\newcommand{\bmu}{\mathbf{\mu}}

\newcommand{\vol}{\textsf{vol}}
\newcommand{\z}{\mathbf{z}}
\DeclareMathOperator{\gen}{gen}
\DeclareMathOperator{\Av}{Av}

\DeclareMathOperator{\Rep}{Rep}

\DeclareMathOperator{\Hom}{Hom}

\DeclareMathOperator{\Supp}{Supp}

\DeclareMathOperator{\iw}{Iw}

\newcommand{\da}{\bullet}

\newcommand{\ve}{\mathbf{v}}

\newcommand{\bv}{b^{\vee}}

\newtheorem*{nthm}{Theorem}
\newtheorem*{nprop}{Proposition}
\newtheorem*{nlem}{Lemma}
\newtheorem*{ncor}{Corollary}
\newtheorem*{nclaim}{Claim}

\newtheorem*{nrem}{Remark}

\theoremstyle{remark}

\newcommand{\be}[1]{\begin{eqnarray} \label{#1}}
	
	\newcommand{\ee}{\end{eqnarray}}

\newcommand{\tpoint}[1]{\subsubsection{#1}}

\newcommand{\spoint}{\subsubsection{}}

\numberwithin{equation}{section}

\newcommand{\hec}{\mathscr{H}}

\renewcommand{\sp}{\mathrm{sph}}
\newcommand{\sph}{\mathbf{1}_K}

\newcommand{\xv}{\check{\xi}}

\newcommand{\av}{\check{a}}

\newcommand{\I}{I}

\renewcommand{\tt}{\mathtt{t}}

\newcommand{\tV}{\wt{V}}
\newcommand{\tVs}{\tV_{\sp}}
\newcommand{\tX}{\t{X}}

\newcommand{\Fr}{\mathrm{Fr}}

\newcommand{\valu}{\mathsf{val}}

\newcommand{\la}{\langle}
\newcommand{\ra}{\rangle}

\newcommand{\mv}{\check{\mu}}

\newcommand{\mf}[1]{\mathfrak{#1}}

\newcommand{\rr}{\rightarrow}
\newcommand{\mc}[1]{\mathcal{#1}}
\newcommand{\wt}[1]{\widetilde{#1}}

\newcommand{\Lambdav}{\Lambda^\vee}
\newcommand{\muv}{\check{\mu}}
\newcommand{\Lv}{\Lambda^\vee}
\newcommand{\lv}{\check{\lambda}}
\newcommand{\ov}[1]{\overline{#1}}

\newcommand{\V}{\mathbb{V}}
\newcommand{\Vsp}{\V_{\sp}}
\newcommand{\tauv}{\check{\tau}}

\newcommand{\K}{\mathcal{K}}
\newcommand{\zee} {\mathbb{Z}}
\newcommand{\C} {\mathbb{C}}
\newcommand{\R}{\mathbb{R}}

\renewcommand{\O}{\mathcal{O}}

\renewcommand{\c}{\mathbf{c}}
\renewcommand{\b}{\mathbf{b}}
\newcommand{\B}{\mathsf{B}}

\newcommand{\resi}{\mathsf{res}}

\newcommand{\kk}{\kappa}
\newcommand{\nm}{n^{-}}

\newcommand{\fg}{\mathfrak{g}}

\newcommand{\opn}[1]{\operatorname{#1}}

\DeclareMathOperator{\Ext}{Ext}

\newcommand{\whit}{\mathscr{W}}

\newcommand{\nv}{\check{\nu}}
\newcommand{\rhov}{\check{\rho}}

\newcommand{\bU}{\mathbf{U}}

\newcommand{\Iop}{I^-}

\newcommand{\vbi}{\mathbf{v}^{\Iop}}

\usepackage{tikz}
\usetikzlibrary{matrix}
\usetikzlibrary{cd}

\begin{document}

\title{Metaplectic Covers of $p$-adic Groups and Quantum Groups at Roots of Unity}

\author{Valentin Buciumas}
\address[V.~Buciumas]{Korteweg--de Vries Institute, 
	The University of Amsterdam, 
	Science Park 107, 1090 GE, Amsterdam, 
	The Netherlands}
\email{valentin.buciumas@gmail.com}

\author{Manish M. Patnaik}
\address[M.~M.~Patnaik]{Department of Mathematical and Statistical Sciences, 
	University of Alberta, 
	Edmonton, AB T6G 2G1, 
	Canada}
\email{patnaik@ualberta.ca}


\newcommand{\tA}{\widetilde{A}}
\newcommand{\Pshi}{\widetilde{\Phi}}
\newcommand{\tG}{\widetilde{G}}
\newcommand{\tK}{\widetilde{K}}
\newcommand{\tU}{\widetilde{U}}
\newcommand{\an}{\mathbf{a}}
\newcommand{\nn}{\mathbf{n}^-}
\newcommand{\tT}{\wt{\mathscr{T}}}
\newcommand{\tB}{\widetilde{B}}
\newcommand{\tGg}{\tG^{gen}}

\newcommand{\un}{\mf{u}}
\newcommand{\mq}{\mathbf{\mu}_{q-1}}
\newcommand{\Q}{\mathbb{Q}}
\newcommand{\m}{\mathsf{n}}
\newcommand{\tH}{\widetilde{H}}

\newcommand{\g}{\mathbf{g}}

\newcommand{\T}{\mathscr{T}}
\newcommand{\spw}{\widetilde{\mathbf{1}}_{\psi}}
\newcommand{\hs}{\widetilde{\mathbf{1}}}

\newcommand{\tve}{\widetilde{\ve}}

\newcommand{\zv}{\check{\zeta}}

\newcommand{\sx}{\mathsf{x}}
\newcommand{\sh}{\mathsf{h}}
\newcommand{\sw}{\mathsf{w}}
\newcommand{\bh}{\mathbf{h}}
\newcommand{\bx}{\mathbf{x}}
\newcommand{\bw}{\mathbf{w}}
\newcommand{\bE}{\mathcal{E}}
\newcommand{\bA}{\mathbf{A}}
\renewcommand{\sc}{\mathsf{c}}

\newcommand{\tI}{\widetilde{I}}
\newcommand{\tIm}{\widetilde{I}^-}

\begin{abstract}
	We describe the structure of the Whittaker or Gelfand--Graev module on a  $n$-fold metaplectic cover of a $p$-adic group $G$ at both the Iwahori and spherical level. We express our answer in terms of the representation theory of a quantum group at a root of unity attached to the Langlands dual group of $G$.  To do so, we introduce an algebro-combinatorial model for these modules and develop for them a Kazhdan--Lusztig theory involving new generic parameters. These parameters can either be specialized to Gauss sums to recover the $p$-adic theory or to the natural grading parameter in the representation theory of quantum groups. As an application of our results, we deduce geometric Casselman--Shalika type results for metaplectic covers,  conjectured in a slightly different form by S. Lysenko, as well as prove a variant of G. Savin's local Shimura type correspondences at the Whittaker level. 
\end{abstract}

\maketitle

\setcounter{tocdepth}{2}

\tableofcontents

\renewcommand{\char}{1}
\newcommand{\J}{\mathscr{J}}
\newcommand{\LL}{\mathscr{L}}
\newcommand{\Jv}{\leftidx_{{\psi_{\nv}}}{\mathscr{J}}}
\newcommand{\id}{\mathrm{Id}}
\newcommand{\bG}{\mathbf{G}}
\renewcommand{\K}{\mathscr{F}}
\newcommand{\rep}{\mathrm{Rep}}
\newcommand{\whi}{\mc{W}}
\newcommand{\cs}{\mathrm{CS}}
\newcommand{\tcs}{\wt{\cs}}
\newcommand{\tJ}{\widetilde{\mathscr{\J}}}
\newcommand{\tL}{\widetilde{\mathscr{\LL}}}
\newcommand{\gf}{\mathbb{g}}

\newcommand{\fh}{\mf{h}}
\newcommand{\fb}{\mf{b}}
\newcommand{\fn}{\mf{n}}
\newcommand{\rts}{\mathscr{R}}
\newcommand{\rtl}{\mathscr{Q}}

\newcommand{\Piv}{\check{\Pi}}
\newcommand{\rtlv}{\check{\rtl}}
\newcommand{\fgv}{\check{\fg}}

\newcommand{\wts}{\mathrm{wts}}
\newcommand{\wtl}{P}
\newcommand{\bN}{\mathbf{N}}

\newcommand{\dw}{\dot{w}}

\newcommand{\affA}{\widehat{\As}}
\newcommand{\affW}{W_{\mathrm{aff}}}
\newcommand{\affI}{\widehat{I}}
\newcommand{\aff}{\mathrm{aff}}
\newcommand{\mult}{\mathsf{m}}
\newcommand{\U}{\mc{U}}

\newcommand{\Rv}{\check{\rts}}
\newcommand{\Xv}{\check{X}}
\newcommand{\Sv}{\check{S}}
\newcommand{\reg}{\mathrm{reg}}
\renewcommand{\c}{\mathbf{c}}
\renewcommand{\b}{\mathbf{b}}
\newcommand{\affH}{H_{\mathrm{aff}}}
\newcommand{\taffH}{\wt{H}_{\mathrm{aff}}}
\newcommand{\tsphH}{\wt{H}_{\mathrm{sph}}}
\newcommand{\ta}{\wt{a}}
\newcommand{\tav}{\ta^\vee}

\newcommand{\Gr}{\mathrm{Gr}}

\newcommand{\gv}{\tau}
\newcommand{\gvs}{\tau^2}
\newcommand{\Cvg}{\C_{\gv, \gf}}
\newcommand{\Cvgm}{\C_{\gv, \gf, \mf{m}}}

\newcommand{\Zvg}{\mathbb{Z}_{\gv, \gf}}

\newcommand{\Zvgm}{\zee_{\gv, \gf, \mf{m}}}

\newcommand{\tY}{\widetilde{Y}}

\newcommand{\Qs}{\mathsf{Q}}
\newcommand{\prim}{\mathrm{prim}}
\newcommand{\Bs}{\mathsf{B}}
\newcommand{\cat}{\mf{C}}

\newcommand{\wsp}{\mathscr{W}_{\psi}^{K}}
\newcommand{\hsp}{\mathscr{H}(G, K)}
\newcommand{\twsp}{\mathscr{\wt{W}}_{\psi}^{\mathrm{sph}}}
\newcommand{\thsp}{\mathscr{H}(\tG, K)}
\newcommand{\thiw}{\mathscr{H}(\tG, \Iop)}

\newcommand{\bn}{\mathbf{n}}

\newcommand{\whitiw}{\whit_{\psi}(\tG, \Iop)}

\newcommand{\tLv}{\Lv_0}

\newcommand{\As}{\mathsf{A}}

\newcommand{\inv}{\kappa_{\mathrm{KL}}}
\newcommand{\s}{\sigma}
\newcommand{\wh}[1]{\widehat{#1}}

\newcommand{\affS}{S_{\mathrm{aff}}}
\newcommand{\affSv}{\Sv_{\mathrm{aff}}}
\newcommand{\rtsv}{\check{\rts}}
\newcommand{\betav}{\check{\beta}}
\newcommand{\gammav}{\check{\gamma}}
\newcommand{\etav}{\check{\eta}}
\newcommand{\taffW}{\wt{W}_{\aff}}

\newcommand{\rat}{\mathbb{Q}}
\newcommand{\alc}{\mathscr{A}}
\newcommand{\Ch}{\mathscr{C}}
\newcommand{\Chv}{\check{\Ch}}
\renewcommand{\u}[1]{\underline{#1}}

\newcommand{\stab}{\mathrm{Stab}}

\newcommand{\X}{\mathscr{X}}

\renewcommand{\1}{\mathbf{1}}
\newcommand{\cf}{see\;}

\newcommand{\whitk}{\whit_{\psi}(\tG, K)}

\newcommand{\spaff}{H_{\sp}}
\newcommand{\tspaff}{\wt{H}_{\sp}}
\newcommand{\HL}{\mathsf{H}}

\newcommand{\ckl}{c}
\newcommand{\tckl}{\wt{\ckl}}
\newcommand{\tHsp}{\wt{H}_{\sp}}

\newcommand{\lket}[1]{[\mathscr{G}_{#1}]}
\newcommand{\lketm}[1]{[\mathscr{G}^-_{#1}]}

\newcommand{\zeev}{\zee_u}
\newcommand{\ratv}{\mathbb{Q}(\qv)}
\newcommand{\falg}{\mathbf{f}}
\newcommand{\divp}[2]{{#1}^{(#2)}}
\newcommand{\dU}{\dot{\mathbf{U}}}
\newcommand{\Ug}{\mathbf{U}}

\newcommand{\sm}{\mathscr{L}}
\newcommand{\cw}{\mathbf{c}^{\flat}}
\newcommand{\Tmwq}{\wt{\mathbf{T}}^{\flat}}	
\newcommand{\Tmw}{\wt{\mathbf{T}}}	
\newcommand{\x}{\mathsf{X}}
\newcommand{\Avg}{\Av^{\gen}}

\newcommand{\talc}{\overline{\alc}^{\da}_{+,n}}
\newcommand{\tpalc}{\u{\alc}^{\da}_{+,n}}
\newcommand{\tnalc}{\overline{\alc}^{\da}_{-,n}}

\renewcommand{\t}[1]{\widetilde{#1}} 
\newcommand{\qv}{u}

\newcommand{\vv}{\mathbb{v}}
\newcommand{\vvket}[1]{[\mathbb{v}_{#1}]}

\newcommand{\yket}[1]{[Y_{#1}]}
\newcommand{\vket}[1]{[\vv_{#1}]}

\newcommand{\At}{\zee_{\tau}}

\newcommand{\spmh}{\wt{H}_{\mathrm{sph}}}
\newcommand{\Th}{\Theta}

\newcommand{\sym}{\mathscr{P}}
\renewcommand{\sp}{\mathrm{sph}}
\newcommand{\ket}[1]{[Y_{#1}] }

\renewcommand{\v}{\tau^{-2}}
\newcommand{\vi}{\tau^2}
\newcommand{\kket}[1]{\mid {#1} \ra}


\newcommand{\gket}[1]{[G_{#1}]}
\newcommand{\gketm}[1]{[G^-_{#1}]}
 \newcommand{\lvecketm}[1]{[\mathbb{G}^-_{#1}]}
 \newcommand{\lvecket}[1]{[\mathbb{G}_{#1}]}


\renewcommand{\til}{\wt{\mc{T}}}


\newcommand{\lvbar}{\lv^\dagger}

\renewcommand{\vec}{\mathsf{v}}
\newcommand{\qlket}[1]{[\mathsf{G}_{#1}]}
\newcommand{\qlketm}[1]{[\mathsf{G}^-_{#1}]}
\newcommand{\vecket}[1]{[\vec_{#1}]}

\section{Introduction}

Let $\K$ be a non-archimedean local field and $G=\bG(\K)$ the points of a simple (split) Chevalley group. To a pair $(\Qs, \ell)$ where $\ell$ is a positive integer and $\Qs$ is a Weyl group invariant quadratic form on the coweight lattice of $G$, we may associate a metaplectic $\ell$-fold covering group $\tG$ of $G$. 
On the other hand, for a specific choice of $\Qs$ (essentially the `minimal' or `primitive' choice, \cf \S\ref{subsub:exampleltwisted}) and any $\ell$, we can also construct another object, namely the (big) quantum group $\dU_{\zeta}\left(\check{\bG}\right)$ of the dual group $\check{\bG}$ at an $2\ell$-th root of unity $\zeta \in \C$.
The main aim of this paper is twofold: first, to elucidate the structure of the space of spherical Whittaker functions on $\tG$, denoted $\whitk$ and sometimes called the spherical Gelfand--Graev representation, as a module over the spherical Hecke algebra $\thsp;$ and second, to express this structure in terms of the category of finite-dimensional representations  $\Rep\left( \dU_{\zeta}(\check{\bG}) \right)$ equipped with the action (by quantum Frobenius) of the category of representations of a complex algebraic group $\Rep\left( \check{\bG}_{\ell} \right)$. In doing so, we believe that a number of curious phenomenon in the representation theory of $p$-adic covering groups that have been trickling into the literature since the pioneering work of Kazhdan--Patterson \cite{kazhdan:patterson} can now be paired to better studied counterparts in the quantum group world where they often have a more transparent meaning. For example, throughout this introduction we focus on the case of `geometric' Casselman-Shalika formulas for metaplectic groups which, as far as we know, were both unavailable and perhaps not thought to have a reasonable form by the $p$-adic community prior to this work. 
We see now that such formulas record the complexity of decomposing certain tensor products of quantum group representations. The latter problem is noticeably more involved than the corresponding one for representations of complex algebraic groups, which (as we explain later in this introduction) is connected to geometric Casselman--Shalika problems for linear groups. 
As another application, we formulate a local Shimura correspondence (\cf\cite{savin:localshimura}) at the level of the Gelfand--Graev representation that encapsulates certain `exceptional'\footnote{Here `exceptional' in the metaplectic world means that it behaves `normally,'  or as in the linear group case } phenomenon in $\whitk$ observed by Gao--Shahidi--Szpruch~\cite{gao:shahidi:szpruch} and provide for it an interpretation in terms of tensor product theorems for irreducible (and indecomposable tilting) modules of quantum groups.

Let us mention that an important source of both motivation and inspiration for us comes from the following picture. Replacing $\whitk$ with a certain category of twisted sheaves on an affine Grassmannian attached to $\bG$, an equivalence to the category of representations of the same quantum group at a root of unity was conjectured by D. Gaitsgory and J. Lurie in \cite{ga:twisted}, described in a form closest in spirit to this work, by Lysenko \cite{lys}, and studied further in \cites{gl:parameters, gl:fle, campbell-dhillon-raskin}. In fact it was the Casselman--Shalika formula conjectured by S. Lysenko \cite[Conjecture 11.2.4]{lys} linking these twisted sheaves to quantum groups  which lies at the origin of this work, and one of our main results is a verification of a version of this conjecture at the level of graded Grothendieck groups. Indeed, whereas the cited works are concerned with relating certain categories of sheaves with categories of representations, the present work draws out aspects of this correspondence visible at the (function theoretic) level of $p$-adic Whittaker integrals on metaplectic groups. 

To connect the $p$-adic and quantum pictures, we introduce a third `generic' space $\Vsp$ depending on certain formal parameters that can be specialized to either the Gauss sums in the $p$-adic world or to the grading within the representation category of the quantum group. This space carries an action of a certain spherical Hecke algebra $\tsphH$, and the connections we  describe in this work can be summarized pictorially 
\noindent 
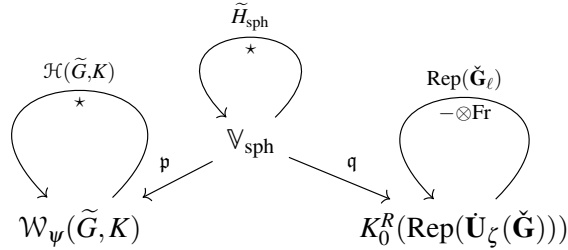
\begin{center} 
	\begin{tikzcd}[row sep =1em, column sep = 2em]
		& \Vsp  \arrow[loop, "\star", "\tsphH"']\arrow[dl, "\mf{p}"'] \arrow[dr, "\mf{q}"] &  \\ 
		\whitk \arrow[loop, "\star", "\thsp"']&  & 	\arrow[loop, "-\otimes \Fr", "\Rep(\check{\bG}_{\ell})"' ] K_0^R(\Rep(\dU_{\zeta}(\check{\bG})))
	\end{tikzcd} 
\end{center} 
where $K_0^R(\Rep(\dU_{\zeta}(\check{\bG})))$ is a graded version of the Grothendieck group (more precisely, the enriched right Grothendieck ring of \cite{ClineParshallScott:enriched}). In other words, the structure of $\whitk$ as a $\thsp$-module is the same (up to certain specializations involving Gauss sum) as that of $K_0^R(\Rep(\dU_{\zeta}(\check{\bG}))$ with the action of $K_0(\Rep(\check{\bG}_{\ell})).$

To describe $\Vsp,$ let us first recall some aspects of the structure of  $\Rep\left( \dU_{\zeta}(\check{\bG}) \right)$, a category which in many ways resembles the representation category of algebraic groups in positive characteristic. Irreducible objects in $\Rep\left( \dU_{\zeta}(\check{\bG}) \right)$ have an intricate structure and computing their characters in terms of known characters of (co)standard objects is a hard problem that was described through certain conjectures of Lusztig~\cite{lus:conjecturecharp} (and now mostly all proven) and involves certain (affine, parabolic)  Kazhdan--Lusztig polynomials \cite{KL-Cox, deodhar:87}. 
Now, the combinatorial nature of $\Rep\left( \dU_{\zeta}(\check{\bG}) \right)$ and certain \emph{canonical bases} for them were further elaborated upon through the works of Lascoux--Leclerc--Thibon~\cite{lascoux:leclerc:thibon,leclerc:thibon} in type $A$ (see \cite{haiman:grojnowski, laniniramsobaje, laniniram} for extensions to general type), and these works play a decisive role in this paper. 
The space $\Vsp$ has a similar construction to the `Fock' space in \cite{leclerc:thibon}, but must now be enhanced to include the Gauss sum parameters, which are inbuilt to the action of $\tsphH$ and so seem essential to include. 
What surprised us was that this `upgrade' to Gauss sums can \textit{a posterori} be `hidden' if bases are chosen appropriately. In other words, we may choose basis elements in $\Vsp$ so that the action 
given by the metaplectic Demazure-Lusztig operators of Chinta--Gunnells--Pusk\'as \cite{cgp, chinta:gunnells} (which are central to studying $\whitk$) exactly match the action of the affine Hecke algebras in \cite{leclerc:thibon}. To connect to the $p$-adic setting, one need to put back in the Gauss sums. This can be done in a prescribed manner and this process motivated us to introduce a Gauss-sum twisted version of Kazhdan--Lusztig theory based on an involution built out of the operators mentioned above. The polynomials in this theory are presumably connected to the metaplectic Macdonald polynomials introduced recently by \cite{ssv1, ssv2}, \cf \cite{beck:jing:frenkel}. We might also note here that these Gauss sums (in the $p$-adic context) contain crucial arithmetic information needed for the applications of metaplectic Whittaker functions (\emph{e.g.} to the theory of multiple Dirichlet series, see~\cite{bbf:wgmds:stable,bbf:wgmds:stable:applications}).

Before describing our main results in more detail, let us note that there are other areas in which similar types of affine Weyl group combinatorics arise, \textit{e.g.} affine Lie algebras at critical level,  representation theory of algebraic groups in positive characteristic, and so our results might equally link the metaplectic world to these areas. However, we prefer to work in the quantum world where earlier works by the first named author and collaborators (see~\cite{bbbg:iwahori, bbb, bbbg:metahori}) uncovered connections between certain affine quantum groups at generic values (so not at roots of unity) and metaplectic groups. Though this is \textit{not} the connection described here, the present work seems to `predict' this other connection, as will be explained in forthcoming work by the first named author.

\newcommand{\dualG}{\check{\bG}}

\tpoint{Formulas of Casselman--Shalika type for linear groups} \label{intro:geom-cs}  Before turning to covers, let us recall a few results of Casselman--Shalika type for linear (i.e. trivial cover) groups. These results are well-known, but our presentation here is influenced by our work in the non-linear setting. As above, write $G:= \bG(\K)$ and consider $\hsp,$ the spherical Hecke algebra of $G$ with respect to a maximal compact subgroup $K$. It is an algebra under convolution and is equipped with a natural $p$-adic basis $h_{\lv}$ given by characteristic functions of double cosets $K\pi^{\lv}K$ with $\lv \in Y_+$ a dominant coweight. The Satake isomorphism  \cite{sat}, as reinterpreted by Langlands \cite[Ch. 2]{lan:ep}, gives an identification
\be{} \label{sat:intro} S: \hsp \stackrel{\sim}{\rr} K_0(\Rep(\dualG(\C)))\ee 
between the spherical Hecke algebra and the representation ring of the dual group $\dualG(\C)$ of $G.$  
To each irreducible highest weight representation $V_{\lv} \in \rep(\dualG(\C))$ we can assign an element $\ckl_{\lv} \in \hsp$ and an expression for such an element in terms of the $h_{\lv}$ basis of $\hsp$ is described by a formula due to Kato and Lusztig \cites{lus:sing, kato:sph} (it involves certain affine, parabolic  Kazhdan-Lusztig polynomials).

Let $\psi$ be a non-trivial additive character of $\K$ of conductor $0$, extended to $U= \bU(\K)$, the unipotent radical of some fixed Borel subgroup, and consider $\whit_{\psi}(G, K)$ the space of compactly supported functions on $G$ which are $(U, \psi)$-left invariant and right $K$-invariant. 
To each coweight $\mv \in Y$, we can consider the unique function $\J_{\mv} \in \whit_{\psi}(G, K)$ which takes the value $1$ on $\pi^{\mv}$ and is supported on $U \pi^{\mv} K.$ One finds that  $\J_{\mv}$ is only well-defined  when $\mv \in Y_+$ and that such elements form a $\C$-basis of $\whit_{\psi}(G, K)$.  Denoting by $\star$ the natural right convolution of $\hsp$ on $\whit_{\psi}(G, K),$ for  $\mv, \lv \in Y_+,$ we may ask to rewrite the product $ \J_{\mv} \star h_{\lv}$ in terms of the basis $\{ \J_{\zv} \}_{\zv \in Y_+}$. For $\mv$ fixed, and $\lv$ large and dominant (compared to $\mv$), the answer to this question is given by the usual Casselman--Shalika formula \cite{cs}. 
To explain this, denote by 
\be{} \cs(\mv) :=\prod_{a > 0 } ( 1 - q^{-1} Y_{ - \av}) \chi_{\muv}(Y) \ee
the formula for the unramified spherical Whittaker function found in \emph{op. cit.}, where the notation used is as follows:  $Y_{-\av}$ represents an element in the group algebra of coweights $\C[Y]$ corresponding to the simple coroot $\av$; $\chi_{\mv}(Y)$ is the Weyl character for the irreducible representation $V_{\mv}$ of $\dualG(\C)$ of highest-weight $\mv;$ $q$ is the cardinality of the residue field of $\K$; and the product is over all positive roots of $G$. In analogy with \cite[Thm 6.6]{lus:sing}, which works in the context of affine Hecke algebras, one may show  
\be{} \label{classical:csatmu} \begin{array}{lcr} \J_{\mv} \star h_{\lv} = \cs(\muv) \circ \J_{\lv} & \mbox{ where we set } & Y_{\mv} \circ \J_{\zv} = \J_{\mv + \zv}.\end{array} \ee
It is important to note that upon applying $\circ$, one may encounter elements in the right hand side of the form  $\J_{\lv + \zv}$ where $\lv + \zv$ is no longer dominant. So to correctly interpret \eqref{classical:csatmu} as a formula in $\whit_{\psi}(G, K)$, one needs to \emph{formally} introduce elements $\J_{\mv}$ for $\mv \in Y \setminus Y_+,$ subject to the following `straightening' rules in terms of the `dot' action $\da$ (see \eqref{dot-action}) of the Wey group : 
\be{} \label{Lusztig-straightening-rule} \J_{\mv}= \begin{cases} - \J_{\mv \da s_{a} } & \text{ if } \la \mv + \rhov, a \ra  \neq 0   \\ 0  & \text{ if } \la \mv + \rhov, a \ra = 0,  \end{cases} \ee where in the above formula $a$ is a simple root and $\rhov$ is the half-sum of the positive coroots. Notice that if $\mv$ is fixed and $\lv$ is chosen very large, we can arrange for all the terms in the right hand side of \eqref{classical:csatmu} to be dominant, i.e. no straightening rules are involved, and one just recovers a $\lv$-shift of the usual Casselman--Shalika formula. Note that such a relation for $\lv$ large compared to $\mv$ can be proven directly using a limiting procedure, i.e. independently of the computation of the formula in \cite{cs}.  For this reason, we might refer to the formula for $\cs(\mv)$ in  \cite{cs} as the \emph{asymptotic} Casselman--Shalika formula. 
As a corollary of \eqref{classical:csatmu}, (see \cite[Corollary 6.8]{lus:sing}), one deduces the existence of an element $\mathscr{C}_{\lv} \in \hsp$ for $\lv \in Y_+$ that satisfies  both 
\be{geom:cs}  \J_{0} \star \mathscr{C}_{\lv} = \J_{\lv} \ee and the condition that, under the Satake map \eqref{sat:intro},  \be{} \label{geom:cs-2} S(\mathscr{C}_{\lv}) = [V_{\lv}], \text{ so that in fact  } \mathscr{C}_{\lv} = \ckl_{\lv} , \ee the element introduced above.  As the elements $\ckl_{\lv}$ have a natural description in a geometric context, namely within the category of perverse sheaves on the affine Grassmannian, the geometric analogue of \eqref{geom:cs}  is often called the \emph{geometric} Casselman--Shalika formula, \cf \cite{fgkv, fgv, ngo, ngo:polo}. Through a slight abuse of notation, one often calls the formula \eqref{geom:cs} by the same name. Using \eqref{geom:cs}, we compute 
\be{eq:geomCSatmu}
\J_{\mv} \star \ckl_{\lv} = \J_{0} \star \ckl_{\mv} \star \ckl_{\lv} =  \J_{0} \star \sum_{\etav} c_{\muv, \lv}^{\etav} \ckl_{\etav} = \sum_{\etav} c_{\muv, \lv}^{\etav} \J_{\etav},
\ee 
where $c_{\muv, \lv}^{\etav} := \dim_{\C} \Hom (V_{\etav}, V_{\muv} \otimes V_{\lv})$ are the Littlewood--Richardson coefficients for $\dualG(\C)$. The above results show that the map $h \mapsto \J_0 \star h$ from $\hsp \rr \whit_{\psi}(G, K)$ is an isomorphism of $\hsp$ modules sending $\ckl_{\lv}$ to $\J_{\lv}$. Combined with the Satake isomorphism, we also obtain an isomorphism $$\whit_{\psi}(G, K) \simeq K_0\left(\Rep(\dualG(\C))\right)$$ intertwining the $\hsp$ action on $\whit_{\psi}(G, K)$ with the $K_0\left(\Rep(\dualG(\C))\right)$-action (by tensor product) on itself. 

\tpoint{Metaplectic Casselman--Shalika type problems} Let us now turn to the case of an $n$-fold metaplectic cover $\tG$ of $G.$ One can again construct a space of (genuine) spherical Whittaker functions $\whitk$ equipped with a right convolution  by the spherical Hecke algebra $\thsp$. 
The space $\whitk$ has a basis denoted $\tJ_{\lv}, \lv \in Y_+$ and constructed in a similar fashion to $\J_{\lv}.$ 
On the other hand, although one can define elements $\t{h}_{\lv} \in \thsp$ for any $\lv \in Y$ as above, it turns out that unless $\lv$ lies in a certain sublattice of finite index $\tY_+ \subset Y_+$, the element $\t{h}_{\lv}$ is not well-defined. 
This `reduction in support' is a hallmark of the metaplectic world.  Nonetheless, one may work with $\thsp$ in a similar manner to $\hsp$, and in fact Savin~\cites{savin:localshimura,mcnamara:ps} showed $\thsp \cong K_0(\Rep(\dualG_{(\Qs, n)}(\C)))$ for a complex group whose root data is determined from that of $\bG$ and the metaplectic structure $(\Qs, n)$ used to define $\tG$. 

One can define $\t{\ckl}_{\lv} \in \thsp$ (uniquely) for any $\lv \in \tY_+$  to satisfy the first condition in \eqref{geom:cs-2}. Although the condition \eqref{geom:cs} breaks down, i.e.     $\tJ_{0} \star \t{\ckl}_{\lv} \neq \tJ_{\lv}$  for some $\lv \in \tY_+,$
we show in this paper that the relation~\eqref{classical:csatmu} persists if we replace $\cs(\mv)$ with its metaplectic version $\tcs(\mv)$ as studied in \cite{PPAIM, mcnamara:ps}, \emph{i.e.} 
\be{} \label{j:cs-met} \tJ_{\mv} \star \t{h}_{\lv} = \tcs(\mv) \circ \tJ_{\lv}, \text{ where } \mv \in Y_+, \lv \in \tY_+, \ee 
where to make sense of the right hand side we introduce certain `metaplectic' straightening rules, see Proposition \ref{prop:straightening-sph}. These rules are computationally involved,  but as mentioned above, they remarkably (at least to us) recover the straightening rules found some time ago by Lascoux--Leclerc--Thibon~\cite{leclerc:thibon} in type $A$ (see \cite{haiman:grojnowski} and \cite{laniniramsobaje} for the extension to general type) under a quantum specialization. Inspired by \cite{leclerc:thibon}, we use these straightening rules to define an involutions of Kazhdan--Lusztig type that allows us to construct two new bases $\{ \tL_{\lv} \}_{\lv \in Y_+}$ and $\{ \til_{\lv} \}_{\lv \in Y_+}$ of $\whitk$ starting from our original basis $\tJ_{\lv}.$ We call these the \emph{canonical bases} in analogy with their quantum counterpart. 
One should understand the bases $\tL_{\lv}$, $\tJ_{\lv}$ and $\til_{\lv}$ as $p$-adic versions of the irreducible, (co)standard and indecomposable tilting modules for the corresponding quantum group at roots of unity.
Let us note that the existence of these new bases suggests three natural metaplectic\footnote{If one performs the same procedure in the linear case, one would find that both canonical bases agree with the standard basis.} analogues of the linear geometric Casselman--Shalika problem \eqref{geom:cs}: \begin{itemize}
	\item  compute $\tJ_{\mv} \star \t{c}_{\lv}$ for $\lv \in \tY_+$ and $\mv \in Y_+$; and 
	\item  compute $\tL_{\mv} \star \t{c}_{\lv}$ and $\til_{\mv} \star \t{c}_{\lv}$ for $\lv \in \tY_+$ and $\mv \in Y_+$. 
\end{itemize}
As it turns out, each of these questions can be answered.

\tpoint{Quantum groups at roots of unity} The quantum groups of relevance in this paper are Lusztig's dotted version with divided powers. To define them, one constructs an `integral' form (or rather a $\zee[\qv, \qv^{-1}]$ form, where $\qv$ is the deformation parameter in the quantum group) and then specializes the variable $\qv$ to a root of unity $\zeta \in \C$. The corresponding object will be called $\dU_{\zeta}(\bG)$ where $\bG(\C)$ is the complex group attached to some root datum. In the main body of this paper, we adopt a slighly different notation, but still emphasize that  the quantum group we construct depends on a root datum, and not just a Cartan datum, i.e. not just the data needed to specify a semi-simple Lie algebra.   What is of particular importance for us is not the quantum group itself, but the structure of its (graded) representation category of finite-dimensional modules $\Rep(\dU_{\zeta}(\bG))$\footnote{The dotted and non-dotted quantum groups at a root of unity have equivalent representation theory (see~\cite[\S31]{lus:qg} or~\cite[\S3.7]{AndersenParadowski:fusion}).}. 
When $\qv$ is specialized to a root of unity, the representation theory of the quantum group diverges from the complex representation theory of semi-simple Lie algebras. For example, although irreducible highest weight modules from the complex semi-simple theory can be deformed to objects $\Delta_{\lambda} \in  \Rep(\dU_{\zeta}(\bG))$ for $\lambda \in X_+$ a dominant weight, such modules may be reducible. One can still compute their character by the Weyl character formula, and so these modules are often called Weyl or standard modules. Understanding their irreducible quotients $L_{\lambda}$ was the subject of conjectures put forth by Lusztig and answered by combining the work of several groups, see~\cite{KazhdanLusztigequivalence12,KazhdanLusztigequivalence3, KazhdanLusztigequivalence4, KashiwaraTanisaki1, KashiwaraTanisaki2, AndersenJantzenSoergel}. In addition to the irreducible and (co)standard modules, there is another important class of modules in $\Rep(\dU_{\zeta}(\bG))$ called the indecomposable tilting modules $T_{\lambda}$, again indexed by $\lambda \in X_+.$ The relation between the $T_{\lambda}$ and the (co)standard modules were the subject of conjectures and then theorems of Soergel \cite{soergel:combinatoric, soergel:tilting}.

In contrast to the representation theory of complex Lie algebras, the category $\Rep(\dU_{\zeta}(\bG))$ is no longer semi-simple as one has non-trivial extensions between irreducible objects. As such, the usual Grothendieck group construction loses important information about the category, and one may work with various enhancements to try to recapture this lost data. It seems to be well-understood that the category is also graded in the sense of \cite{soergel:icm} and hence its Grothendieck group has a natural $\zee[\tau, \tau^{-1}] 
$-structure.  We could not find an explicit reference in the literature outside of rank $1$ (see  \cite{andersen:tubbenhauer}), so as a substitute we work here with the so-called left and right enriched Grothendieck rings $K_0^L(\Rep(\dU_{\zeta}(\bG)))$ and $K_0^R(\Rep(\dU_{\zeta}(\bG)))$ of \cite{ClineParshallScott:enriched}. They are $\zee[\tau, \tau^{-1}]$-modules with basis indexed by the classes of standard objects $[\Delta_{\lambda}]$ and costandard objects $[\nabla_{\lambda}]$, respectively, where $\lambda \in X_+$. In these spaces one has expansions
\be{}\label{intro:L_to_simple} \begin{array}{lcr} 
[L_{\lambda}] = \sum_{\mu \in X_+ \cap \lambda \da W_{\aff} } o^-_{\lambda, \mu}(\tau) [\nabla_{\mu}], & \text{and conjecturally} & [T_{\lambda}] = \sum_{\mu \in X_+ \cap \lambda \da W_{\aff} } o^+_{\lambda, \mu}(\tau^{-1}) [\Delta_{\mu}],   
\end{array} 
\ee  
where $o^{\pm}_{\lambda, \mu}(\tau)$ are certain parabolic Kazhdan--Lusztig polynomials and $\da$ is the dilated dot action on the weight lattice. In contrast to this, working with the usual Grothendieck ring, one finds similar relations, but with a specialization of the polynomials $o^{\pm}_{\lambda, \mu}(-1)$ that produces character formulas for irreducibles (Lusztig's conjecture) and for indecomposable tiltings (Soergel's conjecture~\cite[\S7]{soergel:combinatoric}).

For a  certain \emph{algebraic} group $\bG_{\ell}$ constructed from the root data for $\bG$ and the integer $\ell$, there exists a functor $\Fr: \Rep(\bG_{\ell}(\C)) \rr \Rep(\dU_{\zeta}(\bG))$ which equips the latter category with an action by the former. We write this as $W, V \mapsto W \otimes \Fr(V)$ where $W \in \Rep(\dU_{\zeta}(\bG)), V \in \Rep(\bG_{\ell}(\C))$ and call it \emph{quantum Frobenius}. That the construction of the group $\bG_{\ell}$ parallels the construction of the algebraic group controlling the spherical Hecke algebras of metaplectic $\ell$-fold covers is an observation, often regarded as a curiosity in the metaplectic community (\cf \cite{weis}), that we believe can be given a suitable context through the present work. One may pose the following questions in $K_0^R(\Rep(\dU_{\zeta}(\bG)))$:  
\begin{itemize}
	\item write $\nabla_{\mu} \otimes \Fr(V_{\lambda})$ in terms of the $\nabla_{\eta}$ (or more precisely, understand a $\nabla$-filtration of $\nabla_{\mu} \otimes \Fr(V_{\lambda})$); 
	\item write $L_{\mu} \otimes \Fr(V_{\lambda})$ and $T_{\mu} \otimes \Fr(V_{\lambda})$ in terms of the $\nabla_{\eta}$. 
\end{itemize}
As it turns out, all of these questions can be answered. The second is the subject of the so-called Steinberg-Lusztig theorem, which states that $L_{\mu} \otimes \Fr(V_{\lambda}) \simeq L_{\mu + \lambda}$ whenever $\mu$ is in some `restricted' set of weights and $\lambda$ is in the dominant $l$-weight lattice $X_{l,+}$ of $\bG_{\ell}$. The last is the subject of the tilting tensor product theorem, which states that for $\mu$ `restricted' one can define a new weight $\mu^{\dagger}$ such that $T_{\mu^{\dagger}} \otimes \Fr(V_{\lambda})= T_{\mu^{\dagger}+ \lambda}$ (see \S\ref{subsub:quantum:tensor:products} for more details). As for the first question, the answer is given by  the $q$-Littlewood--Richardson coefficients of Lascoux--Leclerc--Thibon~\cite{lascoux:leclerc:thibon,leclerc:thibon} (type $A$) and Haiman--Grojnowski~\cite{haiman:grojnowski} (general type).

\tpoint{Main result and some consequences} As mentioned above, the connection between the metaplectic and quantum worlds goes through a combinatorial model consisting of a representation $\Vsp$ of the spherical subalgebra $\tspaff$ of $\taffH$. We define a ring $\Zvg$ (see , see \S\ref{notation:Cvg}) which depends on both a parameter $\gv$ as well as a family of other parameters $\gf_k$ modelling the behavior of certain Gauss sums from the $p$-adic world, and $\Vsp$ is a module over it. 
The space $\Vsp$ can be obtained as a quotient of a representation $\V$ on which the affine Hecke algebra $\taffH$ acts via the metaplectic Demazure-Lusztig operators of \cite{cgp, PPAIM}. 
Moreover, $\Vsp$ has a natural  basis $[\vv_{\mv}], \mv \in Y_+$ and an involution of Kazhdan--Lusztig type which produces new bases denoted $\lvecket{\mv}$ and $\lvecketm{\mv}$. Working with the basis $\vvket{\mv}$, the space $\Vsp$ behaves identically to the Fock spaces considered in Leclerc--Thibon \cite{leclerc:thibon} -- in other words, it is the one in which the Gauss sum parameters can be `hidden'. To make a connection to the $p$-adic world however, we work with renormalizations $\yket{\mv}$,  $\gket{\mv}$, and $\gketm{\mv}$ of $[\vv_{\mv}]$, $\lvecket{\mv}$, and $\lvecketm{\mv}$,  respectively, that makes the Gauss sum parameters again manifest.

By specializing the parameters $\gv, \gf_k$ accordingly, the space $\Vsp$ recovers $\whitk$ under what we call a  \emph{$p$-adic specialization} $\mf{p}$ and $K_0^R(\Rep(\dU_{\zeta}(\check{\bG}))$ under what we call a  \emph{quantum specialization} $\mf{q}$.  Each of the spaces $\whitk$, $\Vsp$, and $K_0^R(\Rep(\dU_{\zeta}(\check{\bG}))$ carry natural actions by essentially isomorphic algebras $\mathscr{H}(\tG, K)$,  $\tsphH$, and  $K_0(\Rep( \check{\bG}_l))$, respectively. Our main result asserts that $\mf{p}$ and $\mf{q}$ intertwine these actions. 

\begin{nthm}[see Thm. \ref{thm:main-thm}]\label{thm:main-thm:intro} 
Let $\ell$ be a positive integer.
Let $(\Qs, \ell)$ be a metaplectic twist on the group $\bG$ and $\tG$ the corresponding $\ell$-fold metaplectic cover. Assume that $\bG$ and  $\bG_{(\Qs, \ell)}$ are of simply-connected type. 
\begin{enumerate}
\item There exists an isomorphism that that intertwines the $\tsphH$ and $\thsp$ actions, and denoted again as 
\be{}
\mf{p}: \C \otimes_{\Zvg} \Vsp \stackrel{\simeq}{\longrightarrow} \whitk& &\textrm{ sending } [Y_{\mv}] \mapsto \tJ_{\mv}, \ee 
 The map $\mf{p}$ sends $\gketm{\mv}$ and $\gket{\mv}$ to the canonical basis $\tL_{\mv}$ and $\til_{\mv}$, respectively. 
\item 	If $\ell$ is larger than the Coxeter number of $\check{\mf{D}}$ and KL-good (see \S\ref{sub:extension}), there exists an isomorphism 
\be{eq:mainthmqside} \mf{q}: \zee[\tau, \tau^{-1}]  \otimes_{\Zvg} \Vsp \stackrel{\simeq}{\longrightarrow} K_0^R(\Rep(\dU_{\zeta}(\check{\bG}))) & &\textrm{ sending } \yket{\mv} \mapsto [\nabla_{\mv}] \mbox{ for } \mv \in Y_+,
\ee that intertwines the $\tsphH$ and $K_0(\Rep(\check{\bG}_{\ell}))$ actions. Moreover, the map $\mf{q}$ sends $\gketm{\mv}$ to $[L_{\mv}].$ 
\end{enumerate} 
\end{nthm}

\begin{nrem} 
	\begin{enumerate}
		\item  A (conjectural) relation between $\gket{\mv}$ and the tilting modules $T_{\lv}$ is explained in \S \ref{subsub:tiltings:dualities}. This depends on certain facts about graded decomposition numbers which we could not find in the literature on quantum groups (though we believe they may be known). 
		
		\item The hypothesis that $\bG_{(\Qs, \ell)}$ is simply connected can be easily discarded if one works with extended affine Hecke algebras; we only impose it to simplify some of the exposition and computations. On the quantum side, the hypothesis that  $\check{\bG}$ is of adjoint type and the conditions on $\ell$ are more serious: they allow us to use a generalization of Lusztig's conjecture whose proof depends on the Kazhdan--Lusztig equivalence \cite{KL:IMRN}. We do not believe the KL-good condition is necessary for the proof of~\eqref{eq:mainthmqside} (and therefore Lysenko's conjecture), see Remark \ref{rem:KLgoodnotneeded}.    
	\end{enumerate} 
\end{nrem}
	
\noindent One can now obtain solutions to the geometric Casselman--Shalika problems described above 

\begin{ncor}[See Thm. \ref{thm:met-geom-CS-LR} and Prop.~\ref{prop:LtoN:Ko}] \label{intro-geom-CS-cor} 
Let $\zv \in \tY_+$.
\begin{enumerate} 
\item If $\lv \in Y_+$ is `restricted' (see \S \ref{subsub:boxes}) then
$\tL_{\lv} \star \t{c}_{\zv} = \tL_{\lv + \zv}.$
\item If $\lv \in Y_+$  is `restricted' and we define  $\lvbar:= \lv_0 \cdot w_0 + 2 (\t{\rho}^{\vee} - \rhov),$ where  $\t{\rho}^{\vee}$ is the analogue of $\rhov$ for $\bG_{(\Qs, n)}$,  then  $\til_{\lvbar} \star \tckl_{\zv} = \til_{\lvbar + \zv}.$
\item There exist $\leftidx^{\gf}Q^{\etav}_{\mv, \lv} \in \Zvg$ such that with respect to $\mf{p}: \Zvg \rr \C$ and $\mf{q}: \Zvg \rr \zee[\tau, \tau^{-1}]$ 
\be{eq:QcoeffLysenko}  \begin{array}{lcr} 
	\tJ_{\mv} \star \t{c}_{\lv}  = \sum_{\etav} \mf{p} (\leftidx^{\gf}Q^{\etav}_{\mv, \lv} ) \tJ_{\etav} & 
	\mbox{ and }  &  
	[\nabla_{\mv} \otimes \Fr(V_{\lv})] = 
	\sum_{\etav \in Y_+} \mf{q} (\leftidx^{\gf}Q^{\etav}_{\mv, \lv} ) [\nabla_{\etav}].   
\end{array} \ee \end{enumerate} \end{ncor} 
\noindent 
Let us note the following about Corollary~\ref{intro-geom-CS-cor} (3). First, the left side of~\eqref{eq:QcoeffLysenko} describes the action of the spherical Hecke algebra on the basis $\tJ_{\muv}$ of $\whitk$ (which is the most natural basis from the $p$-adic perspective; the bases $\tL_{\muv}$ and $\til_{\muv}$ have a more geometric flavor).  
The answer is given in terms of Gauss sum twisted versions of the $q$-Littlewood--Richardson coefficients which can still be computed by a similar combinatorial algorithm as for the untwisted version. Note that these latter objects are the building blocks of LLT polynomials and objects of considerable combinatorial interest (see~\cites{haiman:grojnowski,HHL}).

Second, part (3) answers a version of Lysenko \cite[Conjecture 11.2.4]{lys} at the level of enriched Grothendieck groups. This seems simultaneously both `less and more' than the original conjecture--  less as \emph{op. cit.} works at the level of the derived category whereas we work at the level of the enriched Grothendieck group; and perhaps more (though this may be only due to the authors' ignorance) since our arguments naturally produce the elements $\leftidx^{\gf}Q^{\etav}_{\mv, \lv}$ containing arithmetic information in the form of Gauss sums which we do not see in Lysenko's conjecture. 
It would be interesting to understand the representation theoretic or geometric significance of $\leftidx^{\gf}Q^{\etav}_{\mv, \lv}$.

\tpoint{Iwahori level analysis}\label{subsub:Iwahorilevel} Although the results described above are at the spherical level, nearly everything we have said requires an Iwahori level analysis. Let $\Iop$ be the Iwahori subgroup attached to the Borel with unipotent radical $U^-$ (the one opposed to $U$). 
Consider $\whitiw$ the space of (genuine) left $(U, \psi)$ and right $\Iop$ invariant compactly supported functions on $\tG$ which carries an action of the Iwahori--Hecke algebra $\thiw$. 
The structure of $\thiw$ was first described by G. Savin \cite{savin:localshimura, savin:crelle, mcnamara:ps}. As in the linear case, $\thiw$ has two different descriptions, an Iwahori--Matsumoto description and a Bernstein type presentation, the latter of which identifies it as $H_W \otimes \C[\tY].$ Now, as was observed in \cite{PPAIM}, the larger space $\C[Y]$ carries the action of certain metaplectic Demazure--Lusztig operators $\Tmw_{w}$ for $w \in W$, as well as the translation action of $\C[\tY]$ and hence can be equipped with a $\thiw$-action. 

\begin{nthm}[see Prop. \ref{prop:averaging-isom} and  \ref{thm:i-basis}] The natural averaging map (restricted to the big cell), 
\be{}\label{average:intro} \Avg_{U^-}:  \whitiw \stackrel{\simeq}{\longrightarrow}  \C[Y] \ee 
is an isomorphism of vector spaces which intertwines the action of $\thiw$ on $\whitiw$ by convolution and the action via translation and metapletic Demazure--Lusztig operators on $\C[Y]$. \end{nthm}

The proof uses techniques from \cite{PPAIM} combined with ideas from \cite{leclerc:thibon}. It gives   
\begin{itemize}
	\item a natural basis $\{ \mathscr{Y}_{\lv} \}_{\lv \in Y}$ of $\whitiw$ defined by $\Avg_{U^-}(\mathscr{Y}_{\lv})= Y_{\lv}$ for $\lv \in Y$, and  an explicit description of the action of $\thiw$ in this basis;
	\item a decomposition into a sum of (explicitly described) $\thiw$-modules (see Prop. \ref{prop:VtoP})
\be{} \whitiw = \oplus_{\etav \in \tnalc} \whitiw(\etav), \ee 
where $\tnalc$ is the upper-closure of a certain alcove defined by an affine root system connected to $\tY$.
\end{itemize}

\begin{nrem} The module $\whitiw$ is the $p$-adic specialization of a $\Zvg$- module $\V$ (see \S\ref{sub:met-poly})  for an affine Hecke algebra $\taffH$. Although we do not pursue it here, $\V$ can be equipped with a Kazhdan-Lusztig type basis that specializes to one for $\whitiw$ (and which can be related to the natural $p$-adic basis $\mathscr{Y}_{\lv}$ above). 
 \end{nrem}

\tpoint{Local Shimura correspondence and other applications} The Iwahori analysis presented above allows us to formulate a Gelfand--Graev version of Savin's local Shimura correspondence~\cite{savin:localshimura}; namely that the $\thiw$-submodule $\whitiw(-\rhov)$ and the  $\thsp$-submodule $\whitk(-\rhov)$ behave essentially like their counterparts in the linear (or non-metaplectic) setting for the group $\bG_{(\Qs, n)}$ (see Propositions~\ref{prop:classicalIwahoriCS} and~\ref{cor:CSisospherical}). We use this to prove in Proposition~\ref{cor:CSisospherical} a `linear' \emph{geometric} Casselman--Shalika formula  in the metaplectic world (\cf \cite{gao:shahidi:szpruch} for an \emph{asymptotic} Casselman-Shalika formula at the same coweight). 
\begin{nprop} 
Let $\t{\rho}^{\vee}$ be the analogue of $\rhov$ for $\bG_{(\Qs, n)}$ and let $\muv \in \tY_+$. Then
$\tJ_{\t{\rho}^\vee-\rhov} \star \t{c}_{\muv} = \tJ_{\t{\rho}^\vee-\rhov + \muv}$.
\end{nprop}

A corollary of this result is the fact that for coweights of the form $\t{\rho}^\vee-\rhov + \muv$, the $p$-adic elements $\tJ_{\t{\rho}^\vee-\rhov + \muv}$ and the canonical elements $\tL_{\t{\rho}^\vee-\rhov + \muv}$ and $\til_{\t{\rho}^\vee-\rhov + \muv}$ are equal.  
This result has a nice interpretation on the quantum side: the irreducible, indecomposable tilting, and (co)standard modules with (Steinberg) weights $\t{\rho}-\rho + \mu$ for $\mu$ in the dominant $l$-dilated weight lattice $X_{l,+}$ are equal (see \cite[Corollary 6.8]{mcgerty:cmp},~\cite{andersen:CMP}).

Other applications in \S\ref{sub:questions} include relations of the transition coefficients from the basis $\tJ_{\lv}$ to the basis $\tL_{\lv}$ with the strong linkage principle from the theory of quantum groups (\cites{Andersen:linkage,AndersenParadowski:fusion}), a $\mv$-large asymptotic version of the geometric Casselman-Shalika formula (not to be confused with the formula \eqref{classical:csatmu}  discussed earlier where $\mv$ small compared to $\lv$). We also add some speculations about the relation of our work to that of Frenkel--Hernandez~\cite{fh:dual} and McGerty~\cite{mcgerty:cmp}.

\tpoint{Relation to Existing Literature} This work was inspired by an attempt to understand the $p$-adic significance of some of the ideas and conjectures in the work of Lysenko \cite{lys} within the general program initiated in \cite{ga:twisted}. The works of Lascoux--Leclerc--Thibon, especially \cite{leclerc:thibon}, played a central role in this paper as we hope the introduction has already made clear. We found the extensions of this work in the papers of Grojnowski--Haiman and Lanini--Ram(--Sobaje) to be especially helpful, and our proof of the tensor product theorems follows the elegant argument in~\cite{laniniram} using the Littelmann path model.
We note here that the second tensor product theorem we prove in Proposition~\ref{prop:TensorProd} (and which is inspired by Andersen's tensor product theorem for tilting modules~\cite{andersen:CMP}) seems to be new, though we emphasize that the method of proof is a simple modification of the argument in~\cite{laniniram}.
As a small repayment towards our debt to these works, let us perhaps mention that our Proposition~\ref{prop:qLRCS} produces natural Demazure--Lusztig type operators, built out of the Chinta--Gunnells actions (suitably abstracted from their $p$-adic origins), which encapsulate the actions considered in \cite{leclerc:thibon}. Whether the suggested connection to Casselman--Shalika type problems yields anything new on the combinatorial side remains to be investigated.

A link between quantum affine groups (at generic parameter) and metaplectic Whittaker functions appeared in the work by the first author and collaborators \cites{bbb,bbbf,bbbg:cmp,bbbg:iwahori,bbbg:metahori} and which has gaven rise to: connections to (super-symmetric) LLT polynomials in type A~\cite{bbbg:cmp}, metaplectic representations of the affine Hecke algebra~\cite{bbbf}, the combinatorial study of metaplectic Whittaker functions via the theory of lattice models~\cite{bbb,bbbg:metahori} and a theory of vector-valued metaplectic Demazure--Lusztig operators~\cite{bbbg:metahori}.
We will explore the connections between these ideas and the current paper in future work.

The interesting recent work of Sahi--Stokman--Venkateswaran~\cites{ssv1,ssv2} builds on ideas similar to those we used from~\cite{leclerc:thibon}. 
The affine Hecke algebra representation we introduce in \S\ref{subsub:met-poly} appears in~\cite[Theorem 3.7]{ssv1} from an algebraic perspective. 
A highlight of \cites{ssv1,ssv2} is the extension of this space to a representation of the double affine Hecke algebra, which produces a natural basis depending on an additional parameter. 
This is used in the construction of `metaplectic' Macdonald polynomials that are known to specialize to values of metaplectic Iwahori--Whittaker functions. 
It would be interesting to better understand the relations between our work and theirs, in particular to understand the $p$-adic meaning of the DAHA structure and the extra parameter it introduces.  
Moreover, the polynomials $\leftidx^{\gf}Q^{\etav}_{\mv, \lv}$ may be used to construct $\gf$-twisted LLT polynomials by using the formula after \cite[Eq. (13)]{haiman:grojnowski} and it seems natural to relate these polynomials to the metaplectic Macdonald polynomials of~\cites{ssv1,ssv2} in the spirit of~\cites{haiman:grojnowski,HHL}.

Building upon (what we understand were) ideas and suggestions by G. Savin and S. Lysenko, the recent work \cite{ggk} approaches $\whitiw$ by working at the level of pro-$p$ Iwahori subgroups.  Using this work, one can extract an alternate proof of Proposition \ref{thm:i-basis}.  As our proof is based on ~\cite{PPAIM}, we do not actually need any pro-$p$ techniques. Nonetheless, one hopes that the pro-$p$ methods are more than just a technical tool that can be circumvented and may in fact be an essential feature of the theory.

\tpoint{Remarks on the organization of the paper} We refer the reader to the table of contents which hopefully makes our organizational scheme transparent. For a table of frequently used notation, we refer to \S \ref{sec:freq:used:notation}.  Let us just comment here that the longest section of this work, Section \ref{sec:qKL}, introduces both preliminaries, but also revisits the works of Leclerc-Thibon \cite{leclerc:thibon} in such a way that the applications to $p$-adic groups and the metaplectic polynomial representation follow quite easily. The key new results in Section \ref{sec:gKL} are the `straightening' rules for the Chinta--Gunnells--Pusk\'as operators which allow us to produce $\gf$-twisted versions of the results in Section~\ref{sec:qKL} and connect them to our $p$-adic results in \S\ref{sub:IwahoriAction}. The main $p$-adic results in \S\ref{sub:p-adic:spherical} make use of this connection. The proof of our main result in \S\ref{sub:main} follows immediately from what we established before. To conclude our paper, we present in \S\ref{sub:local:shimura:correspondence} some applications of our work to the `local' Shimura correspondence and investigate natural combinatorial questions in \S\ref{sub:questions}.

\tpoint{Acknowledgements} M.P. would like to thank A. Braverman for a helpful discussion about topics related to this work and J. Sussan for very useful guidance into the world of quantum groups at a root of unity and especially their categorical aspects. 
V.B. would like to thank Dinakar Muthiah, Anna Pusk\'as for interesting discussions and encouragement, Hankyung Ko for many illuminating discussions about quantum groups and Jasper Stokman for interesting discussion regarding this work and~\cites{ssv1,ssv2}. 

Both authors acknowledge funding by the endowment of M.V. Subbarao Professorship in number theory as well as NSERC RGPIN-2019-06112. V.B. acknowledges support by the Netherlands Organization for Scientific Research (NWO) project number 613.009.126.
We also acknowledge the Newton Institute for its hospitality during the program on ``New Connections in Number Theory and Physics'' where some parts of this work were carried out.

\section{Notations and conventions}

\noindent \textbf{General conventions.}

\spoint \label{notation:functor} Bold faced objects denote functors (of groups usually) and the corresponding roman letters will denote the field-valued points. For example, $\bG$ will be a group-valued functor and $G$ will denote $\bG(F)$, where $F$ is a field, assumed to be specified implicitly in our discussion.  

\spoint For an abelian group $H$ and (commutative, unital) ring $R$, we write $R[H]$ for the group algebra of $H$ with coefficients in $R$ and typically denote the elements in $R[H]$ as $H_a, H_b$ etc. with $a, b \in H$ and with multiplication defined as $H_{a} H_b = H_{a+b}$. 

\tpoint{The Grothendieck group} \label{subsub:grothendieck-groups} Let $\mathcal{C}$ be an abelian $\C$-linear monoidal category (not necessarily semi-simple) whose objects have finite length.
The Grothendieck group of $\mathcal{C}$ shall be denoted by $K_0(\mathcal{C})$ and its complexification $\C \otimes_\zee K_0(\mathcal{C})$ by $K_0^{\C}(\mathcal{C})$. To every object in $X \in \mathcal{C}$ we write its class as $[X] \in K_0(\mathcal{C}).$ 

\spoint Typically, we reserve the symbol $F$ for a general field and $\K$ for a non-archimedean local field. 
\newline

\noindent \textbf{Notation for $p$-adic fields and Gauss sums} 

\tpoint{Non-archimedean local fields} \label{subsub:notationpadic} Let $\K$ be a non-archimedean local field with ring of integers $\O$ and valuation map $\valu: \K^* \rr \zee.$ Let $\pi \in \O$ be a uniformizing element, $\kk:= \O / \pi \O$ be the residue field, $\omega: \O \rr \kk$ the natural surjection, and let us write $q$ for the cardinality of $\kk$. Denote by $\varpi: \O \rr \kappa$ the natural quotient map. For $k \geq 0$, set $\O^*[k]= \{ x \in \K^* \mid \valu(x) = k \} $ and $ \O(k) = \{ x \in \K^* \mid \valu(x) \geq k \} $ so that the units in $\O$ are $\O^*=\O^*[0]$, where $\O^*[k]= \O(k) \setminus \O(k+1)$.

\tpoint{Hilbert symbols, assumptions on $q$ and $n$} \label{notation:stein-symb} For $A$ an abelian group, a \emph{bilinear Steinberg symbol} is a map $(\cdot, \cdot): F^* \times F^*  \rr A$ such that $(\cdot, \cdot)$ is bimultiplicative, i.e.,  $(x, yz)  = (x, y) (x, z) $ and $(xy, z) = (x, z) (y, z)$ and also $(x, 1-x) = 1$ if $x \neq 1.$ In this paper, we focus exclusively on the case of (tame) Hilbert symbols. To define these, assume $q \equiv 1 \mod 2n $ and let  $\bmu_n \subset \K$ be the set of $n$-th roots of unity with $| \bmu_n |=n$. The $n$-th order Hilbert symbol (see e.g. \cite[\S9.2, 9.3]{serre:lf}) is a bilinear map $( \cdot, \cdot)_n : \K^* \times \K^* \rr \bmu_n.$ As $n$ is fixed throughout our paper, we often drop it from our notation. Note that $(\cdot, \cdot)$ is a bilinear Steinberg symbol (\cf \cite[Chapter V, Proposition 3.2]{neu}) and is also unramified, \textit{i.e.} $(x, y)=1$ if $x, y \in \O^*.$ To avoid certain sign issues, we assumed  $q \equiv 1 \mod{2n},$ since we now have $(-1, -1)=(-1, x)=1$ for $x \in \K^*$, and also $(\pi, \pi)=1$ and $(\pi, u) = \varpi(u)^{\frac{q-1}{n}}$ for $u \in \mc{O}^*.$

\tpoint{Additive characters} \label{subsub:additive-characters} Let $\psi: \K \rr \C^*$ be an additive character. For $a \in \zee$, if $$\psi|_{\O(a)}= \id|_{\O(a)} \text{ and } \psi\mid_{\O[a-1]} \text{ is non-trivial}, $$ we say that $\psi$ has \emph{conductor} $a$. We are chiefly interested characters $\psi$ of conductor $0$\footnote{Note that in~\cite{chan:savin, ggk},  the conductor of the additive character is taken to be $-1$; this introduces a ``$\rho$''-shift when comparing formulas with these sources. On the other hand, our choice of conductor is in alignment with the classical work of Casselman--Shalika \cite{cs} and the literature on its geometric version \cite{fgkv, fgv}. }, i.e. $\psi$ is the identity on $\O$ and non-trivial on $\pi^{-1}\O.$   

\tpoint{Gauss sums} \label{subsub:GaussSums} For $\tau: \K \rr \C^*$  an additive character and $\epsilon: \K^* \rr \C^*$ a multiplicative one, set $\g(\sigma, \tau) = \int_{\O^*} \epsilon(u') \tau(u') du' $ where $du'$ is the Haar measure on $\K$ giving $\O^*$ volume $q-1.$ This is a \textit{Gauss sum}. 
Fix$\psi$ an additive character of conductor $0$ and define for each integer $k$ a multiplicative and additive character by $ \begin{array}{lcr} \sigma(u) = \sigma(u, \pi)_n^{-k} & \text{ and } & \tau(u) = \psi(- \pi^{-1} u), \end{array} \text{ for } u \in \K^*, $  respectively. Setting $\g_k:= \g(\sigma, \psi),$ we note (see \cite{neu}) that  \be{g:sum} \g_k = \g_l \text{ if } n | k-l, \, \g_0 = -1, \text{ and if } k \neq 0 \mod n, \text{ then } \g_k \g_{-k} = q, \ee where for the last equality we must again assume that $q \equiv 1 \mod 2n$.
\newline

\noindent \textbf{Generic ring $\Zvg$ and its quantum, and $p$-adic specializations}

\tpoint{The generic ring $\Zvg$} \label{notation:Cvg} Fix a positive integer $n$, and let $\gv$ and $t$ be a formal parameters related by $\tau^2=t^{-1}.$ Introduce a system of formal parameters $\gf_k$  for $k \in \zee$ satisfying the conditions \be{}\label{formal-gauss-sums} \gf_k = \gf_{\ell} \text{ if } k \equiv \ell \mod n, \, \, \gf_k \gf_{-k} = \gvs = t^{-1} \text{ for } k \neq 0, \text{ and  }\gf_0=-1. \ee A central role in this work will be played by the rings \be{} \label{Cvg} \begin{array}{lcr} \Zvg:= \zee[\gv^{\pm 1}, \{ \gf_k \}_{k \in \zee} ] / \sim & \text{ and } & \Cvg:= \C \otimes _{\zee} \Zvg \end{array} \ee obtained by formally adjoining $\tau, \tau^{-1}$ and the $\gf_k$ subject to the above relations. We note that \be{} \label{Cvg:2} \begin{array}{lcr} \Zvg \simeq \zee[\gv^{\pm 1},  \gf_1^{\pm1}, \ldots, \gf_{n-1}^{\pm1} ]  & \text{ and } & \Cvg \simeq \C[\gv^{\pm 1},  \gf_1^{\pm1}, \ldots, \gf_{n-1}^{\pm1} ]. \end{array}  \ee

\tpoint{The quantum specialization $\mf{q}$} \label{notation:quantum-spec} Let $\tau, t$ be indeterminates satisfying the condition $\tau^2 = t^{-1}$ and introduce the ring $\At:= \zee[\tau, \tau^{-1}]$ of Laurent polynomials in $\tau$. It carries an involution sending $\tau \mapsto \tau^{-1}$. The map $\mf{q}: \Zvg \rr \At$ which sends  
\be{} \label{quantum-specialization} \tau \mapsto \tau, \quad \gf_0 \mapsto -1, \quad \gf_i \mapsto \tau, \quad \text{ for } i \in \{1, \ldots, n-1\} \ee   
will be called the \emph{quantum specialization}.

\tpoint{The $p$-adic specialization $\mf{p}$} \label{notation:padic-spec}Suppose we are working over a $p$-adic field (i.e. non-archimedean local field) $\K$ of residue cardinality $q$ (a power of a prime) which is also equipped with a non-degenerate additive character $\psi$ (see \S\ref{subsub:additive-characters} for more details). For each positive integer $n$, one can then define Gauss sums $\g_k \in \C$ for each $k \in \zee$ as in \S\ref{subsub:GaussSums}. These elements satisfy the same properties as (\ref{formal-gauss-sums}) and so we can consider what we call the \emph{$p$-adic specialization} $\mf{p}: \Cvg \rr \C$ defined by sending   \be{} \label{eq:padic-specialization} t=\tau^{-2} \mapsto q^{-1}, \, \gf_i \mapsto \g_i \text{ for } i \in \zee. \ee

\section*{Frequently Used Notation} 

\label{sec:freq:used:notation} 
{\scriptsize 
\begin{center}
	\begin{tabular}{l|l|l} \label{frequentnotation}
		Notation & Meaning & Location  \\ \hline
		$\tau$, $t$ & parameters for Hecke algebra & \S\ref{subsub:CoxeterHecke} and \S\ref{subsub:renormalized:Coxeter:Hecke} \\
		$\gf_k$ and $\g_i$ & formal and $p$-adic Gauss sum parameters & \S\ref{notation:Cvg} and \S\ref{subsub:GaussSums} \\
		$\tV$ and $\tVs$ & quantum polynomial and spherical representations & \S\ref{subsub:V-quant} and \S\ref{subsub:Vsp-eta} \\ 
		$\V$ and $\Vsp$ & generic polynomial and spherical representations & \S\ref{subsub:Vstructure} and \S\ref{subsub:Vsp-g} \\
		$\tnalc$ & upper closure of twisted negative alcove & \S \ref{subsub:dotaction} \\ 
		$\taffH$ & twisted affine Hecke algebra & start of \S \ref{sub:twisted-affine-combinatorics} \\
		$\vec_{\lv}$ and $\vecket{\mv}$  & basis of $\tV$ and $\tVs$ & \S\ref{subsub:V-quant} and \S\ref{subsub:Vsp-eta} \\
		$\qlket{\lv}, \qlketm{\lv}$ & canonical basis of $\tVs$ & \S\ref{subsub:qVspherical} \\
		$Y_{\mv}$ and $\vv_{\mv}$ & basis of $\V$ & \S\ref{subsub:Vstructure} and \S\ref{subsub:vv-basis} \\
		$ \lvecket{\lv}, \lvecketm{\lv}$ & canonical basis of $\Vsp$ & \S\ref{subsub:Vsp-g} \\
		$\gket{\lv}, \gketm{\lv}$ & ($\gf$-twisted) canonical basis of $\Vsp$ & \S\ref{subsub:Vsp-g} \\
		$I^-$ and $K$ & Iwahori and maximal compact of $\tG$ & \S \ref{subsub:CompactIwahori}, \S  \ref{subsub:IwahoriCompactMet} \\
		$\hec(\tG, I^-)$, $\hec(\tG, K)$ & Iwahori and spherical Hecke algebras &  \S \ref{subsub:SpherticalHeckeAlgebra} and \S \ref{subsub:metIwahorHecke}  \\
		$\wt{\mathscr{Y}}_{\lv} $ & basis of the Whittaker space $\whitiw$ & \S\ref{subsub:met:iw:whit:basis} \\
		$\tJ_{\lv}$ & $p$-adic basis of $\whitk$ & \S\ref{subsub:action:thsph:on:twhitk} \\
		$\tL_{\lv}$ and $\til_{\lv}$ & canonical bases of $\whitk$ & \S\ref{subsub:action:thsph:on:twhitk} \\
		$V_{\etav}$ & irreducibles in $\Rep(\check{\bG}_{\ell}(\C))$ & \S\ref{subsub:K_0K^Rmodule} \\
		$\Delta_{\lv}$, $\nabla_{\lv}$, $L_{\lv}$ and $T_{\lv}$ & standards, costandards, simples and ind. tiltings in $\dU_\zeta(\check{\mf{D}})$ & \S\ref{subsub:K_0K^Rmodule}
	\end{tabular} 
\end{center} }

\part{Combinatorial models} \label{part:combinatorial_models}

\section{Affine Hecke algebras and canonical bases }  \label{sec:qKL}

\newcommand{\sW}{\mathscr{W}}

In this section, we collect various algebro-combinatorial preliminaries needed in this paper. 
In \S\ref{subsub:Hw-DL} we interpret certain straightening rules in terms of a quantum specialization of metaplectic Demazure--Lusztig operators and use it in~\S\ref{subsub:expressionLRcoeffCS} express Littlewood--Richardson polynomials in terms of the metaplectic Casselman--Shalika formula. These results and Proposition~\ref{prop:TensorProd} (2) seem to be new.

\subsection{Cartan datum and associated structures} 

\label{section:CartanDataumAndAssociatedStructures}

\tpoint{Cartan datum}

\label{subsub:CartanDat} \label{s:cd} Following \cite[\S1.1]{lus:qg}, we define a \emph{Cartan datum} to be a pair $(I, \cdot)$, where $I$ is a finite set and $\cdot$ is a $\zee$-valued symmetric bilinear form  on the $\zee$-module $\zee[I]$ satisfying
\be{car:dat} i \cdot i \in \{ 2, 4, 6, \ldots \} \text{ for } i \in I; \text{ and } \ \ 
2 \frac{ i \cdot j}{i \cdot i} \in \{ 0, -1, -2, \ldots \} \text{ for } i, j \in I, i \neq j. \ee 
Given a Cartan datum $(I, \cdot)$, one may construct a \emph{generalized Cartan matrix} (gcm)  as in~\cite[p.1]{kac}:   \be{} \label{CartanMatrixFromCartanDatum} \As= \As(I, \cdot) = (a_{ij})_{i,j\in I} \text{ where } a_{ij} = 2 \frac{i \cdot j}{i \cdot i}.\ee We assume throughout that the Cartan datum $(I, \cdot)$ is irreducible (\cf \cite[\S2.1.3]{lus:qg}) which ensures that $\As$ is irreducible in the sense of \cite[p.2]{kac}.  We deal in this paper exclusively with such Cartan datum  $(I, \cdot)$ of \emph{finite type}, i.e. the symmetric matrix $(i\cdot j )_{i, j \in I}$ is positive definite, or of \emph{affine type}, when the corresponding symmetric matrix has nullity one. 

\newcommand{\checkdot}{ \check{\cdot}}

\tpoint{Dual Cartan datum} \label{subsub:dualCartanDatum} For $(I, \cdot)$ a Cartan datum, define $m_I$ to be the smallest positive integer such that $ \frac{m_I }{2 (i \cdot i)}  \in \zee,$ and then define the dual Cartan datum $(I, \checkdot)$ ( see \cite[p.92]{weis}) by setting \be{} \label{dual:dot} i \, \checkdot \,  j := m_I \, \frac{ i \cdot j }{ (i \cdot i  )( j \cdot j)} \text{ for } i, j \in I. \ee 
Writing $\check{\As}:= \As(I, \checkdot)$ for the corresponding Cartan matrix, one finds that $\check{\As} = \leftidx^t \As.$

\tpoint{Minimal Cartan datum attached to a Cartan matrix} \label{subsub:minCartanDatum} Let $\As= \left( a_{ij} \right)$ be an  irreducible Cartan matrix of finite type and of size $k \times k.$ Write $I:= \{ 1, \ldots, k \}$ and let $D $ be a diagonal matrix such that $D \As = S $ is a symmetric matrix. 
Note that such a $D$ always exists;  it is unique up to rescaling the matrix by a constant, and so we can choose $D$ to have positive, integral values. 
Denote by $D_{\mathrm{min}}$ a symmetrizing matrix with the property that every other choice of $D$ is a positive, integral multiple of $D_{\mathrm{min}}$. We write \be{} \label{Dmin} D_{\mathrm{min}}= \mathrm{diag}(d_1, \ldots, d_{k}). \ee We explain in the next paragraph how to compute these values. We can use them to define the \emph{minimal Cartan datum $(I, \cdot)$ attached to a given Cartan matrix $\As$ } by setting $ i \cdot j := d_i a_{ij}.$

\tpoint{Untwisted affinizations and computing the symmetrization of $\As$}  \label{subsub:untwistedAffinization}  Denote the (untwisted) affinization of $\As$ by $\As_{\aff}$; it is defined as in  \cite[p.100, (7.4.3)]{kac}. The matrix $\As_{\aff}$ is of size $(k+1) \times (k+1)$ and our convention is to identify the first last $k$-rows and columns of $\As_{\aff}$ with $\As$. 
As such, we also write in this case \be{}\label{I:aff} I_{\aff}:= I \sqcup \{ 0 \}, \text{ where } I = \{ 1, \ldots, k \}. \ee The matrix $\As_{\aff}$ has a kernel containing a unique vector with strictly positive, relatively prime entries $ \delta(\As_{\aff})= (m_0(\As), \ldots, m_{k}(\As)).$ 
For these untwisted affinizations, it turns out that  $m_{0}(\As)=1$ for all irreducible $\As$. 
The transpose $\leftidx^t \As$ is again a Cartan matrix of finite type and we denote by  $\check{m}_i(\As):= m_i(\leftidx^t \As).$ With these definitions, the matrix \be{}\label{D:affinization}D := \mathrm{diag} \left( \check{m}_0(\As)/m_{1}(\As), \ldots,  \check{m}_{k}(\As)/m_{k}(\As) \right) \ee symmetrizes $\As$, i.e. $D A$ is symmetric. However, $D$ only has positive and rational entries in general. One can construct $D_{\mathrm{min}}$ by multiplying this matrix by an appropriate integral multiple.

\tpoint{Braid and Weyl groups}\label{subsub:BraidGroups} Given a Cartan datum $(I, \cdot)$ with associated Cartan matrix $\As$ we define integers $h_{i j}$  for $i, j \in I$ according the rules as in \cite[\S2.1.1]{lus:qg}, i.e. $\cos^2 \, \frac{\pi}{h_{ij}} = \frac{a_{ij} \, a_{ji}}{4}$. The \emph{braid group} $\mathscr{B}(I, \cdot)$, is the free group on symbols $s_i \, (i \in I)$ equipped with relations \be{bd:rel}  \underbrace{s_i s_j s_i \cdots }_{h_{ij}} = \underbrace{s_j s_i s_j \cdots }_{h_{ij}} \text{ for } i \neq j,  \ee where both sides have $h_{ij} < \infty$ terms. If we further impose the relation $s_i^2=1$ for all $i \in I$ we obtain the \emph{Weyl group} $\sW(I, \cdot)$. As both $\mathscr{B}(I, \cdot)$ and $\sW(I, \cdot)$ only depend on the associated gcm $\As$, we often just write these as $\mathscr{B}(\As)$ or $\sW(\As)$. In the case that $\As_{\aff}$ is the untwisted affinitization of $\As$, we often write \be{}\label{WeylGroupConventions} \affW:= \sW(\As_{\aff}) \text{ and } W:= \sW(\As). \ee

\tpoint{Coxeter groups and Bruhat order }\label{subsub:coxeter}  \label{subsub:bruhatorder} The pair $(\sW(\As), S)$ where $S=\{ s_i\}_{ i \in I }$ described in the previous paragraph forms a Coxeter system and $\sW:= \sW(\As)$ is called a Coxeter group (see \cite[Ch. IV, \S1.3, Def. 3, p. 4]{bour456}). Note that every element $s \in S$ satisfies $s^2=1$ (i.e. $S^{-1}=S$) so words in $S$ are just products of elements from $S.$ We refer to \cite[Ch. IV, \S1]{bour456} for the definitions of reduced expressions, the length function $\ell: \sW \rr \zee$, etc.  
Let us write $\leq_S,$ or just $\leq$ if the set $S$ is not in question, for the Bruhat order on $\sW(\As)$ induced from the Coxeter structure (see \cite[Ch. 2]{brenti}). 
As usual, we write $x < y$ to mean that $x \leq y$ and $x \neq y.$ If $\As$ is of finite type, there exists a unique element in $\sW(\As)$ which is maximal for the Bruhat order, which we write as $w_0.$ 

\tpoint{Parabolic subgroups}  For  $J \subset I$, if we define $\sW_J \subset \sW$ to be the subgroup generated by $s_j , j \in J,$ then from \cite[Ch. IV, \S8, Thm. 2]{bour456}, one knows that $(\sW_J, \{ s_j \}_{j \in J})$ is itself  Coxeter group. These are called \emph{parabolic} subgroups.   For example, if $\As_{\aff}$ is the affinization of $\As$ and $\sW(\As_{\aff}):= \sW(\As_{\aff})$, then we have $(\sW_{\aff})_{I}= \sW(\As)= W$ (where $I \subset I_{\aff}$ is an in~\ref{I:aff}). We write $w_0^J$ for the unique maximal length element of $\sW_J$, whenever it exists. 
For example, $w_0^J$ exists for any $J \subsetneq I_{\aff}$.

\newcommand{\poin}{\mc{P}}

\tpoint{Poincar\'e polynomials} \label{subsub:poincare} For $\mathtt{q}$ a formal variable, $J \subset I$ such that the parabolic subgroup $\sW_J$ defined above is finite, we define the Poincare polynomial \be{} \label{poincare} \poin_{\sW_J}(\mathtt{q})= \poin_J(\mathtt{q}) = \sum_{w \in \sW_J} \mathtt{q}^{\ell(w)}. \ee In the case $J = I$, we sometimes just write $\poin_{\sW}(\mathtt{q})$ or even $\poin(\mathtt{q})$ for the corresponding expression.  

\tpoint{Cosets} \label{subsub:parabolicSubgroups}  

For any subset $J \subset I$,  define the set $\sW^J:= \{ w \in \sW \mid \ell(w s_j) > \ell(w) \text{ for all  } j \in J \}, $ and note (see \cite[Prop 2.4.4]{brenti}) that every $w \in W$ has a unique factorization $ \label{kostant:factorization} w = w_1 w_2 \text{ where } w_1 \in \sW^J, w_2 \in \sW_J, $
and in this factorization $ \ell ( w ) = \ell(w_1) + \ell(w_2). $ It follows from \cite[Lemmas 3.1-3.2]{deodhar:77} that \be{} \label{soergel-obs} \mbox{if } s \in S, \sigma \in \sW^J, \mbox{ but } sw \notin \sW^J, \mbox{ then } s\sigma = \sigma s_j \mbox{ for some } j \in J.\ee 

\noindent Similarly,  consider the set  $  \leftidx^J\sW:= \{ w \in \sW \mid \ell(s_j w)  > \ell(w) \text{ for } j \in J \}, $ which consists of minimal length representatives of the cosets $\sW_J \setminus \sW$.  Finally, suppose now $J, K \subset I$ are given. Then we can consider the set of double cosets $\sW_K \setminus \sW / \sW_J$ and it is known that the set \be{} \leftidx^K \sW ^J:= \{ w \in \sW \mid s_k w > w, w s_j > w \text{ for all } k \in K, j \in K \} \ee forms a set of minimal length representatives for these double cosets. We say that such a double coset is \emph{regular} if its stabilizer (or equivalently the stabilizer of any element in the double coset) under the action of $\sW_K \times \sW_J$ (acting on left and right, respectively) is trivial. Equivalently, an element $t \in \leftidx^K \sW ^J$ is regular if $\sW_K t \cap t \sW_J = \emptyset$. The set of such regular double cosets will be denoted by $\left( \sW_K \setminus \sW / \sW_J \right)_{\reg} $ and the set of minimal length regular representatives is denoted  $ \left(\leftidx^K \sW ^J\right)_{\reg}$.

\subsection{Root datum and their twists}

In this section all Cartan datum $(I, \cdot)$ will be of finite type. 

\tpoint{Root datum} \label{subsub:RootDatum} By a \emph{root datum} of type $(I, \cdot)$, we shall mean a quadruple $\mf{D}= (Y, \{ y_i \}_{i \in I}, X, \{ x_i \}_{i \in I})$ where $Y, X$ are free abelian groups of finite rank that are in duality under a pairing that we denote $\la \cdot, \cdot \ra: Y \times X \rr \zee$, and where $  \{ y_i \}_{i \in I} \subset Y$ and $ \{ x_i \}_{i \in I} \subset X$ satisfy \be{}  \la y_i, x_j \ra = a_{ij} \text{ for } i, j \in I.\ee Here, as in \S\ref{subsub:CartanDat}, $\As = (a_{ij})$ is the associated gcm to $(I, \cdot)$. 
The root datum is said to be of \emph{simply connected} type if $Y$ is the free $\zee$-module with basis $\{ y_j \}_{j \in I},$ and it is said to be of \textit{adjoint type} if $X$ is the free $\zee$-module with basis $\{ x_i \}.$ 

 If $\mf{D}= (Y, \{ y_i \}_{i \in I}, X, \{ x_i \}_{i \in I})$ and $\mf{D}'= (Y', \{ y'_i \}_{i \in I}, X', \{ x'_i \}_{i \in I})$ are two root datum for $(I, \cdot)$, we say that $\mf{D}$ and $\mf{D}'$ are \emph{isomorphic root datum} and write $\mf{D} \cong \mf{D}'$ if there exist isomorphisms of abelian groups $Y \rr Y'$ sending $y_i \mapsto y'_i$ and $X \rr X'$ sending $x_i \mapsto x'_i$ compatible with the pairings $Y \times X \rr \zee$ and $Y' \times X' \rr \zee$.

\tpoint{Dual root datum} \label{subsub:LanglandsRootDatum}
Recall the dual Cartan datum $(I, \checkdot)$ attached to $(I, \cdot)$ from \S\ref{subsub:dualCartanDatum}. If $\mf{D} = (Y, \{ y_i \}_{i \in I}, X, \{ x_i \}_{i \in I})$ is a root datum of Cartan type $(I, \cdot)$, we define its (Langlands) dual $(I, \checkdot, \check{\mf{D}})$ by setting $\check{\mf{D}}:= (X, \{ x_i \}_{i \in I}, Y, \{ y_i \}_{i \in I}).$ 

\tpoint{Weyl group action} \label{subsub:WeylGroupActionRootDatum} Given a root datum $\mf{D}$, let $\s_i: Y \rr Y$ be the map defined by $\s_i(y) = y - \la y, x_i \ra y_i$  for  $\mv \in Y. $ We can similarly define an operation $\s_i:  X \rr X.$ We note that the $\s_i$ so defined satisfy the Coxeter presentation for the Weyl group. Hence $W(I, \cdot)$ acts on $Y$ by sending $s_i \in W(I, \cdot)$ to $\sigma_i$ for $i \in I$.

\tpoint{Roots systems attached to $\mf{D}$.}\label{subsub:rootsystemmetD} Fix $(I, \cdot, \mf{D})$ as above. 
We refer to the set $\Piv:= \{ y_i \}_{i \in I}$ as the set of simple coroots, and $\Rv:= W \Piv$ as the set of coroots of the root datum. 
Every $\av \in \Rv$ can be written as a unique linear combination of the elements from $\Piv$ with all positive or all negative coefficients. 
In this way, we can define $\Rv_{\pm}$ as the set of positive and negative coroots. 
We also define the \emph{coroot}, \emph{root}, and \emph{weight lattices} as 
\be{} \label{coroot-lattice}  \begin{array}{lcr}
\rtlv:= \bigoplus_{i \in I } \zee y_i, \ \  \rtl:= \bigoplus_{i \in I} \zee x_i, & \text{ and } & 
\label{eq:weight-lattice}\Lambda := \{ x \in X\otimes_{\zee} \Q \mid \la y_i, x \ra \in \zee \text{ for } i \in I \},  \end{array}
\ee
respectively. We define the positive root and coroot lattices by replacing $\zee$ with $\zee_{\geq0}$ in the above, and also write $\Lambda_+ \subset \Lambda$ for the \emph{dominant weights}, defined by the additional condition that $\la y_i, x \ra \geq 0$ for all $i \in I.$ A \emph{regular dominant weight} is one such that $\la y_i, x \ra > 0$ for all $i \in I$. These are denoted as $\Lambda_{+, \reg}$. Also, we define the dominance order $\leq$ on $\Lambda$ to be the relation \be{}\label{dominance-order} \mu \leq \lambda \text{ if } \lambda - \mu \in \rtl_+ \text{ for } \lambda, \mu \in \Lambda. \ee 

In a dual manner, replacing the $\{ x_i \}$ with $\{ y_i\}$ in the above we can also define the simple roots $\Pi$, all roots $\rts$, positive/negative roots $\rts_{\pm},$ coweights $\Lambdav$, dominant coweights $\Lambdav_+,$ and dominance order (again denoted $\leq$) on $\Lambdav.$ For each $a \in \rts$, if $x_i \in \Pi$ is such that $w  x_i = a$ then the element $\av:= w y_i \in \Rv$ is well-defined and independent of the choice of such $y$. The element $\av$ is called the \emph{coroot} attached to $a.$  
Let us also note here that since $(I, \cdot)$ is of finite-type, there is a unique root $\theta \in \rts$ such that when we write 
\be{} \label{def:highestroot} \theta:= \sum_{i \in I} m_i x_i \text{ with } m_i \in \zee_{\geq 0}, \ee 
the $m_i$ are maximized.
It is known that all $m_i>0$ and in fact $m_i = m_i(\As)$ are the numbers introduced in \S\ref{subsub:untwistedAffinization}. 
In a similar way we can define a highest coroot $\check{\theta}$ written as in (\ref{def:highestroot}) with $m_i$ replaced by $\check{m}_i(\As)$ and $x_i$ replaced by $y_i$. 
We shall also introduce the elements $\rho, \rhov$ in $\Lambda$ and $\Lv$ by the following definitions:
\be{} \label{rho:roots} 2 \rho = \sum_{a \in \rts} \, a &\mbox{ and }& 2 \rhov = \sum_{\av \in \Rv} \, \av. \ee 
Note that this implies that $ \la y_i, \rho \ra =1.  $
A similar formula hold for $\rhov$.

\tpoint{$(\Qs, n)$-twists} \label{subsub:MetStructure} Fix $(I, \cdot, \mf{D})$ be a root datum with associated Cartan matrix $\As= (a_{ij})$ as in \S\ref{subsub:CartanDat}. Writing $\mf{D}:=(Y, \{ y_i \} , X,  \{x_i\})$, a \emph{$(\Qs, n)$-twist}\footnote{These are called metaplectic structures in \cite{weis}, where they were first introduced} on $\mf{D}$ is the data of an integer $n \geq 1$ and a $W= W(I, \cdot)$-invariant (for the action defined in \S\ref{subsub:WeylGroupActionRootDatum}) quadratic form $\Qs$ on $Y.$  Attached to $\Qs$, we define the symmetric, bilinear form $\Bs: Y \times Y \rr \zee, \,  \Bs(y_1, y_2) := \Qs(y_1+y_2) - \Qs(y_1) - \Qs(y_2)$  for  $y_1, y_2 \in Y.$ One verifies   
\be{} \begin{array}{lcr} \label{Q:pair}
\Bs(y_i, y) = \Qs(y_i)\la y, x_i \ra \mbox{ for } y \in Y, i \in I & \text{ and hence } & 
\Qs(y_j)/\Qs(y_i)= a_{ji}/a_{ij}. 
\end{array} \ee 
From \cite[Proposition 3.10]{weis:metaplectiphobia}, there exists a unique $W$-invariant, quadratic form $\Qs$ on $\rtlv$ (thought of as a subset of $Y$) which takes the value $1$ on all short coroots. We shall call such a structure \emph{primitive} and note that every $\zee$-valued $W$-invariant form on $\rtlv$ is an integer multiple of $\Qs.$ Said another way, simply-connected root datum admit primitive twists.

\newcommand{\qsn}{(\Qs, n)}

\tpoint{$(\Qs, n)$-twisted root datum} \label{subsub:TwistedRootDatum} \label{subsub:TwistedCartanDatum} Starting from a root datum $(I, \cdot, \mf{D})$ and $(\Qs, n)$-twist, we may twist the Cartan datum $(I, \cdot)$ to obtain a new Cartan datum $(I, \circ_{(\Qs, n)})$ (or just $(I, \circ)$ if the twist is implicitly understood) as follows (\cf\cite[Construction 1.3]{weis} and \cite[\S2.2.4]{lus:qg}) 
\be{new-mt-cd:1} i \; \circ \;  j := \frac{n^2}{n(y_i) \, n(y_j) } i \cdot j, \ee where for $i \in I$ we define  $n(y_i) $ as the smallest positive integer satisfying \be{} \label{ni:def}  n(y_i) \, \Qs(y_i) \equiv 0 \mod n. \ee  The associated Cartan matrix is now $\wt{\As}:= \As_{(\Qs, n)} = \left(\frac{ n(y_i) }{n(y_j)} a_{ij} \right)_{i,j\in I}$ from which we may deduce that \be{} \label{twisted-equals-untwisted} \begin{array}{lcr} B(I, \circ) \simeq B(I, \cdot) & \text{ and } & W(I, \circ) \simeq W(I, \cdot) \end{array} .\ee 
If we now set \begin{itemize} 
\item $Y_{\qsn}:= \t{Y}:= \{ y \in Y \mid \Bs(y, y') \in n \zee \text{ for all } y' \in Y \}$,
\item $y_{\qsn, i}:= \t{y}_i:= n(y_i) y_i $ for $i \in I$,
\item $X_{\qsn}:= \t{X}:= \{ x \in X \otimes \mathbb{Q} \mid \la y, x \ra \in \zee \text{ for all } y \in \t{Y} \}$,
\item $x_{\qsn, i}:= \t{x}_i:= n(y_i)^{-1} x_i$ for $i \in I$,
\end{itemize} 
one can verify, as in \cite[\S2.2.5]{lus:qg}, \cite[\S11]{mcnamara:ps}, or \cite[Construction 1.3]{weis}, that  \be{}\label{twisted-root-datum} \mf{D}_{(\Qs, n)}:= \widetilde{\mf{D}}= (\t{Y}, \{\t{y}_i\}_{i\in I},\t{X}, \{\t{x}_i\}_{i\in I}) \ee is a root datum for $(I, \circ)= \left(I, \circ_{(\Qs, n)}\right)$. We shall usually adopt the convention that tildes will designate the corresponding notion for twisted root systems, but ocassionally when we need to keep track of the precise twist under consideration, we revert to the more precise notations $\mf{D}_{\qsn}, Y_{\qsn}$, etc.  So, for example $\t{\rts}$ (resp. $\t{\rts}^{\vee}$) will denote the set of roots (resp. coroots), etc. The quantities $\t{\rho}, \t{\rho}^{\vee}$ are defined as in \eqref{rho:roots} using the twisted root systems. 

\tpoint{A rank one example} Consider the Cartan datum $I=\{ a \}$ with $a \cdot a =2.$ We may associate to it the root datum $\mf{D}= (Y, \av, X, a)$, with $Y := \zee \av$, $X := \frac{1}{2}\zee a$ and $\la \av, a \ra=2$. This is the root datum which specifies, over $\C$, the simply connected group $\mathbf{PGL}_2(\C)$. The Langlands dual $
\check{\mf{D}}= \left( \frac{1}{2}\zee a, a, \zee \av, \av \right)$, which corresponds to the group $\mathbf{SL}_2(\C)$, is not of simply-connected type. The primitive $(\Qs,n)$ twist of $\mf{D}$ is determined by specifying $\Qs (\av) = 1$. Then $n(\av) = n$ and the associated bilinear form $\Bs$ on $Y$ is specified by $\Bs(k \av, k' \av) = 2 \, k \, k'$ for $ k, k' \in \zee.$ One may then check that 
\be{}
\mf{D}_{(\Qs,n)} := 
\begin{cases} \left(n\zee \av, n\av, \frac{1}{2n} \zee a, \frac{1}{n}a\right)  \simeq \mf{D} \, \mbox{ if $n$ is odd} , \\
\left(\frac{n}{2} \zee \av, n \av, \frac{1}{n}\zee a, \frac{1}{n} a \right)  \simeq \check{\mf{D}}  \mbox{ if $n$ is even}. 
\end{cases}
\ee
\begin{nrem} In general, the primitive twist of a finite Cartan datum will be isomorphic to itself (if $n$ is odd), or to its Langlands dual (if $n$ is even). This is not always true for affine type (cf~\cite[Table 2.3.2]{PPDuke}).  \end{nrem}

\tpoint{Example: $l$-twisted Cartan and root datum}\label{subsub:exampleltwisted} Fix a positive integer $l$ and a root datum $(I, \cdot, \mf{D})$ with $\mf{D}= (Y, \{ y_i \}, X, \{ x_i \})$. In \cite[2.2.4-2.2.5]{lus:qg}, Lusztig has defined the notion of an $l$-twited root datum. To introduce it, let $l_i$ be the smallest positive integers such that $l_i \frac{i \cdot i}{2} \in l \zee$. Then one defines $(I, \circ_l)$ as the Cartan datum with $i \circ_l j = l_i l_j (i \cdot j)$ and attaches to it a root datum $\mf{D}_l = (Y_l, \{ y_{l, i}\}, X_l , \{ x_{l, i}\})$ defined as follows: 
\be{} \label{def:Xl}  \begin{array}{lccr} X_l := \{ \zeta \in X \mid \la y_i, \zeta \ra \in l_i \zee \}, & Y_l := \Hom(X_l, \zee), & x_{l, i} := l_i x_i, & \text{ and } y_{l, i} := l_i^{-1} y_i. \end{array} \ee 
This root datum $(I, \circ_l, \mf{D}_l)$ can be subsumed into the theory of twists from \S\ref{subsub:TwistedRootDatum} as follows. 

\newcommand{\Qd}{\check{\Qs}}

We start with $\As= (a_{ij})$ an irreducible Cartan matrix, and fix the minimal Cartan datum $(I, \cdot)$ as in \S\ref{subsub:minCartanDatum} so that $i \cdot j = d_i a_{ij}$ with the positive integers $d_i , i \in I$ as in \S\ref{subsub:untwistedAffinization}. Hence 
\be{}\begin{array}{lcr} \frac{i \cdot i}{ 2 } = \frac{ d_i a_{ii} }{2} = d_i, & \text{ and so defining } &
 \label{eq:l_i} l_i := l / (l,d_i) \text{ for } i \in I, \end{array} \ee 
we find that $l_i$ are the smallest positive integers such that $l_i \frac{i \cdot i}{2} = l_i d_i  \in l \zee$. 
Fix $\mf{D}= (Y, \{ y_i \}, X, \{ x_i\})$ a root datum of adjoint type, so that $\check{\mf{D}}$ is of simply-connected type and hence admits a primitive twist $\Qd: X \times X \rr \zee $ in the sense of \S\ref{subsub:MetStructure}, i.e. $\Qd(x_j)=1$ for all short roots. Note that from \eqref{Q:pair} applied to the Cartan datum $(I, \checkdot)$ with Cartan matrix $\check{\As}=\leftidx^t \As$, we have
\be{}
\frac{\Qd(x_j)}{\Qd(x_i)} = \frac{a_{ij}}{a_{ji}} = \frac{i \cdot j}{d_i} \frac{d_j}{j \cdot i} =  \frac{d_j}{d_i}.
\ee 
If $x_i$ is any long root attached to a short root $x_j$, then $\Qd(x_j)= d_j / d_i = d_j $, since $d_i=1$ for $x_i$ short. In other words, if $\Qd$ is the primitive twist on $\check{\mf{D}},$ we have \be{} \begin{array}{lcr} \label{Q:d} \Qd(x_i) = d_i \text{ for all } i \in I & \text{ and } & n(x_i) = l_i  \end{array} \ee by using (\ref{ni:def}) which states that $n(x_i)$ is the smallest positive integer such that  $\Qd(x_i) \equiv 0 \mod n.$ 

\begin{nclaim} \label{claim:LusVsQtwist} Let $(I, \cdot)$ be the minimal Cartan datum attached to a Cartan matrix $\As$ and $(I, \cdot, \mf{D})$ a root datum of adjoint type. Writing $\Qd$ is the primitive twist on $\check{\mf{D}}$, we have an isomorphism of root datum \be{} \label{ltwist:Ql-twist} 
(I, \circ_l ,\mf{D}_l ) \simeq
\left(I, (  (\checkdot)_{(\Qd, l)})^{\vee},  \left((\check{\mf{D}})_{(\Qd, l)} \right)^{\vee} \right). \ee  \end{nclaim}

\noindent We apply this claim to $\check{\mf{D}}$, which is assumed to be of adjoint type, in Part~\ref{part:qg}. In this case, writing $\check{\mf{D}}_l$ for the corresponding $l$-twist and $\Qs$ for the primitive twist on the simply-connected $\mf{D}$, we have 	
\be{} \label{ltwist:Ql-twist:duals} 
	(I, \circ_l ,\check{\mf{D}}_l ) \simeq
	(I, (  (\checkdot)_{(\Qs, l)})^{\vee},  \left(\mf{D}_{(\Qs, l)} \right)^{\vee}).
\ee

\begin{nrem}\label{rem:nonprimitivetwist} For $k$ a positive integer, let $(k\Qd,l)$ be a multiple of the primitive twist $(\Qd,l)$ as in the Claim. Then $\left((\check{\mf{D}})_{(k \Qd, l)} \right)^{\vee}$ is isomorphic to $\mf{D}_{l'}$, where $l' = \frac{l}{\gcd(l,k)}$. \end{nrem}

\tpoint{A $\mathbf{GL}_r$-example}

\newcommand{\eev}{\check{e}}

Let us now give an example of a root system that does not have a unique primitive twist. Define the Cartan datum $(I, \cdot, \mf{D})$, where $I := \{1, \ldots, r-1 \}$, 
\[ i \cdot j = \begin{cases} 2 &\mbox{ if } i =j, \\ -1 &\mbox{ if } i \neq j, \end{cases}\]
and $\mf{D} = (Y, \{\av_i\}, X, \{\av_i\})$ is the root datum constructed as follows: $Y$ is the free $\zee$-module with basis $\{\eev_i,  1 \leq i \leq r\}$, $\av_{i} := \eev_i-\eev_{i+1}$, $X$ is the free $\zee$-module with basis $\{e_i, 1 \leq i \leq r\}$, $a_i:=e_i -e_{i+1},$ and the pairing between $Y$ and $X$ is $\la \eev_i, e_j \ra := \delta_{ij}$. This is the root datum for the group $\mathbf{GL}_r.$

A $(\Qs, n)$ twist on $\mf{D}$ consists of a positive integer $n$ and a quadratic form $\Qs$ on $Y$. Such a twist is actually determined by two integers $\mathbf{p}:= \frac{1}{2}\Bs(\eev_1,\eev_1) = \Qs(\eev_1)$ and $\mathbf{q}:= \Bs(\eev_1, \eev_2)$ (see~\cite[\S4.1]{gao:weissman:depth:zero}). With this notation, we have that $\Qs(\av_i)=2 \mathbf{p}-\mathbf{q}$. As such, there is no \emph{primitive} twist, and actually several classes of such covers (which are not multiples of one another) appear in the literature. For example, if $2 \mathbf{p}-\mathbf{q}=-1$, one obtains the \emph{Kazhdan-Patterson} (see~\cite{kazhdan:patterson}) covers used in the literature on automorphic forms. Note that this cover is the pull back of the opposite (i.e. $\Qs(\av_i)=-1$) of the primitive cover of $\mathbf{SL}_r$. The covers when  $2\mathbf{p}-\mathbf{q}=-2$ features in the work of Savin (see~\cite[\S4.1]{gao:weissman:depth:zero}). Finally, the coverings studied in the works ~\cite{bbb, bbbf, bbbg:metahori} linking Whittaker functions with quantum $\emph{affine}$ groups and lattice models satisfy $\mathbf{p}=k$ and $\mathbf{q}=0$.

It would be interesting to understand if Lusztig's definition of a quantum group at a root of unity based on $l$-twists (introduced in~\S\ref{subsub:exampleltwisted}) can be generalized to all $(\Qs, n)$ twists. In particular, if one can build a quantum group at a root of unity corresponding to the Kazhdan-Patterson twist and the Savin twist introduced above. 
We note here that for the $\mathbf{p}=k$ and $\mathbf{q}=0$ twist, one may naturally build the $l$-twisted root datum $\tY:= \frac{1}{l} Y, \t{X} := l X$ and $\tav_i=\frac{1}{l}\av_i, \t{a}_i=l a_i$ and its associated quantum $\mathfrak{gl}_r$ at a root of unity.

\subsection{Hecke algebras and (parabolic) Kazhdan--Lusztig theory of Coxeter groups  }\label{sub:HeckealgebrasandKLtheory}

Thoughout this section, we work over the ring $\At= \zee[\tau, \tau^{-1}]$ from \S\ref{notation:quantum-spec}. Recall that $\tau^2=t^{-1}$ and define  
\be{} \label{Zpm} \begin{array}{lcr} \At^+:= \tau \zee[\tau] & \text {and} & \At^-:= \tau^{-1}\zee[\tau^{-1} ]. \end{array} \ee 

\tpoint{Hecke algebra of Coxeter group} \label{subsub:CoxeterHecke} Let $(\sW, S)$ a Coxeter group (so, we allow both finite and affine examples). Then we define its Hecke algebra $H_{\sW}:= H(\sW, t)$ as the algebra over $\At$ with linear basis $\{ T_w \}_{w \in \sW}$ and multiplication determined by the rules:
\begin{itemize} \label{HeckeRelations}
\item{\textbf{Quadratic:}}  For $s \in S$, $T_s^2 = t^{-1}  + (t^{-1}-1) T_s$, i.e. $(T_s -t^{-1}) (T_s + 1)=0$; and 
\item{\textbf{Braid:}} there exists a unique homomorphism of groups $\mathscr{B}(\As) \rr H(\sW, t)$ sending $s_i \mapsto T_{s_i}$. 
\end{itemize} 
The braid relations ensure that we may unambiguously define, for any $w \in W$, the element 
\be{} \label{Tw:def} T_w = T_{s_{i_1}} \cdots T_{s_{i_r}} \in H_\sW \text{ for any reduced decomposition } w= s_{i_1} \cdots s_{i_r}. \ee 
The quadratic relation ensures that $T_{s_i}$ is invertible in $H_\sW$, and we denote its inverse as $T_{s_i}^{-1}$,which is given explicitly as $ T_{s_i}^{-1} = t^{-1} T_{s_i} + (t^{-1} -1)$. Using the braid relations, we can define $T_w^{-1}$ for any $w \in \sW.$ 

\tpoint{Renormalized basis}\label{subsub:renormalized:Coxeter:Hecke} We often work with the following renormalization of the generators. Let \be{H:soergel}  H_w: = \tau^{-\ell(w)} T_w  \text{ for } w \in W, \ee 
which again satisfy the braid relations together with  the new quadratic relation 
\be{}\label{H:quad} (H_{s_i} - \tau) (H_{s_i} + \tau^{-1}) =0, \mbox{ i.e. } H_{s_i}^2 + ( \tau - \tau^{-1}) H_{s_i} -1 = 0.\ee 
Again all $H_w$ are units, and one checks that 
\be{}\label{H:inv} H_{s_i}^{-1} = H_{s_i} + (\tau^{-1} - \tau)   \text{ for } i \in I. \ee 

\tpoint{Kazhdan--Lusztig theory} There exists exactly one ring homomorphism 
\be{}\label{d} d: H(\sW, t) \rr H(\sW, t) \text{ such that }  d\left(\tau\right)= \tau^{-1}, \,  \text{ and  } d\left( H_w \right)=\left( H_{w^{-1}} \right)^{-1}.\ee 
The map $d$ is an involution, called the \emph{Kazhdan--Lusztig involution}, and its application to $f \in H(\sW, t)$ is often just written as $f \mapsto \overline{f }$. 
An element $f \in H(\sW, t)$ is called \emph{self-dual} if $d(f) =  f.$ For future use, we record the following facts: \begin{itemize}
	\item For $w \in W$, there exist $r_{w', w} \in \At$ so that (with respect to the Bruhat order)
	\be{} \label{H:inv-triangle} \overline{H}_w = H_w + \sum_{w' < w} r_{w', w} H_{w'}. \ee  
	\item Suppose $\sW$ is finite with longest element $w_0$. Then for any $\sigma \in \sW$, we have (\cf \cite[Prop. 3.2.2(2)]{brenti}) $\ell(\sigma w_0)+\ell(\sigma^{-1})=\ell(w_0)$, so that $H_{\sigma^{-1}} H_{\sigma w_0} = H_{w_0}$ and hence \be{} \label{inv:Hw0} \overline{H}_{\sigma} =  H_{\sigma w_0 } H_{w_0}^{-1}. \ee 

\end{itemize}

\newcommand{\hm}{h^-}
\newcommand{\Hm}{H^-}
\newcommand{\uHm}{\u{H}^-}
\newcommand{\uMm}{\u{M}^-}
\newcommand{\uNm}{\u{N}^-}
\newcommand{\mm}{m^-}
\renewcommand{\nm}{n^-}
\newcommand{\om}{o^-}
\newcommand{\Om}{O^-}
\newcommand{\uOm}{\underline{O}^-}

\noindent Using \cite[Lemma 24.2.1]{lus:qg} and the relation \eqref{H:inv-triangle}, we may conclude (\cf \cite{KL-Cox} or \cite[Theorem 2.1 and Claim 2.3]{soergel:combinatoric} ) that for each $w \in \sW$, there exists a unique self-dual element $\u{H}_w \in H(\sW,t)$ such that
\be{} \label{KL-triangularity} \u{H}_w = H_w + \sum_{y < w } h_{y,w} H_y \text{ where } h_{y,w} \in \At^+ .\ee We could have worked in $\At^-$ and showed there exists a unique self-dual element $\uHm_w \in H(\sW,t)$ such that
\be{} \label{KL-triangularity:minus} \uHm_w = H_w + \sum_{y < w } \hm_{y,w} H_y \text{ where } \hm_{y,w} \in \At^- .\ee 

\noindent The elements $\{ \u{H}_w \}_{w \in W}$ (resp. $\{ \uHm_w \}_{w \in W}$ ) form the \textit{Kazhdan-Lusztig} $\At$-basis of $H(\sW, t).$ 

\begin{nrem} Typically,  the polynomials 
\be{}\label{KL-poly} P_{y, x}= \tau^{\ell(x)- \ell(y)} \hm_{y, x} \in \zee[\tau, \tau^{-1}] \ee 
are called \emph{Kazhdan--Lusztig polynomials} and satisfy (\cite[\S2]{soergel:combinatoric}) $P_{y, x} \in \zee[\tau^{-2}]$ with constant coefficient $1$. 
\end{nrem}

\tpoint{Symmetrizers and anti-symmetrizers} Fix the same notation as in the previous paragraph and let $J \subset S$ be such that $\sW_J$, the parabolic subgroup defined in \S\ref{subsub:parabolicSubgroups},  is finite so that $w_0^J$, the longest element in $\sW_J$, is well-defined. Define the elements \be{} \begin{array}{lcr}
\label{symmetrizer} \epsilon_J^+ := \tau^{-\ell(w_0^J)} \sum_{w \in \sW_J} \tau^{\ell(w)} H_w & \text{and } &  \label{anti-symmetrizer} \epsilon_J^- := (-\tau)^{\ell(w_0^J)} \sum_{w \in \sW_J} (-\tau^{-1})^{\ell(w)} H_w. \end{array} \ee

\noindent We may compute, \cf \cite[(5.5.17)(i)]{mac:aha}, for $w \in \sW_J$,  
\be{} \begin{array}{lcr} \label{act:H-pos} H_w \epsilon_J^+ = \epsilon_J^+ H_w = \tau^{\ell(w)} \epsilon_J^+ & \text{ and } &  \label{act:H-neg} H_w \epsilon_J^- = \epsilon_J^- H_w = (-\tau)^{-\ell(w)} \epsilon_J^+, \end{array} \ee 
and also verify that (\cf\cite[(5.5.17)(ii))]{mac:aha}) in terms of the Poincar\'e polynomial defined in \eqref{poincare} we have
\be{} \label{sym:square} (\epsilon_J^+)^2 = \tau^{- \ell(w_0^J)} \poin_{\sW_J}(t^{-1}) \epsilon_J^+ \text{ and } (\epsilon_J^-)^2 = (-\tau)^{ \ell(w_0^J)} \poin_{\sW_J}(-t) \epsilon_J^- . \ee  
Finally, it is known (see \cite[Proposition 2.9]{soergel:combinatoric}) that $\epsilon_J^+ = \u{H}_{w_0^J}$ and that $\epsilon_J^-$ is the `signed Kazhdan--Lusztig' basis element denoted as $C_{w_0^J}$ in \cite{KL-Cox}, so that in particular both $\epsilon_J^{\pm}$ are self-dual, i.e.  $d \left ( \epsilon_J^{\pm} \right)= \epsilon_J^{\pm}. $

\tpoint{The parabolic modules $M_J$ and $\leftidx_JN$} \label{subsub:parabolicmodules} Write $H_{\sW}:= H(\sW, t)$. Fix the same notation as in the previous paragraphs, so that $J$ is chosen with $\sW_J$ finite. As in \S\ref{subsub:parabolicSubgroups}, we write $\sW^J$ for a set of minimal length representatives of $W/ W_J$. Consider the $A$-modules 
\be{} \label{pos:parabolic-mod} \begin{array}{lcr} M_J := H_{\sW}  \epsilon_J^+:= \{ h \epsilon_J^+ \mid h \in H_{\sW}  \} & \text{ and } & 
\leftidx_JN := \epsilon_J^- \, H_{\sW}  := \{ \epsilon_J^- \cdot h  \mid h \in H_{\sW} \}.\end{array}
\ee

\noindent From (\ref{act:H-pos}), we see that $M_J$ and $\leftidx_JN$ have $A$-bases $M_w:= H_w \epsilon_J^+$ and $N_w:=\epsilon_J^-  H_w $ where $w$ ranges over  all elements in $\sW^J$ and $\leftidx^J \sW$ respectively. As we shall need it later, we record here the following elementary computation: for $\sigma \in \leftidx^J\sW$ and $s \in S$ \be{}\label{Hs:Nx} N_{\sigma} \, H_s  = \begin{cases} 
N_{\sigma \, s} & \mbox{ if }  \sigma \, s \in \leftidx^J\sW, \sigma s > \sigma, \\ 
N_{\sigma s} + (\tau - \tau^{-1}) N_{\sigma} & \mbox{ if }  \sigma s \in \leftidx^J \, \sW,  \sigma s < \sigma, \\
-\tau^{-1} N_{\sigma} & \mbox{ if }  \sigma s \notin \leftidx^J \sW. 
\end{cases} \ee The first and second cases follow from the braid relations, and the third by also using~\eqref{soergel-obs} and~\eqref{act:H-pos}.

\tpoint{Parabolic Kazhdan--Lusztig polynomials}\label{subsub:parabolicKLpols} 	
The involution $d$ induces one on $M_J$ and $\leftidx_JN$, respectively, so that one can again speak of self-dual elements in these $A$-modules. From Deodhar \cite{deodhar:87} (\cf \cite[Theorem 3.1]{soergel:combinatoric}) for $w \in \sW^J, z \in \leftidx^J \sW$ , there exist unique self-dual elements $\u{M}_w, \uMm_w$ as well as  $\u{N}_z, \, \uNm_z$ such that 
\be{} \begin{array}{lccr} \label{uM:w} 
\u{M}_w := M_w + \sum_{y \leq w}  m_{y, w} M_y, & 
\label{uN:w} \u{N}_z := N_z + \sum_{y \leq z}  n_{y,z} N_y, & \text{ where } &  m_{y, w}, n_{y, z} \in  \At^+, \end{array} \\
\begin{array}{lccr} \label{uM:w:minus} 
	\uMm_w := M_w + \sum_{y \leq w}  \mm_{y, w} M_y, & 
	\label{uN:w:minus} \uNm_z := N_z + \sum_{y \leq z}  \nm_{y,z} N_y, & \text{ where } &  \mm_{y, w}, \nm_{y, z} \in  \At^-, \end{array}
\ee

\noindent These elements satisfy the following properties (cf.~\cite{deodhar:87}, \cite[Prop. 3.4]{soergel:combinatoric}): for $y, w \in \sW^J$,
\be{} \begin{array}{lcr} \label{m:h:minus} \mm_{y, w} = \hm_{ y w_0^J,  w w_0^J} & \text{ and } & \nm_{y, w} = \sum_{z \in \sW_J} (- \tau^{-1})^{\ell(z)} \hm_{zy, w}, \end{array} \\   
\begin{array}{lcr} \label{m:h} m_{y, w} = (-1)^{\ell(y) + \ell(w)}  d(\nm_{y, x}) & \text{ and } & n_{y, w} = (-1)^{\ell(y) + \ell(w)}  d(\mm_{y, x}) , \end{array}
 \ee  where the involution $d$ on $\At$ just interchanges $\tau$ and $\tau^{-1}$. For the last line, see \cite[Thm. 3.5]{soergel:combinatoric}.

\begin{nrem}\label{rem:defmnvsSoergel}
The elements $\tau^{\ell(w)- \ell(y)} m_{y, w}$ and $\tau^{\ell(w) - \ell(y)} n_{y, w}$ are examples of parabolic Kazhdan--Lusztig polynomials \cite{deodhar:87}. Our definitions of $m_{y,w}$ and $n_{y,w}$ match similarly named quantities in ~\cite{soergel:combinatoric} after setting $v = \tau^{-1}$ in \emph{op. cit.} 
\end{nrem}

\newcommand{\ojk}{\leftidx_JO_K}
\newcommand{\regWjk}{\left(\leftidx_J\sW_K\right)_{\reg}}

\tpoint{The module $\leftidx_J O_K$}\label{subsub:JOK} Suppose $K \subset I$ is such that $W_K$ is finite. Then we consider the vector space 
$ \leftidx_J O_K:= \epsilon_J^- H_{\sW} \, \epsilon_K$
whose basis consists of $O_t:= \epsilon_J^- \, H_t \epsilon_K$ for $t \in \left(\leftidx_J\sW_K\right)_{\reg}.$ One can view it as a submodule $\ojk \hookrightarrow \leftidx_JN$  from which we may define  $d: \ojk \rr \ojk$ by restricting the involution $d$ on $\leftidx_JN$.

\begin{nprop} For all $u \in \left(\leftidx_J\sW_K\right)_{\reg}$, there exists a unique self-dual element $\uOm_u$ such that 
\be{eq:otuKLpol}   \uOm_u = O_u + \sum_{t < u }  \om_{t, u} O_t, \mbox { with } \om_{t, u} \in \At^-. \ee For this element, we have \be{}  
\om_{t, u} =  \nm_{t w_0^K, u w_0^K}= \sum_{z \in \sW_J} (- \tau^{-1})^{\ell(z)} \mm_{zt, u}
\label{eq:oxytohxy:min}= \sum_{z \in \sW_J} (- \tau^{-1})^{\ell(z)} \hm_{zt w_0^K, uw_0^K}.\ee We also have unique self-dual elements $ \u{O}_u \in \ojk$ satisfying
\be{eq:def:uot} 
 \u{O}_u &=& O_u + \sum_{t < u }  o_{t, u} O_t, \mbox { with } o_{t, u} \in \At^+. \ee For such elements, we have 
 \be{}  
o_{t, u} =  m_{w_0^J t, w_0^J u}= \sum_{z \in \sW_K} \tau^{\ell(z)} n_{tz, u}
\label{eq:oxytohxy:pos}= \sum_{z \in \sW_K} \tau^{\ell(z)} h_{w_0^J t z , w_0^Ju}, 
\ee

\end{nprop}
  
\begin{proof}
The existence of the unique self-dual elements $\underline{O}^{\pm}_{u}$ satisfying~\eqref{eq:otuKLpol} follows from \cite[Lemma 24.2.1]{lus:qg} using the fact that the involution $d$ on $\ojk$ inherits the required triangularity property from that of the involution on $\leftidx_JN$ (or $H_{\sW}$.) Let us just show ~\eqref{eq:oxytohxy:min}. From  \cite[Proposition 2.9]{soergel:combinatoric},  \be{} \epsilon_K^+=\underline{H}^-_{w_0^K} = \sum_{w \in \sW_K} \tau^{\ell(w)-\ell(w_0^K)} H_w. \ee 
Hence, for $u \in \left(\leftidx_J\sW_K\right)_{\reg},$  under the  map $\zeta: \ojk = \leftidx_JN \epsilon_K^+ \hookrightarrow \leftidx_JN$ we have 
\be{}
\zeta(O^-_u) =  \epsilon_J^- H_u \underline{H}^-_{w_0^K} = \epsilon_J^- \sum_{w \in \sW_K} \tau^{\ell(w)-\ell(w_0^K)} H_{tw}. 
\ee Applying the involution to both sides of the above, using the fact that $\zeta$ is compatible with taking such involutions, and finally using the characterisation of  $\underline{N}^-_z$ for $z \in \leftidx^J W$ reviewed in \S\ref{subsub:parabolicKLpols} we can verify that \be{O:eJK}  \zeta(\underline{O}^-_u) = \epsilon_J^{-} \underline{H}^-_{u w_0^K}. \ee
Expanding the left hand side and using the fact that $H_z \epsilon_K = \tau^{\ell(z)} \epsilon_K$ for all $z \in \sW_K$, we may write
$$ \underline{O}^-_{u} = \sum_{x\in \leftidx^J\sW^K} o^-_{x,u} O_x = \sum_t o^-_{x,u} \epsilon^-_J H_x \epsilon_K =  \sum_{x \in \leftidx^J\sW^K, z \in W_K} o^-_{x,u} \epsilon^-_J  H_{xz} \tau^{\ell(z)-\ell(w_0^K)}= \sum_{ x \in \leftidx^J\sW^K , \, z \in \sW_K} o^-_{x,u} N_{xz} \tau^{\ell(z)-\ell(w_0^K)}. $$ As for the right hand side of \eqref{O:eJK}, we may expand, now using $\epsilon_J^- H_w = (-\tau^{-1})^{\ell(w)}$ for $w \in \sW_J$

\be{}  \epsilon_J^- \underline{H}_{uw^K} &=& \epsilon_J^- \sum_{x \in \sW} h^-_{x,uw_0^K} H_x = \epsilon_J^- \sum_{y \in \sW_J , \, z \in \leftidx^J\sW} h^-_{yz,uw_0^K} H_{yz} 
= \epsilon_J^- \sum_{y \in \sW_J , \, z \in \leftidx^J\sW} h^-_{yz,uw_0^K} (-\tau)^{-\ell(y)}  H_{z} \\ &=&  \sum_{y \in \sW_J , z \in \leftidx^J \sW} h^-_{yz,uw_0^K} (-\tau)^{-\ell(y)}  N_{z} 
=\sum_{y \in \sW_J , xz \in \leftidx^J \sW} h^-_{yxz,uw_0^K} (-\tau)^{-\ell(y)}  N_{xz} \ee

\noindent Hence we obtain that 
\be{}\begin{array}{lcr} \sum_{y \in \sW_J } h^-_{yxz,uw^K} (-\tau)^{-\ell(y)} = o^-_{x,u} \tau^{\ell(z)-\ell(w_0^K)} & \text{and also }  \sum_{y \in \sW_J } h^-_{yxw_0^K,uw_0^K} (-\tau)^{-\ell(y)} = o^-_{x,u},  \end{array} \ee where we set $z=w_0^K$ in the first equation to get the second. The other equalities in \eqref{eq:oxytohxy:min} follow from \eqref{m:h:minus}. 

One may verify \eqref{eq:oxytohxy:pos} in a similar way.  We leave the proof to the reader. \end{proof}

\subsection{Affine Weyl groups}  \label{sub:AffineWeylGroups} 
In this section, $(I, \cdot)$ will be a Cartan datum of finite type with associated Cartan matrix $\As$ and untwisted affinization $\As_{\aff}$. Write $W := \sW(\As)$ for the finite Weyl group and $\affW := \sW(\As_{\aff})$ for the affine Weyl group.  Fix $(I, \cdot, \mf{D})$which is assumed to be of simply-connected type and write $\mf{D}:= (Y, \{ \av_i \}, X, \{ a_i \})$. Using this root datum, we can define the Euclidean space $V:= Y\otimes_{\zee} \R$. Extend the pairing $Y \times X \rr \zee$ to one between $V$ and $\check{V}:= X \otimes_{\zee} \R$ by linearity and denote it by the same symbol. We write $\rts \subset V^*$ to be the roots as in \S\ref{subsub:rootsystemmetD} , $\theta \in \rts$ the highest positive root as in \eqref{def:highestroot}, $\rtl$ for the root lattice, and $\rtlv \cong Y$ for the coroot lattice.

\newcommand{\hyp}{\mathfrak{H}}
\tpoint{Groups generated by affine reflections}  Define for each $i \in I$ and integer $k \in \zee$, the affine reflection $\sigma_{a_i, k}: V \rr V$ by $y \, \cdot \, \sigma_{a_i, k} = y - \left( \la y, a_i \ra - k \right) \av_i $ for $v \in V.$ Note that the reflection $\sigma_{a_i, k}$ fixes the hyperplane $\hyp_{a_i, k}:= \{ v \in V \mid \la v, a_i \ra = k \}.$  Replacing $a_i$ with $a$ for any root $a \in \rtl$ (the root lattice of $(I, \cdot, \mf{D})$ in the formulas above, we define $\sigma_{a, k}$. We define $\hyp_{a, k}$ as the fixed hyperplane of this reflection, and say that $y \in V$ is \emph{positive (resp. negative)} for the hyperplane $\hyp_{a, k}$ if $\la y, a_i \ra > k$ (resp. $\la y, a_i \ra < k $). Define now  \be{affWeyl:D} \affW(I, \cdot, \mf{D})= \la  \sigma_{a_i, k} \mid i \in I, k \in \zee \ra.\ee It contains as a subgroup $W= \sW(I, \cdot) \cong  \la \sigma_{a_i, 0} \mid i \in I \ra $, as well as the subgroup $\mathtt{T}$ generated by the elements \be{} \label{def:translations}  \tt(k \av_i):= \sigma_{a_i, k} \circ \sigma_{a_i, 0} \text{ for } i \in I, k \in \zee. \ee One may check that $y \cdot \tt( k \av_i) = y - k \av_i$ for $y \in V$ and hence $\tt(k \av_i) \circ \tt(m \av_j) =  \tt(m \av_j) \circ \tt(k \av_i).$ We denote this last element as $\tt(k \av_i + m \av_j) \in \mathtt{T}$. In a similar way we may define for any $\av \in \Rv$ a corresponding element $\tt(\av) \in \mathtt{T}$. One verifies $\mathtt{T} \cong \rtlv \cong Y,$ where the last isomorphism  follows from the simply connected hypothesis. Moreover, the subgroup $W:= W(I, \cdot)$ normalizes $\mathtt{T}$ and \be{} \label{affine:semi-direct} \affW(I, \cdot, \mf{D}) \cong W \ltimes Y. \ee 

\newcommand{\Saff}{S_{\mathrm{aff}}}

\tpoint{Coxeter structure} \label{subsub:AffineWeylGroupsCoxeter} As we are in the simply-connected case, the group $\affW(I, \cdot, \mf{D})$ is a Coxeter group with Coxeter structure defined by the set of simple reflections $\Saff$ consisting of $s_i:= \sigma_{a_i, 0}$  for $i \in I$, and the new reflection $s_{0} :=  \sigma_{\theta, 1}.$ In this simply-connected case, we then have $\affW(I, \cdot, \mf{D}) \simeq \sW(\As_{\aff})$ as Coxeter groups; we shall from now on just write
\be{} \affW:= \affW(I, \cdot, \mf{D}) \cong \sW(\As_{\aff}). \ee 
One can then define the Bruhat order on $\affW(I, \cdot, \mf{D})$ as well as the length function $\ell$ with respect to this Coxeter structure, and   from  (\cite[p.20, Prop. 1.23]{iwa:mat}), one has the following useful formula for the length  of an element in $\affW(I, \cdot, \mf{D})$ with respect to the previous decomposition (\ref{affine:semi-direct}): for $w \tt(\betav)$ with $w \in W, \betav \in \rtlv$, 
\be{} \label{length-IwahoriMatsumoto}  \ell \left(w \tt(\betav) \right) = \sum_{a \in \rts_+ \cap w^{-1} \rts_+} | \la \betav, a \ra| + \sum_{a \in \rts_+ \cap w^{-1} (- \rts_+) } |1 +  \la \betav, a \ra|. \ee  From this formula, \eqref{rho:roots}, and the fact that $\ell(w)$ is equal to the cardinality of the set $\rts_+ \cap w^{-1} (-\rts_+)$, (see \cite[Cor. 2, p.170]{bour456}), we see that if $\la \betav, a \ra \geq 0$ for all $a \in \rts$, then \be{}\label{length:pos-beta} \ell( w \tt(\betav)) =  | \, \rts_+ \cap w^{-1} (-\rts_+) | + \sum_{a \in \rts} \la \betav, a \ra = \ell(w) +   \la 2 \rho, \betav \ra. \ee 

\renewcommand{\B}{\mathscr{B}}

\tpoint{Chambers}\label{subsub:chambers}  The Euclidean space $V:= Y \otimes_{\zee} \R$, regarded as an affine space has two chamber structures, in the sense of \cite[Ch. V, \S1]{bour456} corresponding to the action of $W$ and $\affW(I, \cdot, \mf{D})$ on $V$.  The \emph{chambers} are the connected components of $V \setminus \{ L=\hyp_{a, 0}, a \in R\} $. The Weyl group $W$ acts transitively on the collection of chambers. We define the dominant chamber and antidominant chambers as
\be{fund:ch} \Ch_+:= \{ v \in V \mid \la v, a_i \ra >  0 \text{ for } i \in I \} \text{ and } \Ch_-:= \{ v \in V \mid \la v, a_i \ra <  0 \text{ for } i \in I \} \mbox{ respectively. } \ee 

\tpoint{Alcoves} \label{subsub:alcoves} Define the collection of \emph{affine hyperplanes} as $\wh{M}:= \{ \hyp_{a, k} , a \in \rts_+, k \in \zee \},$ and denote the connected components of $V \setminus \wh{M}$ as \emph{alcoves}. Write $\mathcal{A}$ for the collection of alcoves. For any $v \in V \setminus \wh{M}$ and $a \in \rts_+$, there exists an unique integer $k_a$ so that $k_a < \la v, a \ra < k_a+1,$ \textit{i.e.} $v$ lies in the connected region between the walls $\hyp_{a, k}$ and $\hyp_{a, k+1}$. For any alcove $A \in \mc{A}$, there exist integers $k_a, a \in \rts$ such that for $v \in A$, oen has  $k_a < \la v, a \ra < k_a+1$ for all $a \in \rts_+.$  The anti-dominant fundamental alcoves will be defined as 
\be{} \label{def:fund-alcove} 
\alc_- = \{ v \in V \mid -1 < \la v, a \ra < 0 \text{ for all } a \in \rts_+   \} \subset \Ch_-. \ee

\noindent The \emph{upper closure} of an alcove $A$, which will be denoted as $\overline{A},$ and it is obtained by replacing the upper bounds in the definition of an alcove with inequalities. Every $v \in V$ can be uniquely written as $v  = x \cdot w$ for $w \in \affW$ and $v \in \underline{\alc}_-$, where we record here that  
\be{} \label{def:fund-alcove-closures} 
\ov{\alc}_- := \{ v \in V \mid -1 < \la v, a \ra \leq 0 \text{ for all } a \in \rts_+   \} \subset \Ch_-. 
\ee

\tpoint{A result on separating walls}   \label{subsub:separating-walls} 
For  $w \in \affW$ define $\wh{M}(w)$ to be the set of walls which separate $\alc$ and $w\alc$ (a wall is called \emph{separating} if the alcoves lie on different half-spaces defined by the wall). In terms of a reduced decomposition of $w \in \taffW$ , say $w = w_{b_1} \cdots w_{b_d}$ where $b_j \in I_{\aff}= \{ 0  \} \sqcup I$, then we have : \be{}\label{separating:walls} \wh{M}(w)= \{ H_{b_1}, H_{b_2} w_{b_1}, \ldots,  H_{d} w_{b_{d-1}} \cdots w_{b_{1}}\}, \ee 
where $H_j:= H_{(a_j, 0)}$ if $j \in I$ and $H_{0}:= H_{(\theta, 1)}$ the walls defining simple reflections (see~\cite[Theorem 4.5]{humphreys:coxeter}).
\begin{nrem}In terms affine root systems (which we have not formally introduced), these are the walls corresponding to the positive affine roots which are flipped to negative by $w^{-1}$. \end{nrem}

\tpoint{Dot action} \label{subsub:dotaction} We are often interested in the following `dot' action of $\affW$ on $V$  \be{} \label{dot-action} x \da w := (x + \rhov ) \cdot w - \rhov \text{ for } w \in \affW, x \in V, \ee where $\rhov$ was defined in \eqref{rho:roots}. Note that it is indeed an \emph{action} of $W.$ 
Under the dot action, $\sigma_{a, k}$ is a reflection through the hyperplane $ \hyp^{\da}_{a, k}:= \{ x \in V \mid \la x + \rhov, a \ra = k \}$ since   \be{}\label{def:formula-dot-action} x \da \sigma_{a, k}  = x - \left( \la x + \rhov, a \ra - k \right) \check{a} \text{ for } x \in V,  a \in \rts, k \in \zee. \ee
One can again define a set of alcoves with respect to the hyperplanes $\{ \hyp^{\da}_{a, k} \mid a \in R, k \in \zee \}$ and we designate the (anti)-fundamental alcove and its upper closure as 
\be{} \label{def:fund-alcove-dot-action} \alc_-^{\da} &:=& \{ x \in V \mid -1 < \la v + \rhov, a \ra < 0  \text{ for all } a \in \rts_+   \}, \\ 
\ov{\alc}_-^{\da} &:=& \{ x \in V \mid  -1 <  \la x + \rhov, a \ra \leq 0  \text{ for all } a \in \rts_+   \}. \ee

\tpoint{Example: $(\Qs, n)$-twisted affine Weyl group} \label{subsub:twisted-alcoves} Let $(I, \cdot, \mf{D})$ be an arbitrary root datum, $(\Qs, n)$ a given twist with associated root datum $(I, \circ_{(\Qs, n)}, \wt{\mf{D}})$ which we assume is of simply connected type.  We note that $\sW(I, \cdot) \simeq \sW(I, \circ_{(\Qs, n)})$ and both groups are denoted as $W$. Recalling the Cartan matrix of $(I,  \circ_{(\Qs, n)})$ was denoted as $\wt{\As}$, we shall now write 
	\be{} \label{twistedWeylgroup} \taffW:= \affW(I, \circ_{(\Qs, n)}, \wt{\mf{D}})  \simeq \sW(\wt{\As}_{\aff}) \simeq W \ltimes \tY. \ee 
 The dot action of $\taffW$ on $V:= Y \otimes_{\zee} \R$ is defined via \be{} \label{dot-action-twisted} x \da \sigma_{\ta_i, k}  =v - \left(\la v + \rhov, \ta_i \ra -k  \right) \tav_i  \text{ for } v \in V,  i \in I, k \in \zee. \ee Note that the $\rhov$ which appears in the above formula is `untwisted'. A fundamental domain for this action of $\taffW$ on $V$ is given by the set 
\be{}\label{def:fund-alcove-dot-action-twisted}  \ov{ \alc}^{\da}_{-, (\Qs, n)}   &:=& \{ v \in V \mid -1  <  \la v + \rhov, \ta \ra \leq 0  \text{ for all } \ta \in \t{\rts}_+   \} \\ &=& \{ v \in V \mid -n(\av) <  \la v + \rhov, a \ra \leq 0  \text{ for all } a  \in \rts_+  \}. \ee  
When the choice of $\Qs$ is implicitly understood, we often abbreviate our notation and write  \be{} \label{abbrev} \tnalc:= \ov{ \alc}^{\da}_{-, (\Qs, n)} .\ee

\tpoint{Stabilizers, orbits, and cosets} \label{subsub:cosets} For $\etav \in \tnalc$, let $J \subsetneq I_{\aff}$ be such that \be{} \label{def:J} (\taffW)_J:= \stab_{(\taffW, \da)}(\etav). \ee 
Then we may identity $  (\taffW)_J  \setminus \taffW   $ with elements in the orbit $\taffW \da \etav.$ Writing $\muv  \in \taffW \da \etav$ as $\muv = \etav \da w $ with $w \in  \taffW$, we may fix $w$ uniquely by requiring it to lie in $\leftidx^J(\taffW)$.  In this way we get a bijective correspondence \be{}\label{orbits:cosets} \etav \da \taffW  \leftrightarrow  \leftidx^J \taffW  \text{ for } J \subsetneq I_{\aff} \mbox{ as in } \eqref{def:J}.  \ee 
We may extend this result to conclude that for $\etav \in \tnalc$ and $J$ as in \eqref{def:J}, there is a bijection \be{double-cosets:orbits} (\taffW)_J \setminus  \taffW  / W  \stackrel{1:1}{\longleftrightarrow} (W, \da)\text{-orbits in } \etav \da \taffW,\ee where by $(W, \da)$-orbits, we mean orbits under the $\da$-action of $W$.  The above correspondence can be further strengthened to one between regular double cosets (see the end of \S\ref{subsub:parabolicSubgroups}) and regular $(W, \da)$-orbits in $\etav \da \taffW$ (\textit{i.e.} those orbits under the $\da$-action of $W$ whose stabilizer is trivial). One may verify that each regular $(W, \da)$-orbit in $\etav \da \taffW $ has a unique dominant representative and so we obtain a bijection  \be{}  \label{regular-double-cosets-regular-dominant} Y_{+} \cap \etav \da \taffW  \, \stackrel{1:1}{\longleftrightarrow} \, \left( \leftidx^J \taffW^I\right)_{\reg}. \ee Denoting the Bruhat order on the left hand side by $\leq_B$ (the one induced from the Bruhat order on the right hand side). Then one knows that if $x, y \in   \left( \leftidx^J \taffW^I\right)_{\reg}$ correspond to dominant coweights $\lv_x, \lv_y$, then \be{} \label{comparing:orders}  x \leq_B y  \implies \lv_y \leq \lv_x \ee where the order $\leq$ in the right hand side is the dominance order.

\tpoint{Boxes and restricted weights}\label{subsub:boxes}Keep the notation of the previous section and define the lower-closure of the box (in the terminology of \cite{kato:aff}), or set of \emph{restricted weights}, as the set 
\be{} \label{def: box} \Box_{\qsn} := \{ \lv \in Y_+ \mid 0 \leq \la \, \lv, a_i \ra < n(\av_i) \text{ for } i \in I \} .\ee 
One can easily check that each $\lv \in Y_+$ may be uniquely written as 
\be{} \label{box-dec} \lv = \lv_0 + \zv, \lv_0 \in \Box_{\qsn}, \zv \in \tY_+. \ee 
Indeed, from $\lv \in Y_+$, we may just subtract positive multiples of the fundamental coweights attached to $\wt{\mf{D}}$ until we arrive at an element in the box (our simply-connected hypothesis ensures that these fundamental coweights lie in $\tY_+$).

\subsection{Affine Hecke algebras and their combinatorics}  \label{sub:aha-comb}

In this subsection, we assume $(I, \cdot, \mf{D})$ is of simply connected type and write $\mf{D} = (Y, \{ \av_i \}, X, \{ a_i \} ).$  Write $W$ for the Weyl group of $(I, \cdot)$ and $\affW:= \affW(I, \cdot, \mf{D})= W \ltimes Y$ for the corresponding affine Weyl group, which we know is a Coxeter group under the simply connected hypothesis. As such we can define two Hecke algebras: $H_W:= H(W, t)$ will be called the \emph{finite Hecke algebra} and $\affH:= H(\affW, t)$ will be called the \emph{affine Hecke algebra}. 

\newcommand{\dY}{\overline{Y}}

\tpoint{Bernstein presentation} \label{subsub:BernsteinaffineHecke}   If $\betav \in Y_+$ we set $Y_{\betav}:= T_{\tt(\betav)}$ and we note by the length formula \eqref{length-IwahoriMatsumoto} that $Y_{\betav + \gammav} = Y_{\betav} Y_{\gammav} = Y_{\gammav} Y_{\betav} \text{  when } \betav, \gammav \in Y_+.$ For a general element $\betav \in Y$, we may write $\betav = \lv - \mv$ with $\lv, \mv \in Y_+$ and then set $Y_{\betav} = Y_{\lv} Y_{\mv}^{-1}$ which can be shown to be well-defined (i.e. independent of the choice of $\lv, \mv \in Y_+$, cf. \cite[p.40]{mac:aha}).  This definition extends to give a well-defined injective morphism of algebras    $ \iota: \At[Y] \rr \affH. $ Let us also remark that one has the formula (see \cite[Lemma 3.5]{knop} for a conceptual derivation) \be{} \label{involution:Y} \dY_{\betav} = H_{w_0} Y_{w_0 \betav} H_{w_0}^{-1}. \ee One can verify (see~\cite[p. 59, (4.2.7)]{mac:aha}) that the elements $\{ H_w \, Y_{\betav} \}$ for $w \in W, \betav \in Y$ (and similarly $\{ Y_{\betav} \, H_w \}$) form an $\At$-basis of $\affH,$ where one has the following relation \be{BLR}  H_{s_i} Y_{\betav} - Y_{s_i \betav} H_{s_i} = (\tau^{-1}-\tau) \frac{Y_{s_i \betav} - Y_{\betav}}{1-Y_{-\av_i}}. \ee

\tpoint{The module $\affH \epsilon$} \label{subsub:right-sppherical-module}  Regarding $I \subset I_{\aff}$ in the natural way so that $(\taffW)_I = W,$ we then set  
\be{} \label{def:ep}  \epsilon:= \frac{\tau^{\ell(w_0)}}{\poin_W(t^{-1})}\epsilon_I^+.\ee 
From \eqref{act:H-pos} and \eqref{sym:square} we find linethat \be{}\label{epsilon:properties} H_{s_i} \epsilon = \tau \epsilon = \epsilon H_{s_i} \text{ for } i \in I \text{ and } \epsilon^2 =  \epsilon. \ee 
The parabolic module $\affH \epsilon= \affH \epsilon_I^+$ defined as in \S\ref{subsub:parabolicmodules} is sometimes also referred to as the \emph{polynomial representation} in the literature, as the relation \eqref{BLR} and the first of the relations above show that the map $\iota$ above induces an isomorphism of $\At$-modules,  
$ \iota: \At[Y] \stackrel{\cong}{\longrightarrow}  \affH \, \epsilon, $ see \cite[\S4.3]{mac:aha}. The induced action of $\affH$ on $\At[Y]$, denoted $h \mapsto h.p(Y), \, h \in \affH, p(Y) \in \At[Y]$, takes the following form:  \be{} \label{action:parabolic} \begin{array}{lcr} Y_{\lv} . Y_{\mv} = Y_{\lv + \mv} \text{ for } \lv, \mv \in Y, & \text{ and }
 & Y_{\mv} . H_{s_i} = \frac{\tau^{-1} - \tau Y_{-\av_i}}{1 - Y_{-\av_i}} \, Y_{s_i \mv} +  \frac{\tau - \tau^{-1}}{1 - Y_{-\av_i}} \, Y_{\muv}. \end{array}\ee

\tpoint{The spherical Hecke algebra $\spaff$}\label{subsub:sphericalHecke} The \emph{spherical subalgebra} is defined as the $\At$-module   \be{}\label{spherical-subalgebra} \spaff:= \epsilon\, \affH \, \epsilon, \ee with algebra structure induced from $\affH$ using the fact that $\epsilon$ is an idempotent. The inverse of the isomorphism $\iota$ from the previous paragraph induces an isomorphism of $\affH$-subalgebras \be{} \label{satake} S: \spaff  \stackrel{\cong}{\longrightarrow} \At[Y]^W, \ee where $\At[Y]^W$ is the space of $W$-invariant elements in the group algebra of $Y$ over $\At$ (the action of $W$ on $Y$ induces the action of $s_i$ on $\At[Y]$ by the relation $s_i Y_{\betav} = Y_{s_i \betav}$, etc.).
One can show that the elements $h_{\muv} := \epsilon \, Y_{\muv} \, \epsilon$ for $\muv \in Y_+$ form an algebra basis of $\spaff$ and that 
\be{} \label{HL-Satake} 
S(h_{\mv}) = \frac{\tau^{\la \muv, 2\rho \ra}}{\poin_{W_{\muv}}(t^{-1})} 
\sum_{w \in W} \left( \prod_{a \in \rts_+} \frac{1-\tau^{-2}Y_{-w\av}}{1-Y_{-w\av}} \cdot Y_{_{w\muv}}\right),
\ee
where $\poin_{W_{\muv}}(t^{-1})$ is the Poincar\'e polynomial of the stabilizer of $\muv$ in the finite Weyl group $W$ \cf\eqref{poincare}. The space $\At[Y]^W$ has another distinguished basis, namely the characters $\chi_{\lv}, \lv \in \Lv_+$ defined as:
\be{} \label{char:lv} 
\chi_{\lv} := \chi_{\lv}(Y) := \frac{ \sum_{ w \in W} (-1)^{\ell(w)} \, Y_{ w (\lv + \rhov) - \rhov} }{\prod_{\av \in \rts_+}  ( 1- Y_{ - \av}) }.
\ee Denoting by $\ckl_{\lv} \in \spaff$ the element such that $S(\ckl_{\lv})= \chi_{\lv}$, the elements $\{ \ckl_{\lv} \}_{\lv \in Y_+}$ form another basis of $\spaff$, which is related to the basis $\{ h_{\lv}\}$ through the so-called Kato-Lusztig formula 
\be{KL:afterSatake} \begin{array}{lcr} \ckl_{\lv} = \sum_{\mv \leq \lv} \, p_{\mv, \lv}(\tau) \, h_{\mv}, & \text{ or equivalently }& \label{char:HL}  \chi_{\lv}  = \sum_{\mv \leq \lv} \, p_{\mv, \lv}(\tau) \, S(h_{\mv}),  \end{array}\ee where $\leq$ in the above refers to the dominance order on $Y$ as in  \eqref{dominance-order}.
For $\zv \in Y_+$, if we write $w_{\zv}$ for the longest element in $W \tt(\zv) W$, then it can be deduced from \eqref{length-IwahoriMatsumoto} that \be{} \label{w:zv} w_{\zv}:= w_0 \tt(\zv) \text{ and } \ell(w_{\zv})= \ell(w_0)+ \la \zv, 2 \rho \ra.\ee 
In terms of the Kazhdan--Lusztig polynomials \eqref{KL-poly}, we have  
$ p_{\mv, \lv}(\tau) = \tau^{ \la \lv - \muv, 2 \rho \ra} P_{w_{\mv}, w_{\lv}}(\tau^2)$. It is known that $y \leq w_{\lv}$ in the Bruhat order if and only if there exists $\mv \in Y_+$ such that $\mv \leq \lv$ and $y \in W \tt(\mv) W$, see \cite[(4.6)]{kato:sph}. So it follows that if $\mv, \lv \in Y_+$, then $\mv \leq \lv$  if and only if $w_{\mv} \leq w_{\lv}$. 

\tpoint{The modules $\leftidx_JV$ and $\leftidx_JV_{\sp}$}  \label{subsub:VJsph}Let $J \subset I \subsetneq I_{\aff}$ and consider the $\At$-modules 
\be{} \label{V:M-mod} \begin{array}{lcr}  \leftidx_JV:= \epsilon^-_J \affH  & \text{ and } &\leftidx_JV_{\sp}:= \epsilon^-_J \affH \, \epsilon. \end{array}\ee
We have already described a basis for the space $\leftidx_JV$, namely  $N_w:= \epsilon_J^- H_w $ where $w$ ranges over the elements in $\leftidx^J \affW.$ In a similar way, we see that $\leftidx_JV_{\sp}$ has a basis $B_u:= \epsilon_J^- H_u \epsilon$ as $u$ ranges over a set of regular, minimal length coset representatives $\left( \leftidx^J\affW^I \right)_{\reg}$.  One can equip $\leftidx_JV$ and $\leftidx_JV_{\sp}$ with involutions inherited from the involution $d$ on $\affH$, and we continue to call these $d: \leftidx_JV \rr \leftidx_JV$ and $d: \leftidx_JV_{\sp} \rr \leftidx_JV_{\sp}$.  
Finally, let us note that $\leftidx_JV$ carries a right action by $\affH$, denoted as $\cdot$\,  and $\leftidx_JV_{\sp}$ is equipped with a right action by $\spaff$, denote as $\star,$ and computed as follows:  for $h \in \affH, u \in \left( \leftidx^J\affW^I \right)_{\reg}$, we have  
$ \left(\epsilon_J^- H_u \epsilon \right) \,\star \,  \left( \epsilon h \epsilon   \, \right)  = \epsilon_J^- \left(  H_u \, \epsilon \, h \right) \epsilon. $

\subsection{Twisted affine Hecke algebras }  \label{sub:twisted-affine-combinatorics}

Throughout this section, we let $(I, \cdot, \mf{D})$ be a root datum, written $\mf{D}= ( Y, \{ \av_i\}, X, \{ a_i \}),$ and equipped with a twist $(\Qs, n)$ such that the associated root datum $(I, \circ_n, \wt{\mf{D}})$, written $\wt{\mf{D}}= (\tY, \{ \tav_i\}, \tX, \{ \ta_i\}),$  is of simply-connected type.
Set 
\be{}\begin{array}{lccr} 
W:= \sW(I, \cdot)\cong \sW(I, \circ_{(\Qs, n)}), &
\taffW:= \affW(I, \circ_{(\Qs, n}) \cong W \ltimes \tY,  \text{ and} & 
\taffH:= \affH(\taffW, t)   \cong H_W \otimes \At[\tY].   \end{array}
\ee 
Finally, we denote the analogues of the modules in \eqref{V:M-mod} as $\leftidx_J\tV$ and $\leftidx_J\tVs$, i.e. \be{} \begin{array}{lcr} \leftidx_J\tV:= \epsilon_J^- \, \taffH & \text{ and } &   \leftidx_J\tVs:= \epsilon_J^- \, \taffH \, \epsilon. \end{array} \ee 

\tpoint{The module $\tV(\etav)$}\label{subsub:moduletVetav} Pick $\etav \in \tnalc$ and suppose $J \subset I_{\aff}$ is such that $\stab_{(\taffW, \da)}(\etav)= (\taffW)_J.$ Define the $A$-module $\tV(\etav)$ to be a copy of the $\taffH$-module $\leftidx_J\tV$. We now describe  a basis $\{ \vec_{\muv} \}$ of $\tV(\etav)$ indexed by elements $\muv$ in  $\etav \da \taffW $ (recall that, as we explained in \eqref{orbits:cosets}, this indexing set is in bijection with $\leftidx^J\taffW$, the set indexing the natural basis of $\leftidx_J\tV$). The elements $\vec_{\muv}$ are defined as follows:  \begin{itemize}
	\item   for $\muv  \in \etav \da \taffW$,  write $\muv = \etav \da w $ for a unique $w \in \leftidx^J\taffW;$
	\item  decompose $w =   \sigma \, \tt(\betav) $ with $\betav \in \tY$ and $\sigma \in W$;
	\item set  $ \vec_{\muv}:= \epsilon_J^- \,  H_{\sigma } \, Y_{\betav} . $ \end{itemize} The elements $\vec_{\muv}$ with $\muv \in \taffW \da \etav$ satisfy  
\be{} \label{Vmu:1}  \vec_{\muv} &=&  \epsilon^-_J H_{\sigma} \text{ if } \muv =  \etav \da \sigma , \sigma \in \leftidx^J \taffW, \\ 
\label{Vmu:2}  \vec_{\muv} \cdot Y_{\lv} &=& \vec_{\lv + \muv} \text{ if } \lv \in \tY, \muv \in \etav \da \taffW . \ee 

\begin{nlem} \label{lem:dom-vec}
	Let $\etav \in \tnalc$, $J$ defined as in \eqref{def:J}, and $\lv \in Y_+ \cap \etav \da \taffW.$ Then we may write $\lv = \etav \da \sigma \tt(\betav)$ with $\betav \in \tY_+$ and $\sigma \tt(\betav) \in \leftidx^J \taffW$, and hence $\vec_{\lv} = \epsilon_J^- H_{x }$ for the unique $x \in \leftidx^J \taffW$ such that $\etav \da x = \lv$. 

\end{nlem}
	 
\begin{proof}Observe that for each $\sigma \in W$, we have $-n(\av_i) < \la \etav \da \sigma, a_i \ra < n(\av_i)$. Hence writing $\lv = \etav \da \sigma \tt(\betav) = (\etav \da \sigma) + \betav$ for \emph{any} $\sigma \in W, \betav \in \tY$, we must have $\la \betav, a_i \ra \geq 0.$ Indeed, if $\check{\omega}_{\ta_i} =  n(\av_i) \check{\omega}_{a_i}$ denotes the fundamental coweight in the root system attached to $\wt{\mf{D}}$, we have (by the simply-connected hypothesis) that $\check{\omega}_{\ta_i} \in \tY$. The exist integers $f_i \in \zee$ such that 
	 \be{}  \betav = \sum_{i \in I} f_i \, \check{\omega}_{\ta_i}= \sum_i f_i \, n(\av_i) 
 \check{\omega}_{a_i},\ee from which it follows that $\la \betav, a_i \ra \in \{ 0, \pm n(\av_i), \pm 2 n(\av_i), \ldots, \}.$ The dominance of $\betav$ follows, and the second assertion of the Lemma follows from the \eqref{length:pos-beta} and the fact that if $\betav \in \tY_+$, then $H_{\tt{\betav}} = Y_{\betav}.$ 
 \end{proof}

\tpoint{Quantum Demazure--Lusztig operators}  \label{subsub:Hw-DL} The action of the polynomial part of the affine Hecke algebra $\taffH$, i.e. of the group algebra $\At[\tY]$ is given in the basis just introduced of $\tV(\etav)$ by formula \eqref{Vmu:2}. The  action of the finite Hecke algebra $H_W \subset \taffH$ on $\tV(\etav)$ can be described by the quantum specialization \eqref{quantum-specialization} of the \emph{metaplectic Demazure--Lusztig operators} that will be introduced in \S\ref{subsub:metDL}. Since we explain the properties of these (unspecialized) Demazure--Lusztig operators in more detail in \emph{loc. cit.}, we introduce them in a somewhat \textit{ ad hoc} manner now. To begin, we introduce the expression 
\be{} \upsilon(m) = \begin{cases} -1 & \text{ if } m \equiv 0 \mod n \\ \tau & \mbox{ if } m \not\equiv 0 \mod n \end{cases},   \ee 
then define the operators $\Tmwq_{s_i}$ on $\tV(\etav)$, acting on the right, via the formulas :
\be{} \label{rank1:quantumDL-mod}  \vec_{\lv} \cdot \left(\tau^{-1} \Tmwq_{s_i}\right) : = 	
\begin{cases} 
	\tau^{-1} \, \upsilon( \la \lv + \rhov, a_i \ra \Qs(\av_i) ) \, \vec_{ \lv \da s_i } + (\tau-\tau^{-1}) \sum \limits_{\substack{k \geq 0 \\ k n(\av) \leq \la \lv, a \ra }} \vec_{\lv - k n(\av_i) \av_i} & \text{ if } \la \lv, a_i \ra \geq 0, \\	
	\tau^{-1} \, \upsilon( \la \lv + \rhov, a_i \ra \Qs(\av_i) ) \, \vec_{  \lv \da s_i  } + (\tau^{-1}-\tau) \sum \limits_{\substack{k > 0 \\ k n(\av) < - \la \lv, a \ra  }} \vec_{\lv + k n(\av_i) \av_i} & \text{ if } \la \lv,  a_i \ra < 0.  
\end{cases} 
\ee 
Using these formulas, we find that for $\mv \in \etav \da \taffW $, 
\be{} \label{Tsi:vmv-cases} \vec_{\muv} \cdot \left( \tau^{-1} \Tmwq_{s_i} \right)= 
\begin{cases}  
	\vec_{\muv \da s_i  } & \mbox{ if } -n(\av_i)  <  \la \muv + \rhov, a_i \ra < 0,   \\
	\vec_{\muv \da s_i} +(\tau-\tau^{-1}) \vec_{\muv} & \mbox{ if }  0 < \la \muv + \rhov , a_i \ra < n(\av_i),    \\
	-\tau^{-1}  \vec_{\muv} = -\tau^{-1} \vec_{\muv \da s_i}  & \mbox{ if } \la \muv + \rhov, a_i \ra = 0, 
\end{cases} 
\ee where in the first two cases, we note that $ \la \muv + \rhov, a_i \ra \Qs(\av_i) \not\equiv 0 \mod n$ by the definition of $n(\av_i)$ in~\eqref{ni:def}. 

\begin{nprop} \label{prop:H-act-quantum} Let $\etav \in \tnalc.$ Then in $\tV(\etav)$, we have the relation  \be{}  \vec_{\muv} \cdot H_{s_i}  = \vec_{\mv} \cdot \left(\tau^{-1} \Tmwq_{s_i}\right) \text{ for } \muv \in \etav \da \taffW  \text{ and any } i \in I. \ee  	
\end{nprop}

\tpoint{Proof of Proposition \ref{prop:H-act-quantum}, part 1} Let $J:= \{ i \in I_{\aff} \mid \la \etav + \rhov, a_i \ra = 0 \}$, so that by definition $\tV(\etav) \cong \tV_J.$  Suppose that  $\muv \in \etav \da \taffW $ actually lies in the orbit of $W$, i.e. \be{}\label{case1:assumption} \muv =  \etav \da \sigma \text{ for }\sigma \in W . \ee  
Then \eqref{Hs:Nx} together with the fact that $\da$ defines an \emph{action} of $W$ implies for each $i \in I$
\be{} \label{Hsi:vmv-cases} \vec_{\muv} \cdot H_{s_i}  = \epsilon_J^- \cdot H_{\sigma} \cdot H_{s_i} = 
\begin{cases}  
\vec_{ \muv \da s_i } & \mbox{ if } \sigma s_i  > \sigma, \sigma s_i \in \leftidx^J\taffW, \\
\vec_{\mu \da s_i } +(\tau-\tau^{-1}) \vec_{\muv} & \mbox{ if }  \sigma s_i < \sigma ,  \sigma s_i \in \leftidx^J\taffW, \\
-\tau^{-1}  \vec_{\muv} & \mbox{ if }  \sigma s_i \not \in \leftidx^J\taffW.
\end{cases} 
\ee 
Hence, we need to match up the conditions in (\ref{Hsi:vmv-cases}) and (\ref{Tsi:vmv-cases}).

\begin{nlem}\label{lem:TsiHsi-match}  For $\etav \in \tnalc$ and $J \subset I_{\aff}$ such that \eqref{def:J} holds, we have 
\begin{enumerate}
\item If $\sigma s_i  \in \leftidx^J\taffW$ and $\sigma s_i > \sigma$ then $ -n(\av_i) < \la \muv + \rhov, a_i \ra < 0$.
\item If $\sigma s_i  \in \leftidx^J\taffW$ and $\sigma s_i < \sigma$ then $ 0 <  \la \muv+ \rhov, a_i \ra < n(\av_i)$.
\item If $ \sigma s_i \notin \leftidx^J\taffW$ if and only if $\la \muv + \rhov, a_i \ra = 0$.
\end{enumerate} 
\end{nlem}
\begin{proof} We begin with (3). If $ \sigma s_i \notin \leftidx^J\taffW$, then from \eqref{soergel-obs} we may write $\sigma s_i = s_j \sigma $ with $j \in J$.  
Hence \be{} ( \etav \da \sigma) \da s_i = \etav \da (\sigma s_i )   = \etav \da (s_j \sigma) = \etav \da s_j \da \sigma  = \etav \da \sigma \ee by the assumption on the stabilizer of $\etav$.  
But $( \etav \da \sigma) \da s_i =  \etav \da \sigma$ implies that $\la \etav \da \sigma  + \rho, a_i \ra = \la \muv + \rho, a_i \ra = 0.$  Conversely, if $\la \muv + \rhov , a_i \ra =0$, then $\muv \da s_i  = \muv.$ By definition of $\muv$,  $ ( \etav \da \sigma  ) \da s_i  = \etav \da \sigma $ and hence  \be{} \etav \da  ( \sigma s_i ) = \etav \da \sigma . \ee This last equation means that $\sigma s_i \sigma^{-1} \in (\taffW)_J$ or $\sigma s_i= \tau \sigma $ with $1 \neq \tau \in (\taffW)_J$. Therefore $s_i \sigma \notin \leftidx^J \taffW$ as $\sigma$ is the minimal length representative for right $ (\taffW)_J $ cosets. 

As for (1), first observe that  \be{}\la \muv + \rhov , a_i \ra =  \la \etav \da \sigma  + \rhov, a_i \ra = \la ( \etav + \rhov )\cdot \sigma ,a_i \ra = \la \etav + \rhov, \sigma^{-1} (a_i) \ra . \ee 
Using the definition $\ta_i$ from \S\ref{subsub:TwistedRootDatum}, what we need to show is \be{} -1 <  \la ( \etav+ \rhov)\cdot \sigma, \ta_i \ra < 0. \ee 
By assumption, $-1 < \la \etav + \rhov, \ta_k \ra \leq 0$ for all $k \in I$. As $\sigma s_i > \sigma$, $(\ta_i) \sigma^{-1} > 0$ and we find that $\la \muv + \rhov, a_i \ra \leq 0.$ 
However, we cannot have $\la \muv + \rhov, a_i \ra =  0$ since this would contradict part (1), so the desired upper bound follows. 
With our assumption that $\etav \in \talc$, the lower bound is just the statement that $\etav+\rhov$ and $(\etav+ \rhov)\sigma$ are on the same side of the hyperplane defining the reflection $\sigma_{\ta_i, 1}$. But this is clear from the fact that $\sigma \in W$ and the description of the set $\wh{M}(\sigma)$ from \eqref{separating:walls}.  Part (3) follows from a similar analysis.   \end{proof}

\tpoint{Proof of Proposition \ref{prop:H-act-quantum}, part 2} In the general case, suppose $\muv = \etav \da w$, $w \in \leftidx^J\taffW,$ where  $w=\sigma \tt(\betav)$ for $\sigma \in W$ and $\betav \in \tY.$ Defining  $\lv:= \etav \da \sigma$ it follows that $\sigma \in \leftidx^J\taffW$ as well, so that  
\be{} \vec_{\lv} := \epsilon_J H_{\sigma} \text{ and }   \vec_{\mv} =  \vec_{\lv} \cdot Y_{\betav} . \ee  
The Bernstein relation~\eqref{BLR} for the metaplectic Iwahori Hecke algebra shows that
\be{} \label{H:recursion} 
\vec_{\mv} \cdot H_{s_i}  = \vec_{\lv} \cdot  Y_{\betav} \cdot H_{s_i}  = \vec_{\lv} H_{s_i} Y_{\betav s_i}   - (\tau-\tau^{-1}) \, \vec_{\lv} \cdot \frac{Y_{\betav s_i} - Y_{\betav}}{1-Y_{- \ta^{\vee}_i}} .  
\ee  
We leave it as an exercise to verify that $( \vec_{\lv} \cdot Y_{\betav} ) (\tau^{-1} \Tmwq_{s_i})  = \vec_{\lv+ \betav}(\tau^{-1} \Tmwq_{s_i})$ satisfies the same recursion as $\vec_{\lv+\betav} H_{s_i}$ using the explicit formulas given in \eqref{rank1:quantumDL-mod}.

\begin{nrem} The operators $\Tmwq_{s_i}$ can also be built from the so-called Chinta--Gunnells action of the Weyl group $W$ (see \S\ref{cg-not} - \ref{s:metDL} for more details). Using the special property of the Chinta--Gunnells action \eqref{cgw}, the recursion asserted above for $\Tmwq_{s_i}$ then follows immediately from the formula \eqref{T:c-b}. \end{nrem}

\tpoint{The module $\tV$} \label{subsub:V-quant} Keep the same notation as in the previous paragraph, and consider now the sum 
\be{} \tV:= \bigoplus_{\etav \in \tnalc} \tV(\etav). \ee 
as an $\taffH$-module. Note that as a $\taffH$-module, many of the summands $\tV(\etav)$ will be isomorphic to the same $\tV_J$. We may write each $\lv \in Y$ as  $\lv = \etav \da w $ for a unique $\etav \in \tnalc$; defining $\vec_{\lv} \in \tV(\etav)$ as above, we see that $\{ \vec_{\lv} \}_{\lv \in Y}$ form an $A$-basis of $V$, i.e. we have as $A$-modules $\tV \cong \At[\tY].$ The action of $\taffH$ on $V$ and an involution $d$ are obtained by adding together those on individual $\tV(\etav).$

\tpoint{On the modules $\tVs(\etav)$}\label{subsub:V-quantu:spherical} For $\etav \in \tnalc$ and $J$ as in \eqref{def:J}, we define the $\At$-module \be{}\begin{array}{lcr} \label{Vsph:etav} \tVs(\etav):= \epsilon_J^-  \, \taffH \, \epsilon & \text{ and let } & \label{sph:Vmu} [\vec_{\muv}]:=  \vec_{\muv} \epsilon \text{ for }  \mv \in Y \cap \etav \da \taffW. \end{array} \ee The elements $[\vec_{\mv}]$ will span $\tVs(\etav)$ by a simple  application of the Bernstein presentation, but they will not be linearly independent, as the following relations hold. 
\begin{nprop} \label{prop:straightening:no-g}  Let $\etav \in \tnalc$ and let $a_i$ be a simple root with corresponding reflection $s_i$.  
\begin{enumerate}
		\item  If $\la \lv +  \rhov, a_i \ra \geq 0$ and $\la \lv + \rhov, a_i \ra \equiv 0 \mod n(\av_i)$ then 
		$ \label{q:straight:v:1} [\vec_{\lv}] + [\vec_{\lv \da s_i }] = 0. $
		\item If $0 <  \la \lv + \rhov , a_i \ra  < n(\av_i) $, then 
		$\label{q:straight:v:2} [\vec_{\lv}] - \tau [\vec_{\lv \da s_i }] =0. $
		\item  If $\la \lv + \rhov, a_i \ra > n(\av_i)$ but $j:= \la \lv + \rhov, a_i \ra \not\equiv 0 \mod n(\av_i) $,  then setting
		\be{} \begin{array}{lcr} \label{lv:1} \lv_{(1)} := \lv - \resi_{n(\av_i)} (\la \lv +\rhov, a _i \ra) \, \av_i, &\text{ we have}&  
		\label{q:straight:v:3}   0 = [\vec_{\lv}] - \tau [\vec_{\lv \da s_i }] - \tau [\vec_{\lv^{(1)}}] +  [\vec_{\lv^{(1)} \da s_i }]. \end{array}\ee 
\end{enumerate}
\end{nprop}   

\noindent In type $A$, the above is due to \cite[Prop. 5.9]{leclerc:thibon}. The generalization to arbitrary type is in \cite[Prop. 6.3(ii)]{haiman:grojnowski}, see also \cite[Prop. 4.4]{laniniramsobaje}. As it is a specialization of Proposition~\ref{subsub:straightening_rules_Vsph}, we defer the proof to Section~\ref{sec:gKL}.

\tpoint{The space $\tVs$} \label{subsub:Vsp-eta} As $\tVs(\etav)\cong \epsilon_J^- \taffH \epsilon$, we know that a basis consists of regular, minimal length representatives of the double cosets. On the other hand, we have a bijection, \cf\eqref{double-cosets:orbits}: 
\be{} Y_{+} \cap \etav \da \taffW \stackrel{1:1}{\longleftrightarrow} \left( \leftidx^J \taffW^I\right)_{\reg}, \ \ 
\mv = \etav \da w \mapsto (\taffW)_J \, w \,  W. \ee 
Using the straightening rules above, we may show that 

\begin{ncor} Fix $\etav \in \tnalc$. The collection $\{ [\vec_{\mv}] \}$ for $\mv \in Y_{+} \cap \etav \da \taffW$ forms a basis of $\tVs(\etav)$. \end{ncor} 

\noindent If we now define the $\At$-module \be{} \label{V:spherical} \tVs := \bigoplus_{\etav \in \tnalc} \, \tVs(\etav), \ee then we find that, as an $\At$-module, $\tVs$ has a basis $[\vec_{\lv}]$ with $\lv \in Y_+$.

\subsection{Involutions and canonical bases in $\tVs$} 
 Keep the same conventions and notations from \S\ref{sub:twisted-affine-combinatorics}.

\tpoint{Involution on $\tV(\etav)$} \label{subsub:Involution-tV-eta} Recall that each $\tV(\etav) \simeq \epsilon_J^- \taffH \epsilon$, and so is equipped with an involution $d$. It  takes the following explicit form with respect to the basis $\vec_{\mv}$ introduced above. 

\begin{nlem}\label{involution:LT-basis} \cite[Prop. 5.4]{leclerc:thibon} Let $\etav \in \tnalc$ with $J$ as in \eqref{def:J}. Then for $\mv \in \etav \da \taffW $, we have 
\be{}\label{d:on:V} \overline{\vec_{\muv}} := d(\vec_{\muv}) = (-\tau)^{ -  \ell(w_0^J)}  \vec_{ \muv \da w_0 } \cdot H_{w_0}^{-1}. \ee  \end{nlem}

\begin{proof} The proof in \emph{loc. cit.} carries over. To begin, suppose $\etav \da w = \muv$ with $w \in \leftidx^J\taffW$ written as $w = \sigma \tt(\betav)$, so that $\vec_{\mv} = \epsilon_J^-  \cdot  H_{\sigma} Y_{\betav}.$ By definition and using \eqref{inv:Hw0} and \eqref{involution:Y}, we find 
\be{} \overline{ \vec_{\mv}} = \epsilon_J^- \overline{H_{\sigma}} \cdot \overline{Y_{\betav}} = \epsilon_J^- H_{\sigma w_0 } Y_{w_0 \betav} H_{w_0}^{-1}.	 \ee 

\noindent Now we have that \be{} \muv \da  w_0  = \etav \da \sigma \tt(\betav) w_0  = \etav  \da \left(\sigma w_0  \right) \cdot \tt(w_0 \betav) . \ee Writing $\omega:= (\sigma w_0) \tt(w_0 \betav)$, we leave it as an exercise to verify that  if $\sigma \tt(\betav)\in \leftidx^J \taffW$, then $w_0^J \omega \in \leftidx^J\taffW.$ The result  follows since $H_{s_i}, i \in J$ acts on $\epsilon_J$ via the scalar $- \tau^{-1}$.
\end{proof}

Note that by Proposition~\ref{prop:H-act-quantum}, we could have equivalently defined
\be{} \label{involution:Ts} \overline{\vec_{\mv}} := \tau^{-\ell(w_0)} (- \tau)^{-\ell(w_0^J)}  \vec_{\muv \da w_0}  \cdot (\Tmwq_{w_0})^{-1}.  \ee

\begin{nrem}\label{rem:involution:computation} By definition, if $\etav \in \tnalc$, then $\vec_{\etav} = \epsilon_J$ and so $\ov{\vec_{\etav}} = \vec_{\etav}.$ Hence, the above Lemma implies that 
\be{} \label{involution-in-alcove} 
\overline{\vec_{\etav}} = (- \tau)^{-\ell(w_0^J)}  \vec_{\etav \da w_0} H_{w_0}^{-1} = \vec_{\etav}. 
\ee
But the action of $H_{w_0}$ on $\vec_{\etav}$ can be described in terms of the formulas \eqref{Tsi:vmv-cases}: we have 
\be{} \vec_{\etav \da w} H_{s_i} =     \vec_{\etav \da w} \cdot (\tau^{-1} \Tmwq_{s_i}) = c_i \, \vec_{\etav \da ws_i } 
& \mbox{ if } & \ell(w s_i) > \ell(w),
\ee 
where $c_i=1$ if $i \notin J$ and $c_i = - \tau^{-1}$ if $i \in J$. 
\end{nrem}

\tpoint{The involution on $\tVs(\etav)$}\label{subsub:qtriangularityVsph} In order to construct a Kazhdan--Lusztig basis on the space $\tVs(\etav)$ introduced in \S\ref{subsub:Vsp-eta}, we need the following result.

\begin{nprop} \label{prop:involution:quantum} Fix $\etav \in \tnalc$ with $J \subset I_{\aff}$ defined as in~\S\ref{subsub:moduletVetav}. Suppose $\lv \in Y_+ \cap \etav \da \taffW$. Then we have
	\begin{enumerate}
		\item The involution $d: \tVs(\etav) \rr \tVs(\etav)$ takes the following form: 
			\be{}\label{inv-quantum-sph} d( [\vec_{\lv}]) = (-1)^{\ell(w_0^J)} \tau^{-\ell(w_0)- \ell(w_0^J)} [\vec_{\lv \da w_0}]. \ee
	\item With respect to the dominance order $<$, for some $a_{\mv, \lv} \in \At$ one has
	\be{} \label{triangulatiry-involution:quantum} d([ \vec_{\lv}]) = [\vec_{\lv}]+ \sum \limits_{\substack{\mv \in \eta \da \taffW \cap Y_+ \\ \mv < \lv} } a_{\mv, \lv} \, [\vec_{\mv}]. \ee  
\end{enumerate} 
	\end{nprop} 

\begin{proof} We have already seen that $d ( \vec_{\lv} ) = (-\tau^{-1})^{\ell(w_0^J)} \vec_{\lv \da w_0} H_{w_0}^{-1}$. 
On the other hand, we have \be{} d[ \vec_{\lv}] = d( \vec_{\lv} \epsilon )= d( \vec_{\lv}) \epsilon =   (-\tau^{-1})^{\ell(w_0^J)} \vec_{\lv \da w_0} H_{w_0}^{-1} \epsilon. \ee 
Since $H_{w_0} \epsilon = \tau^{\ell(w_0)} \epsilon$, we have $H_{w_0}^{-1} \epsilon = \tau^{- \ell(w_0)} $ and the result follows. The triangularity can either be deduced by inspecting the form of  the straightening rules, or by resorting back to the triangularity of the involution in $\epsilon_J^- \taffH \epsilon$ and then using the comparison of the Bruhat order with the dominance order from \eqref{comparing:orders}. We will explain the former approach in more detail in the $\gf$-twisted setting in the next section. 
\end{proof}

\tpoint{The module $\tVs$ and its canonical bases}\label{subsub:qVspherical}  \label{subsub:canonical-bases} For any $\etav \in \tnalc$, using \cite[Lemma 24.2.1]{lus:qg} and part (2) of the previous Proposition applied to the set $Y_+ \cap \etav \da \taffW$ with the dominance order, we may construct a new basis of $\tVs$ starting from $[\vec_{\mv}], \mv \in Y_+ \cap \etav \da \taffW.$ In fact, we have two possibilities: we may construct $\qlket{\lv}$ and $\qlketm{\lv}$ for $\lv \in Y_+ \cap \etav \da \taffW$ characterized uniquely by the property that they are self-dual with respect to the involution described above and satisfying 
\be{} \label{eq:lvectovec}
\qlket{\lv}  &=&  \qlket{\lv}^+:=  [\vec_{\lv}] + \sum_{\muv < \lv } o_{\muv, \lv} \, [\vec_{\muv}], \text{ where } o_{\muv, \lv}  \in \At^+ \text{ and }   \\ 
\label{eq:lvectovec:neg}  \qlketm{\lv}  &=&  [\vec_{\lv}] + \sum_{\muv < \lv } \om_{\muv, \lv} \, [\vec_{\muv}], \text{ where } \om_{\muv, \lv}  \in \At^-,
\ee 
respectively. 
If $y, x \in (\leftidx^J W^I)_{\reg}$ are such that $\etav \da y =\lv, \etav \da x = \mv$, then Lemma \ref{lem:dom-vec} implies $[\vec_{\lv}]=O_y$ and also $ [\vec_{\mv}] = O_x.$ 
Hence  we may conclude that \be{} \label{o:weights:weyl}  \begin{array}{lcr}
o_{\muv, \lv} = o_{y, x} & \text{and } & \om_{\mv, \lv} = \om_{y, x}, 
\end{array} \ee 
where the elements on the right hand side of each equality were defined in Proposition \ref{subsub:JOK} .

The collections $\{\qlket{\lv} \}$ and  $\{ \qlketm{\lv} \}$ will be referred to as \emph{canonical bases} of $\tVs(\etav)$. 
Putting these bases together (for all $\etav \in \tnalc$), we get bases for $\tVs$ indexed by $Y_+$ which are again called canonical bases.

\subsection{Littlewood--Richardson polynomials and the $\tspaff$-module structure on $\tVs$ } 
\tpoint{$\tspaff$-module structure on $\tVs(\etav)$}\label{subsub:actionHsponVspquantum}  

For $\etav \in \tnalc$ and $J$ defined as in \eqref{def:J}, the space $\tVs(\etav) = \epsilon_J^- \taffH \epsilon$ is naturally a right $\tspaff= \epsilon \taffH \epsilon$-module under a right action $\star$ introduced in \S\ref{subsub:VJsph}. 
Using the Satake isomorphism \eqref{satake}, we obtain an action of $\At[\tY]^W$ on $\tVs(\etav)$, written as $\diamondsuit,$ and defined as
\be{} \label{def:diamond} \vec \diamondsuit f  = \vec \star S^{-1}(f) \text{ for } f \in \At[\tY]^W, \vec \in \tVs(\etav). \ee 

\begin{nprop} \label{prop:invariant-action} Let $f(Y) =  \sum_{\lv \in \tY} c_{\lv} \, Y_{\lv} \in \At[\tY]^W$ with $c_{\lv} \in \At$. Then \be{} \label{f:diamond}   [\vec_{\mv}] \, \diamondsuit \, f  = \sum_{\lv} c_{\lv} \, [\vec_{\lv + \mv } ] \text{ for } \mv \in  \etav \da \taffW. \ee  
\end{nprop}
\begin{proof} It suffices to verify this for $f = S(h_{\lv}), \lv \in \tY_+$, since these form a basis of $\At[Y]^W$. Writing  $S(h_{\lv}) = \sum_{\zv \in \tY} c_{\zv, \lv } Y_{\zv}$  with $\zv \in \tY,$  we compute from the definitions  
\be{}  [\vec_{\mv}] \diamondsuit S(h_{\lv}) =  [\vec_{\mv}] \star h_{\lv}  =  \vec_{\mv} \epsilon \, \epsilon  Y_{\lv} \, \epsilon    = \vec_{\mv}   S(h_{\lv}) \epsilon = \sum_{\zv} c_{\zv, \lv} \, [\vec_{\zv + \lv} ]. \ee \end{proof}

\tpoint{The Littlewood--Richardson polynomials}\label{subsub:qLRcoeffs} For $\lv, \mv, \zv \in \tY_+$, the Littlewood--Richardson coefficients $c_{\mv, \lv}^{\zv} \in \zee$ are defined as decomposition numbers for the multiplication of characters $\chi_{\lv} \in \zee[\tY]$ defined in \eqref{char:lv}: 
$\chi_{\lv} \cdot \chi_{\mv}  = \sum_{\zv} \, c_{\mv, \lv}^{\zv} \, \chi_{\zv}.$ 
A deformation of these numbers to polynomials in type $A$ was introduced in~\cite{leclerc:thibon} under the name $q$-Littlewood--Richardson coefficients \cf \cite{haiman:grojnowski, laniniramsobaje} for the extension to general type. In our setup, the polynomials $Q^{\lv}_{\zv, \mv}(\tau)$ can be defined via the relation 
\be{}\label{LR-poly}   [ \vec_{\mv}] \star \ckl_{\lv}= [ \vec_{\mv}] \diamondsuit \chi_{\lv} = \sum_{\zv \in Y_{+}} \, Q^{\lv}_{\zv, \mv}(\tau) [\vec_{\zv}]. \ee 

\begin{nrem} When $n=1$, dropping all tildes from our notation, one has isomorphism of  $H_{\sp}$-modules  \be{}V_{\sp} \cong H_{\sp} \cong \At[Y]^W \ee under which $[\vec_{\mv}]$ gets mapped to $\ckl_{\mv} \in H_{\sp}$ and $\chi_{\mv} \in \zee[Y]^W$, respectively, reducing $Q^{\lv}_{\zv, \mv}$ to $c^{\lv}_{\zv, \mv}$ .  \end{nrem}

\label{subsub:expressionLRcoeffCS} \noindent Using Proposition~\ref{prop:H-act-quantum} and  \eqref{epsilon:properties}, the action of $\epsilon$ on $\tV$ is given, in terms of the basis $\vec_{\mv}, \mv \in Y,$ as \be{}  \vec_{\mv} \epsilon = \sum_{w \in W} \tau^{-\ell(w)} \vec_{\mv} \Tmwq_w \tau^{-\ell(w)}. \ee
Let us introduce the element $\cs^{\flat}(\mv) \in \tV$ (see below for some remarks on this notation)  by the expression 
\be{} \cs^{\flat}(\mv):= \vec_{\mv} \cdot \epsilon = \sum_{w \in W} \vec_{\mv} \cdot \Tmwq_w \tau^{-\ell(w)}\ee and write its image in $\tVs$ as $[\cs^{\flat}(\mv)].$  Recall the polynomials $p_{\tauv, \lv}$ from \eqref{char:HL}.

\begin{nprop}\label{prop:qLRCS} Let $\lv \in \tY_+$ and $\mv \in Y_+$. Then we have 
\be{} \label{alternate:qLR} \sum_{\tauv \leq \lv} \, p_{\tauv, \lv} \, [\cs^{\flat}(\muv) \cdot Y_{\tauv}] =  \sum_{\zv \in Y_+} \, Q^{\lv}_{\zv, \mv}(\tau) [\vec_{\zv}]. \ee
\end{nprop}
\noindent Note that in order to write the left hand side of the above relation \eqref{alternate:qLR} in terms of the basis elements $[\vec_{\zv}],$ we may need to use the straightening relations of Proposition \ref{prop:straightening:no-g}. 

\begin{proof}  
It suffices to show, in light of the relation \eqref{char:HL}, that $ [\vec_{\mv}] \diamondsuit S(h_{\lv})  = [  \cs^{\flat}(\mv) \cdot Y_{\lv} ].$ This follows, since by definition, $ [\vec_{\mv}]  \diamondsuit S(h_{\lv}) = \vec_{\mv}  \epsilon \, Y_{\lv} \epsilon = \cs^{\flat}(\mv) Y_{\lv} \epsilon. $   
\end{proof}

\begin{nrem} The notation $\cs^{\flat}$ comes from the fact that in the metaplectic setting, the analogous expression $\wt{\cs}(\mv)$ gives the metaplectic Casselman--Shalika formula that is featured prominently in the literature on multiple Dirichlet series. Motivated by this, McNamara~\cite{mcnamara:duke} has proved combinatorial formulas expressing $\wt{\cs}(\muv)$ as a sum over the crystal $\mathcal{B}_{\muv+\rhov}$ (or alternatively as a sum over certain Gelfand--Tsetlin patterns in type A) and in~\cite{bbb}, lattice models are studied whose partition function yields $\wt{\cs}(\mv)$. Although both results are proved in the metaplectic setting, one may treat the parameters introduced there as formal variables and perform a `quantum specialization' to obtain what appears to be new combinatorial formulas for $Q_{\zv, \mv}^{\lv}(\tau).$  \end{nrem}

\subsection{Tensor product theorems}  \label{sub:tensorproduct}

We formulate below two results for the bases $\qlket{\lv}$ and $\qlketm{\lv}$ introduced in \S \ref{subsub:qVspherical}. Modeled on the Steinberg-Lusztig theorem (see Thm. \ref{thm:quantum:tensor:products} (1)), the result for $\qlketm{\lv}$ seems to originate in the work of \cite{leclerc:thibon} in type A. The proof sketched below follows the nice argument due to Lanini-Ram \cite{laniniram} for general type. In fact, the argument in \emph{op. cit.} can be easily modified to also give a tensor product theorem for $\qlket{\lv}$ which matches a result due to Andersen for tilting modules (see Thm. \ref{thm:quantum:tensor:products}(2)).

\tpoint{Tensor product theorem}\label{subsub:tensorprodthm} Recall the notion of $(\Qs, n)$-restricted coweights $\Box_{\qsn}$ from \S\ref{subsub:boxes}. 

\begin{nprop}\label{prop:TensorProd} Let $\lv \in Y_+$ be decomposed uniquely as $\lv = \lv_0 + \zv$, where $\zv \in \tY_+$ and $\lv_0 \in \Box_{\qsn}$. 
	\begin{enumerate}
		\item \cite[Thm 1.4]{laniniram} We have  $\qlketm{\lv_0} \star \tckl_{\zv} = \qlketm{\lv_0 + \zv}. $
		\item Writing $\lvbar_0:= \lv_0 \cdot w_0 + 2 (\t{\rho}^{\vee} - \rhov)$ we have  \be{}\label{tilting:tensor}\qlket{\lvbar_0} \star \tckl_{\zv} = \qlket{\lvbar_0 + \zv}. \ee 
	\end{enumerate}
\end{nprop}

\noindent In both cases, the idea is to verify that the left hand side of the desired equality satisfies the characterizing properties of the right hand side (see \S\ref{subsub:qVspherical}). We shall just focus on the $\qlket{\mv}$ variant. Since \be{}\label{verify:dom} \la \lvbar, a_i \ra = \la \lv_0,  a_i \cdot w_0 \ra + 2 ( n(\av_i)-1 ), \ee  $\lv_0 \in \Box_{\qsn},$ and $a_i \cdot w_0$ is the negative of some simple root, it follows that $\lvbar \in Y_+.$ 
\tpoint{Proof of Proposition \ref{prop:TensorProd}, part 1: self-duality}
Pick $\etav \in \tnalc$ with $J$ defined as in \eqref{def:J}.  We assume $\qlket{\lvbar_0} \in \tVs(\etav)$ is written as $\qlket{\lvbar_0} = \epsilon_J^- h \epsilon$ for $h \in \tspaff.$ Write also  $\tckl_{\zv} =  \epsilon h' \epsilon$ for some $h', h' \in \taffH.$ As both of these elements are self-dual,  
\be{} \begin{array}{lcr} d \left( \qlket{\lvbar_0} \right) =  \overline{\vec_{\etav} h \epsilon  } = \vec_{\etav} \overline{h} \epsilon  = \vec_{\etav}  h \epsilon = \qlket{\lvbar_0} & \text{ and } & d(\tckl_{\zv}) = \overline{\epsilon h' \epsilon} = \epsilon \overline{h'} \epsilon = \tckl_{\zv}, \end{array} \ee we may apply the fact that $d$ is a homomorphism to conclude
\be{} d \left( \qlket{\lvbar_0} \star \tckl_{\zv} \right) = d( \vec_{\etav} h \epsilon \epsilon h' \epsilon  ) = \overline{\vec_{\etav}} \overline{h} \epsilon \overline{h'} \epsilon = \qlket{\lvbar_0} \star \tckl_{\zv} .\ee 

\tpoint{Proof of Proposition \ref{prop:TensorProd}, part 2: triangularity} We may write  
\be{} \begin{array}{lcr} \qlket{\lvbar_0} = \vecket{\lv_0} + \sum_{\mv < \lv_0} o_{\mv, \lv_0} \, \vecket{\mv} & \mbox{ and } & \ckl_{\zv} = h_{\zv} +  \sum_{\etav <  \zv} p_{\etav, \zv} \, h_{\etav} \end{array}\ee 
for elements $o_{\mv, \lv_0} \in \At^+$ and $p_{\etav, \av} \in \At^+$ defined earlier.  From the straightening rules and the proof of Proposition \ref{prop:invariant-action}, we have for any $\lv \in Y_+$ and $\mv \in \tY_+$  \be{} \vecket{\lv} \star h_{\mv}  = \vecket{\lv + \mv} + \sum_{\xv < \lv + \mv} d_{\xv, \lv + \mv } \vecket{\xv} \text{ for some } d_{\xv, \lv+ \mv} \in \At. \ee From here it follows that 
\be{} \qlket{\lvbar_0} \star \tckl_{\zv} = \vecket{\lv} + \sum_{\mv < \lv} a_{\mv, \lv} \vecket{\mv}, \mbox{ for } a_{\mv, \lv} \in \At; \ee
it remains to show $a_{\mv, \lv} \in \At^+$.  To do so, recall a few facts about Littelmann's path model for representations. 

\newcommand{\wtp}{\mathrm{wt}}

\tpoint{Recollections on the path model} \label{subsub:recall-paths} We apply the path model of Littlelmann \cite{littelmann} to the root system defined by $\wt{\mf{D}}$. From \emph{op. cit.}, the character of a highest representation is written as a \emph{positive} sum over \emph{paths} in the highest weight crystal $B(\zv)$,\textit{ i.e.}   \be{} \label{char:paths} \chi_{\lv} := \sum_{p \in B(\zv) } Y_{\wtp(p)}, \ee where $\wtp(p) \in \tY$ denotes the end-point of the path. Recall also that there exists an involution on the set  $\iota: B(\zv) \setminus \{ p_{\zv} \},$ where $p_{\zv}$ is the unique path in $B(\zv)$ with endpoint $\zv$, which has the property that \be{} \wtp(\iota(p)) =  \wtp(p) \da s_i  \text{ for a unique } i \in I. \ee 

\tpoint{Proof of Proof of Proposition \ref{prop:TensorProd}, part 3: straightening}  
Write $a \equiv b$ to mean $a - b \in \oplus_{\lv \in Y} \At^+[Y],$ \textit{i.e.} it is a linear combination of $Y_{\mv}$ with $\mv \in Y$ and the coefficients are in $\At^+$. 
In this notation, the straightening rules in Proposition \ref{prop:straightening-sph} imply that for $\mv \in Y_+$, 
\be{} \label{straightening:equiv} 
\yket{\mv  } \equiv \begin{cases} 0 & \text{ if }   0 \leq \la \mv + \rhov, a_i \ra < n(\av_i), \\ - \yket{\mv^{(1)}\da s_a }  & \text{ otherwise}. 
\end{cases} \ee 

\begin{nlem}{\cite[Lemma 1.5]{laniniram}} In the notation of part (2) of Proposition \ref{prop:TensorProd} \be{}\begin{array}{lcr} \yket{\lvbar_0+ \zv} \equiv - \yket{\lvbar_0 + \zv \da s_i} & \text{ for any  } & i \in I. \end{array} \ee  \end{nlem}

\begin{proof} Recall $\lvbar = \lv_0 \cdot w_0 + 2 (\t{\rho}^{\vee} - \rhov)$ and write $\mv = \lvbar + \zv.$ 
We compute 
\be{} 
	\la \mv + \rho, a_i \ra &=& \la \lv_0 \cdot w_0 , a_i \ra + 2(n(\av_i) - 1) + 1 + \la \zv, a_i \ra \\ 
	&=& \la \lv_0 \cdot w_0 , a_i \ra + 2n(\av_i) - 1 + \la \zv, a_i \ra.
\ee 
Using that $\la \zv, a_i \ra \geq 0$ and $\zv \in \tY$ as well as the fact that $-n(\av_i) < \la \lv_0 \cdot w_0 , a_i \ra \leq 0$  we conclude that 
\be{} \resi_{n(\av_i)}  
\la \mv + \rho, a_i \ra = \la \lv_0 \cdot w_0, a_i \ra + n(\av_i)-1. 
\ee 
Using the definition that $\mv^{(1)} = \mv^{(1)} - \la \mv + \rhov, a_i \ra \av_i$ we obtain  
\be{} \mv^{(1)} \da s_i &=& \mv^{(1)} \cdot s_i - \av_i = \lvbar + \zv -  \resi_{n(\av_i)}  \la \mv + \rho, a_i \ra - \av_i \\ 
	&=& \lv_0 \cdot w_0 \cdot s_i + 2 (\t{\rho}^{\vee} - \rhov - \av_i) + \zv \cdot s_i + \left(\la \lv_0 \cdot w_0, a_i \ra + n(\av_i)-1 \right) \av_i - \av_i  \\ 
	&=& \lv_0 \cdot w_0 +  2 (\t{\rho}^{\vee} - \rhov ) + \zv \cdot s_i - \av_i.
\ee 
The rest of the argument follows as in Lemma \cite[Lemma 1.5]{laniniram} 
\end{proof}

\tpoint{Proof of Proof of Proposition \ref{prop:TensorProd}, part 4: conclusion}Since the $o_{\mv, \lv} \in \At^+$, we have that \be{}\qlket{\lvbar_0} \star \tckl_{\zv} \equiv \vecket{\lvbar_0} \star \tckl_{\zv}. \ee From our description of the action by $\tckl_{\zv}$ in \eqref{LR-poly} and formula \eqref{char:paths}, this amounts to checking that  \be{} \vecket{\lvbar_0} \diamondsuit \chi_{\zv} = \sum_{p \in B(\zv)} \vecket{\lvbar_0 + \wtp(p)} \equiv \vecket{\lvbar_0 + \zv}. \ee Using the previous Lemma and what we recalled about the involution $\iota$ in \S\ref{subsub:recall-paths}, we are done.

\section{$\gf$-twisted canonical bases and Kazhdan-Lusztig polynomials }\label{sec:gKL}

Recall the generic ring $\Zvg$ from \S\ref{notation:Cvg} which was defined in terms of formal parameters $\tau$ (where $\tau^2= t^{-1}$) and the $\gf_k, k \in \zee$  subject to the relations as described in \emph{loc. cit.}  In this section, we introduce a `$\gf$-twisted' version of the constructions from the previous one. Under the quantum specialization (see \S\ref{notation:quantum-spec}) $\gf_k \mapsto \tau, \, k \neq 0$ we recover the constructions from the previous section, and under the $p$-adic specialization of \S\ref{notation:padic-spec}), as we explain in Part \ref{part:padic} of this paper, the constructions of the present section provide information about the structure of the Whittaker module on covers of $p$-adic groups. 

We keep the same conventions as at the start of \S\ref{sub:twisted-affine-combinatorics}. Namely, we fix a root datum $(I, \cdot, \mf{D})$, which is written  $\mf{D}= (Y, \{ \av_i \}, X, \{a_i\})$, and a twist $(\Qs, n)$ on $\mf{D}$ such that the twisted root datum $(I, \circ_{(\Qs, n)}, \wt{\mf{D}})$ is of simply-connected type. 
Set $\wt{\mf{D}} = (\t{Y}, \{\tav_i\}, \t{X}, \{\ta_i\})$ and write $\taffH:= \affH(I, \circ_{(\Qs, n)}, \wt{\mf{D}}) $ and $\tspaff$ for the corresponding affine Hecke algebra and spherical Hecke algebra, respectively. Recall that the Bernstein presentation for $\taffH$ allows us to identify it as 
\be{}\taffH \cong  H_W \otimes \At[\tY], \ee where $H_W$ is the finite Hecke algebra of $W:= W(I, \cdot)\cong W(I, \circ_{\qsn})$ with generators $T_{s_i}, i \in I$ , or their renormalizations $ H_{s_i} = \tau^{-1} T_{s_i},$ satisfying the braid relations and the quadratic relations 
\be{}\label{quad:relations:summary} (T_{s_i} - \vi) (T_{s_i}+1)= 0 \mbox{ or equivalently } ( H_{s_i}- \tau)(H_{s_i}- \tau^{-1}) = 0. \ee
The key to our constructions in this section is to use a representations of the affine Hecke algebra $\taffH$ on a space of polynomials $\Zvg[Y]$. Such a construction was introduced in \cite{cgp, PPAIM}, and what is new here is the use of ideas from the previous section to equip $\Zvg[Y]$ with an explicit decomposition into $\taffH$-submodules and an involution that allows us to introduce a $\gf$-twisted version of Kazhdan--Lusztig theory.

\subsection{Metaplectic Demazure--Lusztig operators } \label{subsub:metDL} 
Our goal in this section is to construct a representation of $\taffH$ on a certain space of polynomials. 
Our computations will take place in the ring $\Zvg[Y]$ (or $\Zvg[\tY]$). We also introduce localizations $\Zvgm[Y]$ and $\Zvgm[\tY]$ by the ideal \be{} \label{ideal:m} \mf{m} = \left(1 - Y_{-\tav_i},  1 - \v Y_{-\tav_i} \mid i \in I  \right) \subset \Cvg[Y],\ee and  for a positive integer $n > 0$, define the residue map $\mathrm{res}_{n}: \zee \rr \{ 0, 1, \ldots, n-1 \}. $

\tpoint{The Chinta--Gunnells action} \label{cg-not} 
Following Chinta and Gunnells \cite[Def. 3.1]{chinta:gunnells}, define 
{\small \be{cg-act} Y_{\lv} \star s_a  = \frac{Y_{\lv \cdot s_a}}{1 - \v Y_{-\tav}} \left[ (1-\v) Y_{ \resi_{n(\av)} \left( \frac{ \Bs(\lv, \av)}{\Qs(\av)} \right) \av } - \v \gf_{ \Qs(\av) + \Bs(\lv, \av) } Y_{ \tav - \av} (1 - Y_{-\tav}) \right] \text{ for } \lv \in Y, a \in \Pi,\ee} 
where $\lv \cdot s_a$ denotes the right action of $s_a$ on $\lv$ and $\Pi$ denotes the set of simple roots attached to $\mf{D}$. 
Extend this by $\Zvg$-linearity to define $f \star s_a$ for every $f \in \Zvg[Y]$ and use the formula \be{cg-act-quot} \frac{f}{h} \star s_a = \frac{f \star s_a}{h^{s_a}} \text{ for } f \in \Zvg[Y], h\in \mf{m}, \ee to further extend this to an action $\star$ of $W$: $- \star s_a: \Zvgm[Y] \rr \Zvgm[Y].$ 
Note that for $a \in \Pi$ we have that  
\be{cgw} (f \cdot h) \star s_a =  (f \star s_a) \,  \cdot h^{s_a} \text{ for } h\in \Zvg[\tY], f \in \Zvg[Y], \ee where $h^{s_a}$ is the usual action of $s_a$ on $\Zvg[Y]$.

\tpoint{Metaplectic Demazure--Lusztig operators} \label{s:metDL} Introduce the rational functions 
\be{met:cb} \begin{array}{lcr} \b(X) = \frac{ \vi - 1}{1 - Y_{-X}} &  \text{and }  & \c(X) = \frac{\vi - Y_{ -X}}{1 - Y_{ X}}. \end{array} \ee 
For each $a \in \Pi$ with corresponding simple reflection $s_a$, we define the following  elements in $\Zvg(Y)[W]^{\vee}$:  
\be{} \label{T:c-b} 
\Tmw_{s_a}:= \Tmw_a := [s_a]\c(\tav)   + [1] \b(\tav) . \ee

\noindent Consider now the action on $\Cvgm[Y]$ of $\Tmw_a$ by the formulas 
\be{Tma:act} Y_{\lv} \cdot \Tmw_a = \c(\tav) Y_{\lv} \star s_a + \b(\tav)  Y_{\lv}. \ee

\begin{nrem}
	Note that our  operators $\Tmw_{a}$ are the inverse of the ones in~\cite[eq. (4.10)]{PPAIM} (and of course one should replace $v$ with $\tau^{-2}$). They also act on the right here, as opposed to the left in \emph{op. cit.}  
\end{nrem}

\noindent 
The operators $\Tmw_a$ will be used to define a representation of the affine Hecke algebra on $\Cvg[Y]$ (see \S\ref{subsub:met-poly}). Recall the dot action of the affine Weyl groups $\affW$ and $\taffW$ on $Y$ from  \ref{subsub:dotaction} and \eqref{dot-action-twisted}, respectively, as well as the fundamental domain $\tnalc$ for the dot action from~\eqref{def:fund-alcove-dot-action-twisted}.  From the computation in \cite[(4.18) and (4.12)]{PPAIM}\footnote{Note that in (4.12) there is a mistake; the  right hand side of the equality should be replaced by $Y_{\lv}+v \gf_{\Qs(\av)}Y_{\lv - \av}$}, we deduce the following formulas. 
 \label{subsub:rank1met} 

\begin{nlem}\label{lem:explicitDLoperators} \begin{enumerate}
		\item For $a \in \Pi$ and $\lv \in Y$, we have \be{} \label{rank1:metDL} Y_{\lv} \cdot \Tmw_a  = \begin{cases} 
		\gf_{\la \lv + \rhov, a \ra \,  \Qs(\av)} Y_{\lv \da s_a } + (\vi-1) \sum \limits_{\substack{k \geq 0 \\ k n(\av) \leq \la \lv, a \ra }} Y_{\lv - k n(\av) \av} & \text{ if } \la \lv, a \ra \geq 0, \\
		\gf_{\la \lv + \rhov, a \ra \, \Qs(\av)} Y_{\lv \da s_a} + (1-\vi) \sum \limits_{\substack{k > 0 \\ k n(\av) < - \la \lv, a \ra  }} Y_{\lv + k n(\av) \av} & \text{ if } \la \lv,  a \ra < 0.
	\end{cases} \ee 

		\item  For $\mv = \etav \da W$ and $i \in I,$  we have
			\be{} \label{Tsi:Ymv-cases} Y_{\muv} \cdot \left( \tau^{-1} \Tmw_{s_i} \right)= 
				\begin{cases}  
					\tau^{-1} \, \gf_{\la \mv + \rhov, a_i \ra \, \Qs(\av_i )}Y_{\muv \da s_i  } & \mbox{ if } -n(\av_i)  <  \la \muv + \rhov, a_i \ra < 0,   \\
					\tau^{-1} \, \gf_{\la \mv + \rhov, a_i \ra \, \Qs(\av_i )} \, Y_{\muv \da s_i} +(\tau-\tau^{-1}) Y_{\muv} & \mbox{ if }  0 < \la \muv + \rhov , a_i \ra < n(\av_i),    \\
					-\tau^{-1}  Y_{\muv} = -\tau^{-1} Y_{\muv \da s_i}  & \mbox{ if } \la \muv + \rhov, a_i \ra = 0. 
			\end{cases}  \ee
		\item If $\mv \in Y, \zv \in \tY$ and $i \in I$, we have
		\be{}\label{eq:Tsi:Ymv-cases:extension} (Y_{\mv} Y_{\zv}) \cdot (\tau^{-1} \Tmw_{s_i}) = Y_{\mv}  \cdot (\tau^{-1} \Tmw_{s_i}) + (\tau - \tau^{-1}) \frac{Y_{\zv \cdot s_i} - Y_{\zv}}{1- Y_{-\av} }.  \ee   \end{enumerate}
\end{nlem}

\begin{nrem} 
Note that the formulas in~\eqref{rank1:metDL} imply and are implied by the formulas in~\eqref{Tsi:Ymv-cases} and~\eqref{eq:Tsi:Ymv-cases:extension}.
\end{nrem}

\tpoint{Statement of result } The following computation plays a crucial role in our work. In making it we were inspired by the straightening formulas of \cite[(1.3)]{laniniramsobaje} and \cite[Prop. 6.3]{haiman:grojnowski}, \cf \cite{leclerc:thibon}.

\begin{nprop} \label{prop:straightening-rules-Hecke} Let $\lv \in Y$, $a \in \Pi$, and $n(\av)$ as in \eqref{ni:def}. 
\begin{enumerate}
\item If $\la \lv +  \rhov, a \ra \geq 0$ and $\la \lv + \rhov, a \ra \equiv 0 \mod n(\av)$ then  
\be{} \label{Taction:1}  \left(Y_{\lv} + Y_{\lv \da s_a}\right) \Tmw_a = -   \left( Y_{\lv} + Y_{\lv \da s_a} \right). \ee

\item If $0 <  \la \lv + \rhov , a \ra  < n(\av) $, then 
\be{} \label{Taction:2}  \left( Y_{\lv} - \gf_{\la \lv + \rhov, a \ra \, \Qs(\av)}Y_{\lv \da s_a} \right) \Tmw_a = - \,  \left( Y_{\lv} - \, \gf_{\la \lv + \rhov, a \ra \, \Qs(\av)}  Y_{\lv \da s_a} \right). \ee

\item If $\la \lv + \rhov, a \ra > n(\av)$ but $j:= \la \lv + \rhov, a \ra \not\equiv 0 \mod n(\av) $,  then setting $ \lv^{(1)}:= \lv - \resi_{n(\av)} (\la \lv +\rhov, a \ra) \, \av, $
\be{} \label{Taction:3}  \left( Y_{\lv} - \gf_{j\, \Qs(\av)} \, Y_{\lv \da s_a} - \gf_{j\, \Qs(\av)} Y_{\lv^{(1)}} +  Y_{ \lv^{(1)} \da s_a}  \right) \Tmw_a = -  \left( Y_{\lv} - \gf_{j \, \Qs(\av)} \, Y_{\lv \da s_a} - \gf_{j\, \Qs(\av)} Y_{\lv^{(1)}} +  Y_{ \lv^{(1)} \da s_a}  \right). \ee 		
\end{enumerate} 
\end{nprop} 

\noindent The proof will occupy the remainder of this section.  Set $m := n(\av) \geq 1$,  $\tav:= m \av,$ and write 
\be{} \label{lv:j:d}  \la \lv + \rhov, a \ra = j + m d, \text{ where } j \in \{ 0, 1, \ldots, m-1 \} \text{ and } d \in \zee. \ee 
Let us also record here that $  \la \lv \da s_a, a \ra = - \la \lv, a \ra - 2. $

\tpoint{Proof of Proposition \ref{prop:straightening-rules-Hecke} (1)} 	In case  (1), we must have  $j=0$ and $d \geq 0$. Assume further $d=0$ so that $\la \lv + \rhov, a \ra = 0 $ and hence 
\be{} \lv \da s_a = ( \lv + \rhov) s_a - \rhov = \lv + \rhov - \rhov = \lv. \ee 
Since $ \la \lv, a \ra = \la \lv \da s_a, a \ra = -1 $ we may apply the second case of~\eqref{rank1:metDL} to conclude that 
\be{} \left( Y_{\lv} \right)  \Tmw_a  =  \left( Y_{\lv \da s_a} \right) \Tmw_a =   - Y_{\lv}, \ee 
where we used  $\gf_{\la \lv + \rhov, a \ra \, \Qs(\av) } = \gf_0 = -1 $. 
Next, we assume that $d \geq 1$ in \eqref{lv:j:d} and that $j=0$, so 
\be{} \la \lv, a \ra = md -1 \geq 0 \text{ and } \la  \lv \da s_a, a \ra = - md -1 <0. \ee 
Applying the first and second case of \eqref{rank1:metDL} to $Y_{\lv}$ and $Y_{ \lv \da s_a} = Y_{\lv - d m \av}$, respectively, we find  
\be{}  \left(Y_{\lv} \right) \Tmw_a &=& - Y_{\lv \da s_a} + (\tau^2-1)\sum_{0 \leq  k \leq d-1} Y_{\lv - k \tav}, \text{ and  } \\ 
\left(Y_{\lv \da s_a} \right) \Tmw_a &=& - Y_{\lv} + (1-\tau^2)\sum_{0 < k \leq d} Y_{\lv - d \, \tav + k \, \tav}. \ee 
If we add the above expressions, the terms in the sums over $k$ cancel, which leaves with
\be{}  \left(Y_{\lv} + Y_{\lv \da s_a}\right) \Tmw_a = -   \left( Y_{\lv} + Y_{\lv \da s_a} \right) .\ee

\tpoint{Proof of Proposition \ref{prop:straightening-rules-Hecke} (2)} In case (2), we must have $d=0$ and so $j = \la \lv + \rhov, a \ra > 0$.  Hence $\la \lv, a \ra \geq 0$ and we may apply the first case of \eqref{rank1:metDL} to obtain 
\be{} \left( Y_{\lv} \right)\Tmw_a = \gf_{j \Qs(\av)} Y_{\lv \da s_a} +(\tau^2-1)Y_{\lv}. \ee 
On the other hand, to compute $\left(Y_{\lv \da s_a}\right) \Tmw_a$ we note (see after \eqref{lv:j:d})  that 
$\la \lv \da s_a, a \ra = -( j+1 ) < 0 $ 
so that the second case of (\ref{rank1:metDL}) yields  
\be{}\left(Y_{\lv \da s_a}\right) \Tmw_a =\gf_{- j \Qs(\av)} Y_{\lv}, \ee 
since the sum over $k$ in the second expression of~\eqref{rank1:metDL} has no terms since $j+1 \leq m$. One then has 
\be{}\left(\gf_{j \Qs(\av)} Y_{\lv \da s_a} -  Y_{\lv} \right) \Tmw_a &=& \gf_{j \Qs(\av)} \gf_{- j \Qs(\av)} Y_{\lv}  - \gf_{j \Qs(\av)} Y_{\lv \da s_a} -(\tau^2-1)Y_{\lv} \\ 
&=& - \left( \gf_{j \, \Qs(\av)} Y_{\lv \da s_a} - Y_{\lv}  \right), \ee 
where we have used the fact that $\gf_{j \Qs(\av)} \gf_{-j \Qs(\av)} = \tau^2$ whenever $j \neq 0$ and $\Qs(\av) \neq 0 \mod n$ (see \S\ref{notation:Cvg}).

\tpoint{Proof of Proposition \ref{prop:straightening-rules-Hecke} (3)} In case (3), we have $d\geq1$ and also $j \neq 0$ in~\eqref{lv:j:d}. Hence $\la \lv, a \ra = j + md -1 > 0$, so that we may apply the first case of \eqref{rank1:metDL} to write 
\be{ta:1-1}  \left(Y_{\lv} \right) \Tmw_a  &=& \gf_{j \, \Qs(\tav)} Y_{\lv \da s_a} + (\tau^2-1)\sum_{0 \leq k \leq d} Y_{\lv - k \tav}. \ee 
Consider next 
$\lv^{(1)} \da s_a=  (\lv - j \av) s_a - \av = \lv - m d \av $ 
from which we conclude  
\be{} \la \lv^{(1) } \da s_a, a \ra = - ( \lv^{(1)}, a)+2  = j - nd -1 < 0 \ee since $md \geq m$ and $0 < j < m$. Hence we can apply the second case of \eqref{rank1:metDL} to obtain 
\be{} ( Y_{ \lv^{(1)}\da s_a}) \Tmw_a = \gf_{j \, \Qs(\av)} Y_{ \lv^{(1)}}  + (1-\tau^2 ) \, \sum \limits_{\substack{k > 0 \\  k m <  md - j +1 } } Y_{ \lv^{(1)} \da s_a + k \tav} = \\ 
\gf_{j \, \Qs(\av)} Y_{ \lv^{(1)}}  + (1-\tau^2) \, \sum_{0 < k < d} Y_{ \lv - d \tav + k \tav}
= \gf_{j \, \Qs(\av)} Y_{ \lv^{(1)}}  + (1-\tau^2 ) \, \sum_{0 < k < d } Y_{ \lv -  k \tav}.  \ee 
Hence, we find that  
\be{}  \left(Y_{\lv} + Y_{ \lv^{(1)} \da s_a} \right) \Tmw_a =  \gf_{j \, \Qs(\av)} \left(Y_{\lv\da s_a} + Y_{ \lv^{(1)}} \right) + (\tau^2-1)\left( Y_{\lv} + Y_{\lv^{(1)} \da s_a} \right). \ee  
If we now define  
\be{}\label {A:B} A:= Y_{ \lv } + Y_{ \lv^{(1)} \da s_a} \text{ and  } B:= Y_{ \lv \da s_a } + Y_{\lv^{(1)}}, \ee we may restate what we have proven as follows:
\be{} \label{Hecke:TaA} (A ) \Tmw_a = \gf_{j \Qs(\av)} B + (\tau^2-1)A. \ee 
Now recalling the quadratic relation (\ref{mdl:quadratic}), we may apply $\Tmw_a$ to the previous relation to get 
\be{} \label{Hecke:TaB} (\gf_{j \Qs(\av)}B) \Tmw_a = ( A \Tmw_a -(\tau^2-1)A)\Tmw_a = \tau^2A. \ee 
Using $ \gf_{j \Qs(\av)} \gf_{-j \Qs(\av)} = \tau^2$ we get $(B) \Tmw_a =  \gf_{-j \Qs(\av)} A$, hence we conclude using~\eqref{Hecke:TaA} that $A -  \gf_{j \Qs(\av)} B$ is an eigenvector of $\Tmw_a$ with eigenvalue $-1$ as desired. 

\begin{nrem} 
In Proposition~\ref{prop:straightening-rules-Hecke} we essentially compute eigenvectors with eigenvalue $-1$ for the operators $\Tmw_a$. These are used to define a certain ``exterior power'' quotient in~\S\ref{subsub:straightening_rules_Vsph} which models the metaplectic spherical Gelfand--Graev representation.   
A similar exterior power is known to model the representations theory of quantum groups at a root of unity; we'll make this more precise in~\S\ref{sec:qg_at_rou}. 
One may similarly compute $\g$-twisted eigenvectors with eigenvalue $\tau^2$ for $\Tmw_a$ and use them to define ``symmetric power'' quotients. This doesn't have any immediate applications to the $p$-adic setting, so we will not pursue it here.   
\end{nrem}

\tpoint{The classical straightening rules}
Note that Proposition~\ref{prop:straightening-rules-Hecke} holds for any $\Qs$ as defined in~\ref{subsub:MetStructure}, including when $\Qs(\av) \equiv 0 \mod n$. 
In this case (owning to $\gf_{mk} =-1$), we have that equations~\eqref{Taction:1} and~\eqref{Taction:2} are the same and equation~\eqref{Taction:3} is a linear combination of~\eqref{Taction:1} for $\lv$ and $\lv^{(1)}$. The relations become classical in nature, and the straightening rules we develop in Proposition~\ref{prop:straightening-sph} will be the same to the classical straightening rules in~\eqref{Lusztig-straightening-rule}. 
This is consistent with the fact that when $\Qs(\av) \equiv 0 \mod n$ for all $\av \in \Piv$, then the corresponding metaplectic $n$-cover of $\mathbf{G}(\K)$ is just the direct product $\mathbf{G}(\K) \times \mu_n$.

\subsection{The metaplectic polynomial representation  $\V$}  \label{sub:met-poly} 

\tpoint{The representation $\V$}  \label{subsub:met-poly}

Let $\V:= \Zvg[Y]$ be the space of polynomials in the variables $Y_{\lv}, \lv \in Y$.  The elements in $\At[\tY]$ act on $\V$ by translations, i.e. 
\be{}\label{action:translations} Y_{\lv} \cdot Y_{\mv} = Y_{\mv + \lv} \text{ for } \mv \in \tY, \lv \in Y. \ee From the formulas in Lemma \ref{lem:explicitDLoperators}, one sees that $\Tmw_{a}$ preserves the space $\V.$ In fact, one has the following.

\begin{nprop} \label{prop:HeckeActionY}Keep the notation above and recall the space $\tV$ from \S\ref{subsub:V-quant} equipped with its $\taffH$-action.
\begin{enumerate}
\item The formulas (\ref{action:translations}) together with $Y_{\lv} \cdot T_{s_i}:= Y_{\lv} \cdot \Tmw_{a_i}$ for $i \in I, \lv \in Y$ define an action of $\taffH $ on $\V$. 
\item The quantum specialization $\mf{q}: \Zvg \rr \At$ induces an isomorphism,
which we continue to denote by the same name, $\mf{q}: \At \otimes_{\Zvg} \V \stackrel{\cong}{\longrightarrow} \tV$ that is equivariant with respect to the $\taffH$-actions.  
\end{enumerate}
\end{nprop}
\begin{proof}From the arguments in \cite[Prop. 3.1]{cgp} we know the operators $\{ \Tmw_{a_i} \mid i \in I \}$ satisfy the braid relations as well as the quadratic relation 
\be{}\label{mdl:quadratic} \Tmw_{a}^2 = (\vi-1) \Tmw_{a} + \vi \text{ for } a \in \Pi. \ee 
This shows that the given formulas define an action on $\V$. A more conceptual proof of this fact that doesn't use the Chinta-Gunnells action is given in~\cite[Theorem 3.7]{ssv1}. The second part follows by observing that the formulas in Lemma \ref{lem:explicitDLoperators} reduce to the formulas for $\Tmwq_{a}$ under the quantum specialization. 	
\end{proof}

\tpoint{Decomposition of $\V$}
\label{subsub:Vstructure}

In analogy with the constructions of \S\ref{subsub:moduletVetav}, for each $\etav \in \tnalc$, define  the right $\taffH$-module \be{} \label{V:Ov} \V(\etav):=  Y_{\etav} \cdot \taffH := \{ w \in \V \mid w= Y_{\etav} \cdot h  \text{ for some } h \in \taffH \}.  \ee The following result follows immediately from the explicit formulas in Lemma \ref{lem:explicitDLoperators}.

\begin{nprop} \label{prop:Wdav} For each $\etav \in \tnalc$, $\V(\etav)$ is a $\Zvg[\tY]$-module of rank equal to the cardinality of the orbit $\etav \da W,$ and we have
	$ \V(\etav) = \bigoplus_{\zv \in  \etav \da W } Y_{\zv} \cdot \Zvg[\tY] $ and a decomposition of $\V$ into $\taffH$-submodules
	\be{}\label{V:dec} \, \V := \Zvg[Y] \cong \bigoplus_{\etav \in  \tnalc} \,  \V(\etav) . \ee  
\end{nprop}
It is then clear that a basis of $\V(\etav)$ is given by $Y_{\muv}$ for $\muv \in \etav \da \taffW$.

\tpoint{The basis $\vv_{\mv}$} \label{subsub:vv-basis} Let $\etav \in \tnalc$ and $J$ as in \eqref{def:J}. In analogy with the construction of the elements $\vec_{\mv}$ from \S\ref{subsub:moduletVetav}, we introduce elements $\vv_{\mv}$ for $\mv \in \etav \da \taffW$ by the formula 
\be{} 
\begin{array}{lcr} \vv_{\mv}:= Y_{\etav} H_\sigma Y_{\betav} & \text{ when } & \mv = \etav \da \sigma \tt(\betav), \, \, \sigma \tt(\betav) \in \leftidx^J \taffW, \, \, \sigma \in W, \betav \in \tY. \end{array} 
\ee 
where we recall the action of $H_{s_i} =  \left( \tau^{-1} \Tmw_{s_i} \right)$ on $\Zvg[Y]$ is computed in Lemma~\ref{lem:explicitDLoperators}. 
\begin{nlem}
Let $\mv = \etav \da \sigma \tt(\betav)$ with $\sigma \in \leftidx^J W$ of minimal length and $\betav \in \tY$. Then the new basis satisfies  
\be{} \label{eq:vv:kappaY} \vv_{\mv} = \kappa(\mv) Y_{\mv} & \mbox{ where } & \kappa(\mv) = \prod \limits_{a \in \rts_+,  (a)\sigma^{-1} > 0 }  \, (\tau^{-1} \gf_{\la \etav+ \rhov, a \ra \Qs(\av)} ) \in \Zvg.
\ee
\end{nlem}
\begin{proof}
By definition we have that $\vv_{\etav} = Y_{\etav}$. 
If $\muv = \etav \da \sigma$ with $\sigma \in \leftidx^J W$ of minimal length, then we may use induction on the length of $\sigma$ and the argument follows from using the first case of Lemma \ref{lem:explicitDLoperators}(2) repeatedly.  
If $\mv = \etav \da \sigma \tt(\betav)$, then $\kappa(\mv)$ only depends on $\sigma$.
\end{proof}

\begin{nrem}\label{rem:vv:relations}
By construction, we find that if we replace the $Y_{\mv}$ in Lemma \ref{lem:explicitDLoperators} with $\vv_{\mv}$, one obtains the exact relations from the previous section,  namely \eqref{Tsi:Ymv-cases} becomes \eqref{rank1:quantumDL-mod} with $\vec_{\mv}$ replaced by $\vv_{\mv}$. In other words, all the Gauss sum contributions are absorbed into the $\vv_{\mv}$. Finally, let us comment that under the quantum specialization of \eqref{quantum-specialization}, $\kappa(\mv)$ is mapped to $1.$
\end{nrem}

\newcommand{\alphav}{\check{\alpha}}

\tpoint{Examples}
Consider the $\mathbf{PGL}_3$ root datum $\mf{D}$ with simple coroots $\alphav_1, \alphav_2$, simple roots $\alpha_1, \alpha_2$ such that $\la \alphav_i, \alpha_j \ra = -1 + 3 \delta_{ij}$ and $\rhov = \alphav_1+\alphav_2$. 
Let $\mf{D}_{(\Qs,n)}$ be the twist of $\mf{D}$ with $n=6$ and $\Qs(\alphav_1) = \Qs(\alphav_2)=\Qs$.

Consider $\etav = -2\alphav_1-2\alphav_2 \in \tnalc$. Then we have by using the first case of Lemma \ref{lem:explicitDLoperators}(2):
\be{}
\vv_{-2\alphav_1-2\alphav_2} = Y_{-2\alphav_1-2\alphav_2}, \quad
\vv_{-\alphav_1-2\alphav_2} = \tau^{-1}\gf_{ -\Qs} Y_{-\alphav_1-2\alphav_2}, \quad
\vv_{-2\alphav_1-\alphav_2} = \tau^{-1}\gf_{ -\Qs} Y_{-2\alphav_1-\alphav_2},\\
\vv_{-\alphav_1} = \tau^{-2}\gf_{ -\Qs} \gf_{ -2\Qs} Y_{-\alphav_1}, \quad
\vv_{-\alphav_2} = \tau^{-2}\gf_{ -\Qs} \gf_{ -2\Qs} Y_{-\alphav_2}, \quad
\vv_{0} = \tau^{-3}\gf_{ -\Qs}\gf_{ -2\Qs}\gf_{ -\Qs} Y_{0}.
\ee
These are all the elements $\vv_{\muv}$ with $\muv \in \etav \da W$. The rest of the elements $\vv_{\muv}$ in $\V(\etav)$ can be obtained by multiplying with elements $Y_{\betav}$ for $\betav \in \tY$.  
If instead we work with $\etav' = -3\alphav_1-2\alphav_2$, then we have 
\be{}
\vv_{-3\alphav_1-2\alphav_2} = Y_{-3\alphav_1-2\alphav_2}, \quad
\vv_{-2\alphav_2} = \tau^{-1}\gf_{-3\Qs} Y_{-2\alphav_2}, \quad
\vv_{2\alphav_1+3\alphav_2} = \tau^{-2}\gf_{-3\Qs}\gf_{-3\Qs} Y_{2\alphav_1+3\alphav_2},
\ee

\subsection{$\gf$-twisted canonical bases  }\label{sub:generalHsphVsph}

We now construct an involution on the space $\V$ and use it to study a version of Kazhdan--Lusztig theory on a spherical quotient $\Vsp$ of $\V.$

\tpoint{The involution on $\V(\etav)$} \label{subsub:involution-Vetav}  Let $\etav \in \tnalc$ and $J$ as in \eqref{def:J}. We would like to define an involution on $\V(\etav)$ which is compatible with the action of $\taffH$ and reduces to our previous involution on $\tV(\etav)$ under the quantum specialization. We begin with the following simple observation which follows  from Lemma \ref{lem:explicitDLoperators}(2) and the same ideas as in the remark concluding \S\ref{subsub:Involution-tV-eta}. For $\etav \in \tnalc,$ we have 
\be{} \label{Two:alc} \begin{array}{lcr}Y_{\etav} \cdot \Tmw_{w_0} =  \tau^{ \ell(w_0)} (- \tau)^{-\ell(w_0^J)}   \kappa(\etav \da w_0) Y_{\mv \da w_0}  & \mbox{where} &\kappa(\etav \da w_0) = \prod \limits_{\substack{a \in \rts_+, \\ (a) (w^J_0)^{-1} > 0 }}  \, (\tau^{-1} \gf_{\la \etav+ \rhov, a \ra \Qs(\av)} ), \end{array} \ee 
or equivalently in terms of the basis $\vv_{\mv}$ introduced in  \S\ref{subsub:vv-basis},  we may write 
\be{} \vv_{\mv} \cdot   H_{w_0} = (-\tau)^{- \ell(w_0^J)} \vv_{\mv \da w_0}. \ee

\noindent Inspired by \eqref{involution:Ts} and the computation in Remark~\ref{rem:involution:computation}, we now set for $\mv \in \etav \da \taffW$ either 
\be{} \label{involution:Veta}  d(Y_{\mv}) &:=&  \overline{Y_{\mv}} := \tau^{ \ell(w_0)} (- \tau)^{-\ell(w_0^J)}   \kappa(\mv \da w_0) \kappa(\mv ) Y_{\mv \da w_0}  (\Tmw_{w_0})^{-1}   \mbox{ or equivalently  } \\ d(\vv_{\mv}) &=&  (- \tau)^{-\ell(w_0^J)} \, \vv_{\mv \da w_0 } \cdot H_{w_0}^{-1} .  \ee 
The equivalence of the two expressions is immediate from~\eqref{eq:vv:kappaY} and the fact that $\overline{\kappa(\muv)}=\kappa(\muv)^{-1}$. 
If $\etav \in \tnalc$, then $\overline{Y_{\etav}} = Y_{\etav}$ follows from~\eqref{involution-in-alcove}.
 
\begin{nprop} \label{prop:involution-met} The formula (\ref{involution:Veta}) defines an involution on $\V(\etav)$ which satisfies 
\be{} \label{inv-compat-hecke} \begin{array}{lcr} d(v \cdot h )  = \overline{v} \cdot \overline{h} & \mbox{ for } & v \in \V(\etav), h \in \taffH \end{array}.\ee  
\end{nprop}  

\begin{proof}  
First suppose  $h= Y_{\zv}$ with $\zv \in \tY$ and $w= \vv_{\mv}$ for any $\mv \in \etav \da \taffW$. From \eqref{involution:Y}, we have $\overline{Y_{\zv} } = H_{w_0} Y_{\zv \cdot w_0} H_{w_0}^{-1}$, and we wish to show that 
$$
\begin{array}{lcr} 
d(\vv_{\mv} Y_{\zv})  = d( \vv_{\mv + \zv} )  =  (-\tau)^{-\ell(w_0^J)} \vv_{(\zv + \mv)\da w_0} H_{w_0}^{-1} & \text{ is equal to } & 
d(\vv_{\mv}) \overline{Y_{\zv}} =(-\tau)^{-\ell(w_0^J)} \, \vv_{\mv \da w_0} H_{w_0}^{-1} H_{w_0} Y_{\zv \cdot w_0} H_{w_0}^{-1}. 
\end{array}	
$$ 
Since $(\zv + \mv) \da w_0 = \mv \da w_0 + \zv \cdot w_0$, \eqref{inv-compat-hecke} follows for $h= Y_{\zv}$. 
Suppose now that $h = H_{s_i}, i \in I$ and $\mv = \etav \in \tnalc.$ It suffices to show show  \be{} d( \vv_{\etav} \cdot H_{s_i}) = d(\vv_{\etav}) \cdot \overline{H_{s_i}} = \vv_{\etav} \cdot H_{s_i}^{-1} . \ee Now $d(\vv_{\etav} \cdot H_{s_i})$ may be computed using Lemma \ref{lem:explicitDLoperators}(2) as follows
\be{} d( \vv_{\etav} \cdot H_{s_i} ) = \begin{cases} (-\tau)^{-\ell(w_0^J)}  \vv_{\etav \da s_i w_0 } H_{w_0}^{-1} & \text{ if } \la \etav + \rhov, a_i \ra \neq 0 \\ (-\tau)^{-\ell(w_0^J)}     (- \tau^{-1} ) \vv_{\etav \da w_0} H_{w_0}^{-1} & \text{ if } \la \etav + \rhov, a_i \ra = 0 \end{cases}. \ee 
The rest of the proof is an exercise in Coxeter combinatorics with the use of Lemma \ref{lem:explicitDLoperators} (2). 
Note that the involutive property of $d$ follows from \eqref{inv-compat-hecke} since each $v\ in \V(\etav)$ is of the form $Y_{\etav} \cdot h$ for some $h \in \taffH$. 
\end{proof}

\tpoint{On the space $\Vsp(\etav)$} \label{subsub:straightening_rules_Vsph}   
Recall the element $\epsilon$ from \S\ref{subsub:right-sppherical-module}.
For $\etav \in \talc$, we define the $\Zvg$-module 
\be{}\label{def:Vmet-sph} 
\Vsp(\etav):=  \V(\etav) \cdot \epsilon := \mathrm{Span}_{\Zvg} \{ \yket{\mv},   \mv \in \etav \da \taffW \}  =  \mathrm{Span}_{\Zvg} \{ \vket{\mv},   \mv \in \etav \da \taffW\}, 
\ee 
where $[v] :=v \cdot \epsilon \in \V $ for any $v \in \V(\etav)$. Here, the action of $\epsilon$ on $v$ is obtained by writing $\epsilon= \epsilon_I^+$ as in \eqref{symmetrizer} and identifying the action of $H_w, w \in W$ with that of the operator $\tau^{-\ell(w)}\Tmw_w$ from Lemma~\ref{lem:explicitDLoperators}. As with $[\vec_{\mv}]$, the elements $\yket{\mv}$ are not independent and satisfy the straightening rules below.
Since $\epsilon H_{s_i} = \tau \epsilon$,  \be{} \label{T:Yket} [Y_{\mv}] \cdot \Tmw_{s_i} = \tau^2 [Y_{\mv}] \mbox{ for } i \in I .\ee 
As an immediate consequence of Proposition \ref{prop:straightening-rules-Hecke} and the previous equation, we obtain the following.

\begin{nprop} \label{prop:straightening-sph} If $a=a_i, i\in I$, one has the following relations for elements $\yket{\lv}$ (resp. $\vket{\lv}$), $\lv \in \etav \da \taffW$: 
\begin{enumerate}
	\item  If $\la \lv +  \rhov, a \ra \geq 0$ and $\la \lv + \rhov, a \ra \equiv 0 \mod n(\av)$ then  \be{} \begin{array}{lcr} \label{straight:v:1} \yket{\lv} + \yket{\lv \da s_a } = 0 & \text{ or } &  \vket{\lv} + \vket{\lv \da s_a } = 0. \end{array} \ee
		
	\item If $0 <  \la \lv + \rhov , a \ra  < n(\av) $, then \be{} \begin{array}{lcr} \label{straight:v:2} \yket{\lv} = \gf_{\la \lv + \rhov, \av \ra \Qs(\av)} \yket{\lv \da s_a }  & \text{ or } & \vket{\lv} = \tau \vket{\lv \da s_a }. \end{array} \ee 
		
	\item  If $\la \lv + \rhov, a \ra > n(\av)$ , \,  $j:= \la \lv + \rhov, a \ra \not\equiv 0 \mod n(\av) $,  then setting $\lv_{(1)} := \lv - \resi_{n(\av)} (\la \lv +\rhov, a \ra) \, \av$
	\be{}  \label{straight:v:3} \begin{array}{lcr} \ket{\lv} - \gf_{j \, \Qs(\av)} \ket{\lv \da s_a } - \gf_{j \, \Qs(\av)} \ket{\lv^{(1)}} +  \ket{\lv^{(1)} \da s_a }  = 0 & \text{ or }  & \vket{\lv} - \tau \vket{\lv \da s_a } -\tau \vket{\lv^{(1)}} +  \vket{\lv^{(1)} \da s_a }  = 0. \end{array} \ee 	
\end{enumerate}
\end{nprop}

\tpoint{ On the space $\Vsp$ and its canonical bases} \label{subsub:Vsp-g} 
Fix $\etav \in \tnalc$. Using the straightening rules above, we may show that 

\begin{ncor} The collection $\{ \yket{\mv} \}$ for $\mv \in Y_{+} \cap \etav \da \taffW$ forms a basis of $\tVs(\etav)$, and similarly for $\{ \vket{\mv} \}$. \end{ncor} 

\noindent If we now define the $\Zvg$-module   
\be{} \label{eq:def:vsp:met}
\Vsp := \bigoplus_{\etav \in \tnalc} \, \Vsp(\etav), 
\ee 
then we find that, as an $\Zvg$-module, $\Vsp$ has a basis $\{ \yket{\lv} \}$ or $\{ \vket{\lv} \}$ where  $\lv$ ranges over $Y_+.$ The involution $d: \V(\etav) \rr \V(\etav)$ which was studied in \S\ref{subsub:involution-Vetav} induces an involution on the subspace $\Vsp(\etav)$ by using Proposition \ref{prop:involution-met} and the fact that $\overline{\epsilon}= \epsilon.$ We have  \be{} d(\vket{\mv}):= d ( \vv_{\mv} \epsilon ) = d(\vv_{\mv}) \epsilon \text{ for } \mv \in \etav \da \taffW . \ee 
Combining the arguments in Proposition \ref{prop:involution:quantum} with that of Proposition \ref{prop:involution-met}, we obtain the following formula for the involution $d: \V(\etav) \rr \V(\etav)$ restricted to the subspace $\Vsp(\etav)$, 
\be{} \label{eq:involution:Vetasph}  
d(\vket{\lv}) &=& (-1)^{\ell(w_0^J)} \tau^{-\ell(w_0)- \ell(w_0^J)}  \vket{\lv \da w_0} \mbox{ or equivalently,}\\ 
d(\yket{\lv}) &=& (-1)^{\ell(w_0^J)} \tau^{-\ell(w_0)- \ell(w_0^J)} \kappa(\lv) \kappa(\lv \da w_0) \yket{\lv \da w_0}.  
\ee 
The involution $d$ defined above naturally extends to one on $\Vsp.$

\begin{nprop} \label{prop:KL-basis-met} 
For each $\lv \in Y_+$, there exist a unique self-dual elements $ \lvecket{\lv}, \lvecketm{\lv} \in \Vsp(\etav)$ (here $\etav \in \talc$ is such that $\lv = \etav \da w$ for $w \in \taffW$) which satisfies the triangularity condition 
\be{} \label{eq:KL-basis-met}  \begin{array}{lcr} 
\lvecket{\lv} = \vket{\lv} + \sum_{\mv < \lv } a_{\mv, \lv} \vket{\mv}, \mbox{ for } a^+_{\mv, \lv} \in \At^+ & \text{and } & \lvecketm{\lv} = \vket{\lv} + \sum_{\mv < \lv } a_{\mv, \lv} \vket{\mv}, \mbox{ for } a^-_{\mv, \lv} \in \At^-. \end{array}
\ee  
\end{nprop}

\noindent  By the way they were constructed, one can verify that 
$a^{\pm}_{\mv, \lv} = o^{\pm}_{\mv, \lv}$
are the parabolic, singular Kazhdan--Lusztig polynomials from~\eqref{eq:lvectovec}.   
Writing \eqref{eq:KL-basis-met} in terms of the basis $\vket{\mv} = \kappa(\mv) \yket{\mv}$ we find relations 
\be{} \label{eq:KL-basis-met-quantum}   \begin{array}{lcr} \lvecket{\lv} = \kappa(\lv) \yket{\lv} + \sum_{\mv < \lv } \kappa(\mv) o^{+}_{\mv, \lv}(\tau^{-1}) \yket{\mv} 
	& \text{and} &
	 \lvecketm{\lv} = \kappa(\lv) \yket{\lv} + \sum_{\mv < \lv } \kappa(\mv) o^{-}_{\mv, \lv}(\tau^{-1}) \yket{\mv} \end{array} \ee 
 
\noindent Defining $\gket{\lv}, \gketm{\lv}$ to satisfy $\lvecket{\lv} = \kappa(\lv)\gket{\lv}$ and $\lvecketm{\lv}= \kappa(\lv) \gketm{\lv}$,  the previous equations become
\be{} \label{eq:KL-basis-met:G}   \begin{array}{lcr} 
\gket{\lv} = \yket{\lv} + \sum_{\mv < \lv } \kappa(\mv) \kappa(\lv)^{-1} o^{+}_{\mv, \lv}(\tau^{-1}) \yket{\mv} & \text{ and } & \gketm{\lv} = \yket{\lv} + \sum_{\mv < \lv } \kappa(\mv) \kappa(\lv)^{-1} o^{-}_{\mv, \lv}(\tau^{-1}) \yket{\mv}. \end{array}
\ee 
We shall refer to $ \lvecket{\lv} $ and $\lvecketm{\lv}$ as the canonical basis in $\Vsp(\etav)$ and $\gket{\lv}, \gketm{\lv}$ as the $\gf$-twisted canonical bases. We also set \be{o:gf}  \leftidx^{\gf}o^{\pm}_{\lv, \mv}:= \frac{\kappa(\mv)}{\kappa(\lv)} o^{\pm}_{\mv, \lv}(\tau^{\pm1}) \ee and refer to these as $\gf$-twisted parabolic, singular Kazhdan--Lusztig polynomials.

\subsection{$\gf$-twisted Littlewood--Richardson theory} \label{subsub:gTwistedLR}  

Recall the spherical subalgebra $\tspaff:= \epsilon \taffH \epsilon$ introduced in~\S\ref{subsub:sphericalHecke}. It is equipped with two bases, $\t{h}_{\lv}:= \epsilon Y_{\lv} \epsilon$ and $\tckl_{\lv}$, both indexed by $\lv \in \tY_+$.
Under the Satake isomorphism $S: \tspaff \rr \At[\tY]^W$ we have $S(h_{\lv})$ is given by the formula \eqref{HL-Satake} and $S(\t{c}_{\lv})= \chi_{\lv}$ the Weyl character, see \eqref{char:lv}. The $\Zvg$-modules $\Vsp(\etav)$ and $\Vsp$ can be equipped with a right $\tspaff$ action, denoted by $\star$, and defined as in \S \ref{subsub:VJsph}. We shall also write $\diamondsuit$ for the action of $\At[\tY]^W$  as in \eqref{def:diamond}, \textit{i.e. }
\be{} w \diamondsuit f = v \star S^{-1}(f) \text{ for } w \in \Vsp, f \in \tspaff. \ee

\tpoint{Action of $\t{h}_{\lv}$ }\label{subsub:actionHsphonVsph} Let $\yket{\mv}:= Y_{\mv} \epsilon  \in \Vsp(\etav)$ for $\mv \in Y_+ \cap  \etav \da \taffW$. 
From  \S\ref{subsub:straightening_rules_Vsph}
\be{} \label{yket:h-2} \begin{array}{lcr} \yket{\mv} \star h_{\lv} = Y_{\mv} \cdot  \epsilon Y_{\lv} \epsilon  =  [\wt{\cs}(\mv)\cdot Y_{\lv}]  & \text{ for } & \lv \in \tY_+, \mv \in Y_+, \end{array} \ee 
where we have written $\wt{\cs}(\mv):= Y_{\mv} \epsilon$. The elements $\wt{\cs}(\mv) \in \Vsp$ are precisely the values of the unramified spherical Whittaker function on metapletic covers (see~\cite{PPAIM} and the references therein). Alternatively, one may write the above expression in terms of the (Hall-Littlewood) polynomials $S(h_{\lv})$ introduced in \eqref{HL-Satake}, \be{} \yket{\mv} \star h_{\lv} = [ Y_{\mv} \cdot S(h_{\lv}) ], \ee where again one needs to use the straightening rules to rewrite the above in terms of the basis $\{ \yket{\mv} \}_{\mv \in Y_+}$. One may prove a version of Proposition \ref{prop:invariant-action} and  
\be{} \label{y:c:diamond} \yket{\lv} \star \tckl_{\zv} = \yket{\lv} \diamondsuit \chi_{\zv} = [ Y_{\lv} \cdot \chi_{\zv} ]. \ee

\tpoint{$\gf$-twisted Littlewood--Richardson polynomials} \label{subsub:gtwistedLR} 
If we alternatively expand the product $\yket{\mv} \star \ckl_{\lv}$ for $\mv, \lv$ as in the previous paragraph, we obtain an expression of the form 
\be{} \label{def:gLR} \yket{\lv} \star \ckl_{\zv} = \sum_{\mv \in Y_+} \, \leftidx^{\gf}Q^{\zv}_{\mv, \lv}  \ket{\mv} \text{ for some } \leftidx^{\gf}Q^{\zv}_{\mv, \lv} \in \Zvg.\ee 
We call the elements $\leftidx^{\gf}Q^{\zv}_{\mv, \lv}$ $\gf$-twisted Littlewood--Richardson polynomials and under the quantum specialization, one can easily show that $\mf{q}( \leftidx^{\gf}Q^{\zv}_{\mv, \lv}) = Q^{\zv}_{\mv, \lv}$, where the latter are the polynomials defined in~\eqref{LR-poly}. 
In fact, we can be even more precise. 
From Remark~\ref{rem:vv:relations} it follows that the elements $\vket{\lv}$ will satisfy equation~\eqref{LR-poly} on the nose, and by comparing that with~\eqref{def:gLR} and the use of~\eqref{eq:vv:kappaY} we get 
\be{} \label{eq:qQ:gfQ}
\leftidx^{\gf}Q^{\zv}_{\mv, \lv} = \kappa(\muv)\kappa(\lv)^{-1} Q^{\zv}_{\mv, \lv}.
\ee

\tpoint{$\gf$-twisted Tensor Product Theorems}\label{subsub:gtwisted-Steinberg-Lusztig} 
We now consider the action of $\tspaff$ on the new bases $\lvecket{\mv}, \lvecketm{\mv}$, $\mv \in Y_+$ of the space $\Vsp$ that were introduced in \S\ref{subsub:Vsp-g}. The action is entirely determined from the following $\gf$-twisted analogue of the tensor-product theorem \cite{leclerc:thibon}, \cite{laniniram}.

\begin{nprop} \label{prop:gTensorProd} 
Let $\lv \in Y_+$ be written uniquely as $\lv = \lv_0 + \zv$ with   $\lv_0 \in \Box_{(\Qs, n)}$ and $\zv \in \tY_+.$ 

\begin{enumerate} 
\item We have \label{tensor-product-quantum}    $\lvecket{\lv_0} \star \tckl_{\zv}= \lvecket{\lv_0+ \zv}$ 
\item Writing $\lvbar_0:= \lv_0 \cdot w_0 + 2 (\t{\rho}^{\vee} - \rhov)$  we have  \be{}\label{tilting:tensor:g}\lvecketm{\lvbar_0} \star \tckl_{\zv} = \lvecket{\lvbar_0 + \zv}. \ee 
\end{enumerate}
 
\end{nprop}

\noindent The proof given in Section \S \ref{sub:tensorproduct} carries over here, where one uses the properties of the involution $d$ from Proposition \ref{prop:involution-met} as well as the fact that the straightening relations (see Proposition \ref{prop:straightening-sph}) when viewed with respect to the $\vvket{\lv}$ basis look the same as their `quantum' counterparts from the previous section. 

\begin{nrem}\label{rem:gtensorproduct}
	The Proposition holds if we replace the bases  $\lvecket{\lv}$ and $\lvecketm{\lv}$ with the bases $\gket{\lv}$ and $\gketm{\lv}$, respectively.
	To see this, note that $\gket{\lv} = \kappa^{-1}(\lv)\lvecket{\lv}$ and $\gket{\lv_0} = \kappa^{-1}(\lv_0)\lvecket{\lv_0}$ by definition, and use the fact that $\kappa(\lv) =\kappa(\lv_0)$ because $\lv-\lv_0 = \etav \in \t{Y}_+$. The same argument works for the $\gketm{\lv}$ basis.
\end{nrem}

\part{$p$-adic models} \label{part:padic}

Throughout this part, we assume that $(I, \cdot, \mf{D})$ is a root datum with $\mf{D}= ( Y, \{ \av_i \}, X, \{ a_i \})$ assumed to be of simply-conneted type.  Let $(\Qs, n)$ be a twist as in \S\ref{subsub:MetStructure} and $(I, \circ_{\qsn}, \wt{\mf{D}})$ the corresponding twisted root datum as in  \S\ref{subsub:TwistedRootDatum} again assumed to be of simply-connected type. Write $\wt{\mf{D}}= (\wt{Y}, \{ \tav_i \}, \wt{X}, \{ \ta_i \}).$ Both simply-connected assumptions can be removed (for the purposes of this part) with known techniques, though the constructions become a bit more notationally involved.  Denote also by \be{} \label{affineWeylGroups:pt2} \affW:= \affW(I, \circ, \mf{D}) & \text{ and } & \taffW:= \affW(I, \circ_{\qsn}, \wt{\mf{D}}) \ee the affine Weyl groups attached to $\mf{D}$ and $\wt{\mf{D}}$. Again, by our assumptions both are Coxeter groups and we adopt the same terminology here as in \S\ref{sub:AffineWeylGroups}. On ocassion, we write $\As$ (resp. $\wt{\As}$) for the Cartan matrix attached to $(I, \cdot)$ (resp. $(I, \circ_{\qsn})$) and $\As_{\aff}$ and $\wt{\As}_{\aff}$ for their untwisted affinizations (see \S\ref{subsub:untwistedAffinization}).

\section{Recollections on $p$-adic groups, their covers, and associated Hecke algebras}

In this section, we recall some facts about (covers of) $p$-adic groups and their associated Hecke algebras.

\subsection{Groups attached to root datum}
\label{sub:groupsfromrootdata}
\newcommand{\bH}{\mathbf{H}}

 Throughout this section $F$ will denote an arbitrary field. Denote by $\bG:= \bG_{\mf{D}}$ the corresponding split, simple algebraic group of simply-connectd type whose points over any field $F$ will be denoted as $ \bG(F):= \bG_{\mf{D}}(F)$. We review a few details of this construction following \cite{tits:km}, and adopt the notation in \S\ref{subsub:rootsystemmetD} to describe the roots, weights, etc. associated to $(I, \cdot, \mf{D}).$
 
\tpoint{Torus $\bH$} \label{subsub:torus} The group $\bG$ comes equipped with a (split) torus $\bH$ such that $\bH(F) \cong \Hom(X, F^*).$ For any $s \in F$ and $\lv \in Y$ we write $s^{\lv} \in \bH(F)$ for the element which sends $\mu \in X$ to $s^{\la \lv, \mu \ra}.$ As $\mf{D}$ was assumed simply-connected, $Y$ has basis $\Piv$ and it follows that every element in $\bH(F)$ may be written uniquely as $\prod_{i \in I} s_i^{\av_i}$  with $s_i \in F^*.$ There is a natural action of $W:=W(I, \cdot)$ on $\bH(F)$ induced from the corresponding action of $W$ on $X$, and for $s \in W$ and $h \in \bH(F)$ we denote this action by $s(h) \in \bH(F).$

\tpoint{Unipotent subgroups} For each $a \in \rts$ there exists a subgroup $\bU_a \subset \bG$ and an isomorphism\footnote{To define $x_a$, a further choice is needed to fix some signs, and we refer to \cite{tits:km} for more details.} $x_a: F \rr \bU_a(F).$ Recall that a \emph{nilpotent} set of roots $\Psi \subset \rts$ is a subset such that if $a, b \in \Psi$ and $a+b \in \rts$, then $a+b \in \Psi$ as well. For a nilpotent set of the form $\{a, b \}$, we write $[a, b] = \left( \mathbb{N} a + \mathbb{N} b\right) \cap \rts$ and also  $]a, b[:= [a, b] \setminus \{ a, b \}.$ For each nilpotent pair $\{ a, b \}$, there exist constants $k(a, b; c)$  that can be computed explicitly from $\mf{D}$ with $c \in ]a, b[$ such that \be{} \label{steinberg:rel} \left[x_a(s), x_b(t) \right] = \prod \limits_{\substack{c = ma + nb \\ c \in ]a, b[  } } x_c(k(a, b; c) s^m t^n) \text{ for } a, b, c \in R_+. \ee For any nilpotent set $\Psi \subset \rts,$ there exists a subgroup $\bU_{\Psi} \subset \bG$ equipped with inclusions $\bU_a \rr \bU_{\Psi}$ and such that: i) $\bU_{\Psi}(\C)$ is the unipotent group corresponding to the nilpotent Lie algebra defined by $\Psi$; and ii) for any order on $\Psi,$ we have an isomorphism of schemes, $\prod_{a \in \Psi} \bU_a \rr \bU_{\Psi}.$ The sets $\rts_{\pm}$ are clearly nilpotent, and we write $\bU:= \bU_{\rts_+}$ and $\bU^-:= \bU_{\rts_-}$ for the corresponding `positive' and `negative' unipotent subgroups of $\bG$. Often we drop the $+$ from the positive case.

\tpoint{Some relations} For each $a \in R$, we define elements in $\bG(F)$ by the following formulas \be{wa:ha} \begin{array}{lcr} w_a(s):= x_a(s) x_{-a}(-s^{-1}) x_a(s) \text{ for } x \in F^* &  \text{ and } &  h_a(s) := w_a(s) w_a(1)^{-1} \text{ for } s \in F^*. \end{array} \ee Note that $w_a(s)^{-1} = w_a(-s).$Write for $ \dw_a:= w_a(-1).$ The following relations hold in $\bG(F)$: 
\begin{itemize} 
	\item $t x_a(s) t^{-1} = x_a(t(a)s)$ for $t \in \bH(F) = \Hom(X, F^*)$ and $s \in F$ and $a \in \Pi \subset X.$
	\item For $a \in \Pi$ and $t \in \bH(F)$, $\dw_a t \dw_a^{-1} =s_a(t)$ for $s_a \in W$ the simple reflection in $W$ corresponding to $a.$
	\item For $a \in \Pi$ and $z \in F^*,$ we have $h_a(z) = z^{\av},$ where the latter elements were described in \S\ref{subsub:torus}.
	\item For $a \in \Pi, b \in R$ and $x \in F$, we have $\dw_a x_b(s) \dw_a^{-1} = x_{w_a(b)}(\eta(a, b) s)$ for some $\eta(a, b) = \pm 1$ that can be determined explicitly from $\mf{D}$ (cf. \cite[Lemma 5.1(c)]{mat}).
\end{itemize} 

\tpoint{$N$ and the Weyl group} Denote by by $N:= \la w_a(s) \mid a \in \Pi, s \in F^* \ra$ and note the above relations imply $\bH(F) \subset N$. Then the map $W(\As) \rr N$ which sends $s_i \mapsto \dw_{a_i}$ induces an isomorphism of groups $W \rr N/ \bH(F).$ Note that the lifts $\dw_a$ satisfy the braid relations (see  \cite[Proposition 3]{tits:constants}) so that for $w \in W$ with reduced decomposition $w = s_{i_1} \cdots s_{i_k}$ we may unambiguously define $\dw:= \dw_{a_{i_1}} \cdots \dw_{a_{i_k}} \in \bG(F).$ An explicit set of relations which the generators satisfy is also known (see \cite{tits:tores}). 

\tpoint{BN-pairs}

\newcommand{\bB}{\mathbf{B}} Let $\bB$ be the subgroup generated by $\bU$ and $\bH.$ We know that $\bB(F) = \bU(F) \rtimes \bH(F).$ In an analogous way, we may define $\bB^-$ using $\bU^-.$ The tuple $(\bG(F), \bB(F), N, S)$ where $S:= \{ \dw_a , a \in \Pi \}$ satisfies the condition of a Tits system or BN-pair (\cite[Def.1, Chap IV, \S2.1]{bour456}) and hence one has Bruhat decompositions for the group $\bG(F)$. We can replace $\bB$ with $\bB^-$ and obtain a (conjugate) BN-pair.

\newcommand{\hmet}{\widetilde{h}}
\newcommand{\cmet}{\widetilde{c}}

\spoint \label{subsub:met-ext} Attached to the  chosen twist $\Qs$ as in \S\ref{subsub:MetStructure} and an abelian group $A$ with  bilinear Steinberg symbol $(\cdot, \cdot): F^* \times F^* \rr A$ as in \S\ref{notation:stein-symb},  one can construct a universal covering group $E$ fitting into an exact sequence 
$ \label{eq:universal:extension} 
0 \rr A \rr E \stackrel{p}{\rr} \bG(F) \rr 1 
.$
We follow the approach of Matsumoto \cite{mat}, but adapt the terminology of  \cite{deligne:brylinski} where the construction of covers for general reductive groups is given, and review a few aspects of this construction following \cite{PPDuke} which treats the Kac-Moody (but simply-connected) case in this same notation.

\tpoint{The torus $\tH$} There exists a subgroup $\wt{H} \subset E$ fitting into a sequence $0 \rr A \rr \wt{H} \rr \bH(F) \rr 1.$ More concretely, the group $\tH$ is generated by $A$ and symbols $\hmet_a(s)$ with $a \in \Pi$ and $s \in F^*$ and subject to the relations that $A \subset \tH$ is an abelian subgroup, and that for $\av, \bv \in \Piv$ and $s, t \in F^*:$ 
\be{} \label{tH-rels} \begin{array}{lcr} \hmet_a(s) \hmet_a(t) \hmet_a(st)^{-1} = (s, t)^{\Qs(\av)} &  \text{ and } &[ \hmet_a(s), \hmet_b(t) ] = (s, t)^{\Bs(\av, \bv)}. \end{array} \ee 

\newcommand{\tw}{\widetilde{w}}

\tpoint{Some relations in $E$} \label{subsub:gen-rel-met} Next, we note that there exists a homomorphism $\bU(F)^{\pm} \rr E$ which splits the map $p$ from \eqref{eq:universal:extension}. Abusing notation slightly, we shall henceforth write $x_a(s)$ for $a \in R$ and $s \in F$ to mean the corresponding element in $E$. Define elements $\tw_a(s):= x_a(s) x_{-a}(-s^{-1}) x_a(s) \in E$, and then set $\lambda_a:= \tw_a(-1).$ The elements $\lambda_a$ satisfy the braid relations and $p(\lambda_a) = \dw_a$ for each $a \in \Pi.$ Moreover, one has the following relations
\begin{itemize} 
	 \item $\hmet_a(s) x_b(t) \hmet_a(s)^{-1} = x_b(s^{ \la \av, b \ra} t)$,
	\item  $\tw_a(s) \tw_a(-1) = \hmet_a(s)$,
	\item $\lambda_a^{-1} \hmet_b(s) \lambda_a^{-1} = \hmet_b(s) \hmet_a(s^{- \la a, \bv \ra})$,
	\item $\lambda_a^{-1} \, x_b(s) \lambda_a = \dw_a^{-1} \, x_b(s) \dw_a$.
\end{itemize}

\newcommand{\tN}{\widetilde{N}}

\tpoint{Weyl groups} Define $\tN$ to be the subgroup generated by $\tH$ and the elements $\lambda_a$. Then the pair $(\tN, \tB)$ where $\tB := \tH \rtimes \bU(F)$ satisfies the axioms of a $BN$-pair. Replacing $\bU(F)$ with $\bU^-(F)$, we also obtain the BN-pair $(\tB^-, \tN)$. The Weyl group of $\tN / \tH$ is isomorphic to $W:= W(I, \cdot) \cong W(I, \circ_{\qsn})$.

\subsection{Structure of $p$-adic groups and their covers} \label{sub:padicgroups}

Let $\K$ be a non-archimedean local field as in \S\ref{subsub:notationpadic} and write $G:= \bG(\K)$ (recall the convention in \S\ref{notation:functor}). Let $n$ be a positive integer such that $(q, 2n)=1$, $\bmu_n$ the group of $n$-th roots of unity in $\K$ (which, by our assumption, is of cardinality $n$), and write $(\cdot, \cdot)$ for the Hilbert $n$-symbol as in \S\ref{notation:stein-symb}. We call $\tG$ the central extension of $G$ as in \S\ref{subsub:met-ext} specialized to these particular choices, see \cite[\S6]{PPDuke} for more details in the same notation as this paper. 

\tpoint{Affine Weyl group } \label{subsub:padicTorus}  Recall that we have an isomorphism $W \rr N / H$ sending $s_i \mapsto w_{a_i}(-1) H$ for $i \in \Pi$. We can extend this to a map from the affine Weyl group $\zeta: \affW  \rr N/ H_{\O}$ as follows: let $x \in \affW$ be written as $x = w \tt(\lv)$, where $w \in W$ and $\lv \in \rtlv$ is written as $\lv = \sum_{i \in I} b_i \av_i$ with $b_i \in \zee;$ define $\label{p:lv} \pi^{\lv}:= \prod_{i \in I } h_{a_i}(\pi^{m_i})$ and set $\zeta(x) = \zeta(w \tt(\lv) ) = \dw \, \pi^{\lv} \, H_{\O}.$ The map $\zeta$ is an isomorphism, and we often write, for $x \in \affW$, $\dot{x}:= \dw \pi^{\lv}.$

\tpoint{Compact and Iwahori Subgroups} \label{subsub:CompactIwahori} Let $K = \bG(\O)$ and let us continue to write $\omega$ for the reduction map $\bG(\O) \rr \bG(\kappa).$ Let us define the Iwahori subgroup $ \Iop= \omega^{-1} (\mathbf{B}^-(\kappa)).$ The integral torus and unipotent groups are defined as \be{} \begin{array}{lcr}  \label{integral-torus} H_{\O}:= \bH(F) \cap K = \bH(\O) & \text{and }  \label{integral:unipotent} U^{\pm}_{\O} := K \cap \bU^{\pm}(\K) = \bU(\O), \end{array} \ee where the second equality in each of the above expressions is verified using some simple representation theory. As $U^{\pm}(\O) \subset K$, we also define, using the map $\omega: U_{\O} \rr \bU(\kappa)$, the congruence subgroup $\label{U:pi} U^{\pm}_{\pi}:= \omega^{-1}(1). $ One can verify that $U_{\pi}$ is generated by elements $x_a(s)$ with $a \in R_+$ and $s \in \pi \O;$ in fact one has a unique factorization of any $n \in U_{\pi}$ as $ n = \prod_{a \in \rts\_+} x_a(s_a) $ for some fixed order of $\rts_+$ and where each $s_a \in \pi \O$. From (\cite[Thm 2.5]{iwa:mat}), we have the Iwahori-Matsumoto decompositions $\Iop = U^-_{\O} H_{\O} U_{\pi}.$ We can also obtain different (though equivalent) decompositions by permuting the order of the factors in the above. 
From \cite[Prop. 2.4]{iwa:mat}, we have $K = \sqcup_{w \in W} \Iop \dot{w} \Iop.$   From \cite[Prop 3.2]{iwa:mat}, for $x \in \affW$, we have  $\label{eqn:lengthIwahori} | \Iop \setminus \Iop \dot{x} \Iop | = q^{ \ell(x)},$ and so, with respect to the natural Haar measure on $K$ which is normalized to give $\Iop$ volume $1$, we have \be{} \label{volK:volI} \vol(K)=  \sum_{w \in W} q^{\ell(w)}. \ee

\tpoint{Some $p$-adic decompositions } \label{subsub:padicDecompositions}   The \emph{Iwasawa decomposition} for $G$ asserts that $G = B^- \,K$. We can refine this further to assert the decomposition into disjoint pieces \be{iwa:g}  G = \sqcup_{\lv \in Y} U^- \pi^{\lv} K. \ee In particular, to each $g \in G$ we may write $g = k \pi^{\lv} u$ with $\lv \in Y$ uniquely determined from $g$. Note that $k$ and $u$ are not unique, since we could replace $u$ with any element from $\pi^{-\lv} \bU(\O) \pi^{\lv}$.  The \emph{Cartan decomposition} asserts $G = \sqcup_{\lv \in Y_+} K \pi^{\lv} K.$  Combining the Cartan and Iwasawa decomposition with the ones from the previous paragraph, we obtain what we refer to as the \emph{Iwahori-Matsumoto decompositions} \be{} \label{IwahoriMatsumotoK} \label{g:im-iwa} G = \sqcup_{x \in \affW} \Iop \dot{x} \Iop = \sqcup_{x \in \affW} I^- \dot{x} I^- = \sqcup_{x \in \affW} U^- \, \dot{x} \Iop. \ee

\newcommand{\D}{\mf{D}}
\newcommand{\y}{\wt{y}}
\renewcommand{\x}{\wt{x}}

\newcommand{\rou}{\mathbf{\mu}}

\tpoint{The group $\tG$} \label{subsub:MetPadicGroup} As mentioned above, the group $\tG$ splits over $U$ and $U^-$ and we fix a given splitting as above and continue to adopt the same notation (i.e. $x_a(s)$, etc.) for elements in these unipotent groups.  We continue to denote the Weyl group of $\tG$ by $W$ and write, for $a \in \Pi$, $\tw_a:= \lambda_a$ for the elements as constructed in \S \ref{subsub:gen-rel-met}. As the $\lambda_a$ satisfy the braid relations, we may also define $\tw$ for any $w \in W$ in the natural way.

\tpoint{On $\tH$ and its abelian subgroups} \label{subsub:MetTorus}Fix an ordering on $I.$ As we are in the simply-connected case, every $t \in \tH$ may be written uniquely (with respect to the fixed order) as $t= \zeta \, \prod_{i \in \I} \hmet_{a_i}(s_i)$ with $s_i \in \K^*$ and $\zeta \in \rou_n$. For $\lv \in Y,$ written in terms of the $\zee$-basis $\{ \av_i (i \in I) \}$ as $\lv = \sum_{i \in I} c_i \av_i$, we define the elements in $\tH$, $ \pi^{\lv}:= \prod_{i \in \I} \hmet_{a_i}(\pi^{c_i}).$ As $(\pi, \pi)=1$ (see end of \S \ref{notation:stein-symb}) we may use \eqref{tH-rels} to  see that the above element does not depend on the choice of ordering on $I.$ Define $\tH_{\O}$ to be the subgroup generated by the elements $h_{a_i}(s)$ with $s \in \O^*,$ $i \in I.$ Under the assumption that $q \equiv 1\mod 2n$, $\tH_{\O}$ is in fact an abelian group and $ \tH_{\O} \setminus \tH \simeq \rou_n \times Y.$

\tpoint{Iwahori and compact subgroups of $\tG$} \label{subsub:IwahoriCompactMet} Recall the subgroups $K:= \bG(\O)$ and $\Iop \subset K$ constructed in \S\ref{subsub:CompactIwahori}. Let $\tI^- \subset \tG$ be the subgroup generated by the following elements: 
\be{Ip:gen}\begin{array}{lcr}  \xi_a(s), s \in \O, a \in \rts_-, & \xi_{-a}(t), \, a \in \rts_+, \, t \in \pi \O & \text{ and } \quad \t{h} \in \tH_{\O}. \end{array} \ee  
Writing $\tI^-_+:= \tI \cap U$ and $\tI^-_-:= \tI \cap U^-$, one can show as in \cite[Theorem 2.5]{iwa:mat} that  $ \tI^- = \tI^-_+ \, \tH_{\O} \, \tI^-_-.  $ Moreover one knows \cite[Lemma 6.1.3]{PPDuke} that the covering map $p$ satisfies  $p|_{\tI}: \tI^- \rr \Iop$ is an isomorphism.  For this reason, we continue to write $\Iop$ in place of $\tI^-$ and will let context dictate what we mean.

Let $\tK \subset \tG$ be the subgroup generated by elements $\xi_a(s)$ with $s \in \O,$ $a \in \rts$. For $w \in W$, the elements $\tw$ constructed earlier lie in $K,$ and one has $\tK = \bigsqcup_{w \in W} \tI \, \tw \, \tI. $ Moreover, $p(\tI \, \tw \, \tI) =I \, \dw \, I$. Thus if $x \in \tK$ is such that $p(x)=1$, then $x \in \tI$ and, from what we said earlier about $p|_{\tI}$ we have $x=1$, i.e. $p: \tK \rr K$ is an isomorphism. From now on, we just identify $K$ with $\tK$ using $p$, and shall drop the notation $\tK.$

\tpoint{Decompositions of $\tG$} \label{subsub:tG-dec}  Recall the decompositions from \S \ref{subsub:padicDecompositions} and let us explain the corresponding versions for $\tG$ now. The Cartan decomposition for $\tG$ asserts that $\tG = \sqcup_{\lv \in Y_+} \bmu_n K \pi^{\lv} K $ and easily follows from the corresponding result for $\tG.$ The analogue of the Iwasawa decomposition for $\tG$ is as follows. Every $g \in \tG$ can be written as $g = u^- a k$ with $k \in \tK, \, u^- \in \tU$, and $a \in \tH.$  The class of $a \in  \tH / \tH_{\O} $ is unique in any such decomposition.  For $g \in \tG$ written (non-uniquely) as $g = u^- a k$, we denote the class of $a$ in $\tH / \tH_{\O} $ by $\iw_{\tH}(g).$ Use what was said at the end of \S\ref{subsub:MetTorus}  to set 
\be{log:cen:g} \begin{array}{lcr} \ln (g):= \ln(\iw_{\tH}(g)) \in Y & \text{ and } & \z(g) :=  \z(\iw_{\tA}(g)) \in \rou_n, \end{array} \ee 
so that image of $\iw_{\tA}(g)$ in $Y \times \rou_n$ is $(\ln(g), \z(g))$.

\newcommand{\dx}{\dot{x}}
\newcommand{\dy}{\dot{y}}
\newcommand{\dz}{\dot{z}}
\newcommand{\du}{\dot{u}}
\renewcommand{\i}{I^-}
\renewcommand{\affI}{I_{\mathrm{aff}}}

\subsection{Hecke algebras}  \label{sub:Heckealgebras}

We now review an algebraic (as opposed to measure theoretic) framework to associate convolution Hecke algebras on $p$-adic groups. We refer to \cite[\S4]{bkp} for more details on this approach, which is equivalent to the usual one.

\tpoint{Iwahori--Hecke algebras }  \label{subsub:hecke-rel-IM} For $x \in \affW$, let $\T_x$ denote the characteristic function of $\i \dx \i$. Let  $\hec(G, \i)$ denote the space of $\zee$- linear combinations of the form $\sum_{x \in \affW} n_x \T_x $ with almost all $n_x=0$. Equivalently, we are looking at $\zee$-valued, compactly supported $\Iop$-binvariant functions on $G$. This $\zee$-module is equipped with the structure of an associative, unital algebra under convolution defined as follows. For $x, y \in \affW$, let $m_{x, y}: \i \dx \i \times_{\i} \i \dy \i \rr G$ and define (see \cite[p. 44]{iwa:mat}) \be{iwahori-conv} \T_x \star \T_y := \sum_{z \in \affW} n_z \T_z, \text{ where } n_z:= | m^{-1} (\dz) |  = | \i \setminus \i \dx^{-1} \i z \cap \i \dy \i |. \ee 

\begin{nthm}\cite[Theorem 3.5]{iwa:mat} \label{thm:IMpresentation} 
Writing $\affH$ for the affine Hecke algebra attached to $(I, \cdot, \mf{D})$ (see \S \ref{sub:aha-comb}), the map   $\affH  \rr \zee[q^{\pm1/2}] \otimes \hec(G, I^-)$  sending $T_{s_i} \mapsto \T_{s_i}$ for $i \in I_{\aff}$ is an isomorphism of algebras when $\tau \mapsto q^{1/2}$. 
\end{nthm} 

\newcommand{\Hs}{\mathscr{H}}
\newcommand{\Ys}{\mathscr{Y}}
\noindent Using this result, we construct elements in $\hec(G, \Iop)$ that we again denote (by abuse of notation) as $H_{w}$ and $Y_{\lv}$ for $w \in W, \lv \in Y$ which correspond to similarly named elements in $\affH$.  

\newcommand{\Gdual}{\check{\bG}(\C)}
\newcommand{\hall}{\mathsf{H}}
\newcommand{\ek}{\mathbf{e}_K}

\tpoint{Spherical Hecke algebra} \label{subsub:SpherticalHeckeAlgebra} The \emph{spherical Hecke algebra} $\hec(G, K)$ consists of finite linear combinations of the characteristic functions $h_{\lv}, \lv \in Y_+$ of the double cosets $K \pi^{\lv} K$ and multiplication defined as follows: for $\lv, \mv \in Y_+$ and $ m: K \pi^{\lv} K \times_K K \pi^{\mv} K \rr G,$ we set 
\be{sph:conv} h_{\lv} \star h_{\mv} = \sum_{\etav \in Y_+}  \, | m^{-1}(\pi^{\etav}) | h_{\etav} = \sum_{\etav \in Y_+}  \, | K \setminus K \pi^{- \lv} K \pi^{\etav} \cap K \pi^{\mv} K  | \, h_{\etav}. \ee 
Justifying the abuse of notation (since $h_{\mv}$ were already defined in \S\ref{subsub:sphericalHecke}), we have the following result.

\begin{nthm} \label{thm:spherical-algebraic} 
	The map from the spherical subalgebra $\spaff:= \epsilon \, \affH \epsilon \subset \affH$ to $\zee[q^{\pm1/2}] \otimes \hec(G, K)$  sending $\epsilon \, Y_{\mv} \epsilon$ to $h_{\mv}$ for $\mv \in Y_+$ is an isomorphism of algebras when $\tau \mapsto q^{1/2}$. 
\end{nthm}

\tpoint{Relating $\hec(G, \Iop)$ and $\hec(G, K)$} To relate $\hec(G, \Iop)$ with $\hec(G, K)$ we need to take into account the fact that convolution in the spherical Hecke algebra is with respect to a measure which assigns $K$ volume $1$, whereas in $\hec(G, \Iop)$, the convolution is with respect to a measure that assigns $\Iop$ volume $1$. Recalling equation \eqref{volK:volI},  let $\ek:= \frac{\sph}{\vol(K)}$ where $\sph$ is the characteristic function on $G$ of the subgroup $K$. One may define a left and right convolution of $\hec(G, \Iop)$ with the function $\ek$ which then produces an element in $\hec(G, K)$, and we find that the map 
$\hec(G, \Iop) \rr \hec(G, K), f \mapsto \ek \star f \star \ek$ fits into the commutative diagram 
\be{} \label{algebraic-hecke-1}  
\begin{tikzcd}[column sep = 6em, row sep = 2em]
	\hec(G, \Iop) \arrow[r, "\ek \star\,  \cdot  \, \star \ek" ] \arrow[d, "\cong"'] & \hec(G, K) \arrow[d, "\cong"]  \\ 
	\affH \arrow[r, "\epsilon \, \cdot \, \cdot \, \epsilon" ] & \spaff  		
\end{tikzcd} 
\ee

\noindent where the bottom row is the map $h \mapsto \epsilon h \epsilon, h \in \affH$ and the vertical maps are from Theorems \ref{thm:IMpresentation} and \ref{thm:spherical-algebraic}.

\newcommand{\ep}{\varepsilon}

\tpoint{$\epsilon$-genuine functions} \label{subsub:epsilongenuine} Recall that $\bmu_n \subset \K$ is assumed to have cardinality $n$. Fix a faithful embedding $\epsilon: \rou_n \hookrightarrow \C^*$ which we use to define the Gauss sums as in \S\ref{subsub:GaussSums}. A function $f: \tG \rr \C$ is $\ep$-genuine if \be{} \label{def:epsilon-genuine} f(\zeta \, g) = \epsilon(\zeta) f(g) \text{ for } g \in \tG, \zeta \in \rou_n. \ee Let  $\C_{\ep}(\tG)$ denote the space of compactly supported $\ep$-genuine functions on $f: \tG \rr \C.$

\tpoint{Metaplectic spherical Hecke algebra}\label{subsub:metaplecticsphericalHecke} Define the spherical Hecke algebra 
\be{} \label{def:sphHeckemet} 
\thsp := \{f \in \C_{ \ep}(\tG) \mid f(k_1 g k_2) = f(g) \text{ for } k_1, k_2 \in K\}. \ee 
The Cartan decomposition of $\tG$ (see \S \ref{subsub:tG-dec}) ensures that  the functions, defined for $\lv \in Y_+,$ by setting 
\be{} \label{def:hlambda-met} 
\hmet_{\lv}(x) = 
\begin{cases} 
\epsilon(\z(x) ) & \text{if } x \in \rou_n K \pi^{\lv} K, \\ 0 & \text{otherwise}, 
\end{cases} 
\ee 
span the $\C$-vector space $ H_{\epsilon}(\tG, K). $ In fact, one can also verify that $\hmet_{\lv}$ if well-defined only if $\lv \in \tY$ since if $\lv \notin \tY$ there exists $h_{\O} \in \tH_{\O}$ such that $[h_{\O}, \pi^{\lv}] \neq 1.$ 
The collection $\{ \hmet_{\lv} \}_{\lv \in \tY_+ }$ forms a basis of $\thsp$ as a vector space. 
The vector space $\thsp$ is also equipped with a convolution structure, which is defined by considering the multiplication map $m_{\lv, \nv}:  \bmu_n K \pi^{\lv} K \times_K  K \pi^{\nv} K \rr \tG.$ For $\mv \in \tY_+$ and $x \in m_{\lv, \nv}^{-1}(\pi^{\mv})$, written as $x= (a, b)$ with $a \in \bmu_n \ ,K \pi^{\lv}K$ and $b \in K \pi^{\nv}K$, we set, in the terminology of \S \ref{subsub:tG-dec}, $\z(x) := \z(a) \in \bmu_n$ as it only depends on $x$. Then we define the convolution product as 
\be{} \label{mult:spherical-hecke-met}
\hmet_{\lv} \star \hmet_{\nv} = \sum_{\mv \in Y}  \hmet_{\mv} \left( \sum_{x \in m_{\lv, \nv}^{-1}(\pi^{\mv})}   \epsilon(\z(x))  \right).
\ee
\newcommand{\GdualQs}{\check{\bG}_{(\Qs,n)}(\C)}

\noindent The metaplectic Satake isomorphism (see \cite{savin:localshimura,savin:crelle}, \cite[\S11]{mcnamara:ps}) gives an isomorphism of algebras 
\be{} \label{met-satake} \wt{S}: \thsp \stackrel{\sim}{\longrightarrow}  \C[\tY]^W \cong K_0\left( \Rep(\GdualQs) \right), \ee 
where $\GdualQs$ is the complex group attached to $\wt{\mf{D}}^{\vee}:= \left(\mf{D}_{\qsn} \right)^{\vee}.$ So, for each $\lv \in \tY_+$, we can define elements $\tckl_{\lv} \in \thsp$ such that $\t{S}(\tckl_{\lv}) = \chi_{\lv}$, i.e. $\tckl_{\lv}$ corresponds to the class of the representation $V_{\lv}$ of highest weight $\lv$ of the group $\GdualQs.$ For $\lv, \mv, \zv \in \tY_+$, we have the multiplication rule 
\be{eq:Shimura:LR:coefficients}
\tckl_{\lv} \star \tckl_{\mv} = \sum_{\zv \in \tY_+} \dim \Hom_{\GdualQs}(V_{\lv} \otimes V_{\mv}, V_{\zv}) \,  \tckl_{\zv}.
\ee

\newcommand{\tTh}{\wt{\vartheta}}
\newcommand{\trho}{\wt{\rho}}

\tpoint{Metaplectic Iwahori--Hecke algebras} \label{subsub:metIwahorHecke} We define $\thiw$ to be the space of functions $f \in \C_{\epsilon}(\tG)$ which are also $\Iop$ bi-invariant. It has a convolution structure that can be defined starting from the multiplication map $m_{x, y}: \Iop \dx \Iop \times_{\Iop} \Iop \dy \Iop \tG$ and using a procedure as in \S\ref{subsub:hecke-rel-IM}, but taking into account the $\epsilon$-genuine condition as in the previous paragraph. As an algebra, the structure of this Hecke algebra was first described by Savin in \cite{savin:localshimura}, \cite[\S6]{savin:crelle}. In our notation, his results state
\be{met-iwahori} \thiw \cong \C \otimes_{\At} \taffH:= \C \otimes_{\At} \affH(I, \circ_{(\Qs, n)}, \wt{\mf{D}} ) \cong H_W \otimes \C[\tY], \ee where the map $\At \rr \C$ sends $\tau \mapsto q^{-1/2}$. We shall write $Y_{\lv}, \lv \in \tY$ and $H_w, T_w, w \in W$ for the elements in $\thiw$ corresponding to similarly named elements in $\taffH.$ To relate $\thiw$ and $\thsp$ one has a diagram analogous to \eqref{algebraic-hecke-1}, where we note that $\ek$ is again defined as $\sph/ \vol(K)$ where $\vol(K)$ stays the same in the metaplectic context.

\section{Spherical and Iwahori--Whittaker functions  on covers of $p$-adic groups } 

Fix the same conventions as at the start of Part~\ref{part:padic} and notations as in \S \ref{sub:padicgroups}-\ref{sub:Heckealgebras}. 

\subsection{Whittaker spaces for covering groups}

\tpoint{Character of unipotents} \label{subsub:charactersUnipotents} Fix an additive character $\psi: \K\rr \C^*$ as in \S\ref{subsub:additive-characters}. For $a \in \Pi,$ using the chosen isomorphism $x_a: \K \rr \bU_a(\K), s \mapsto x_a(s)$, we regard $\psi$ first as a character of $U_a := \bU_a(\K) $, and then, via the isomorphism $U \rr U/[U, U] \cong \oplus_{a \in \Pi} U_a,$ as a character, to be denoted by the same name, $\psi: U \rr \C^*.$ One can easily verify that $\psi|_{U(\O)}$ is trivial, and $\psi|_{U(\pi^{-1}\O)}$ is non-trivial. As we identify the unipotent group $U$ with its preimage in the cover $\tG$, we can regard $\psi$ as a character of $U \subset \tG$ as well.

\tpoint{Spherical Whittaker space $\whitk$} \label{W:psi-1}  An $\epsilon$-genuine  function $f: \tG \rr \C$ is said to be $(U, \psi)$ left invariant if $f(u g) = \psi(u) f(g)$ for any $u \in U$ and $g \in G$. Denote by $\whitk$ the space of $\epsilon$-genuine functions on $\tG$ which are right $K$-invariant, left $(U, \psi)$-invariant and supported on a finite union of sets of the form $\rou_nU \pi^{\mv} K$ with $\mv \in Y.$ For $f \in \whitk,$ if we define $\Supp(f) \subset Y$ as the set of $\mv \in Y$ such that $f(\pi^{\mv})\neq 0,$ then one can show that $\Supp(f) \subset Y_+$ for any $f \in \whitk$.   For $\mv \in Y$, we define $\tJ_{\mv} \in \whitk$ to be the unique $\epsilon$-genuine function that takes value $1$ on $\pi^{\mv}$ and is $0$ outside of $\bmu_nU \pi^{\mv} \tK.$  Using the Iwasawa decomposition from \S \ref{subsub:tG-dec}, one shows that the set $\{ \tJ_{\mv} \}_{\mv \in Y_+}$ is a basis of $\whitk$.

\tpoint{Convolution action of $\thsp$ on $\whitk$} 
One has a right convolution action, denoted again by $\star$, of $\thsp$ on the space $\whitk.$ Let us recall here the definition of $\tJ_{\mv} \star h_{\lv}$ with $\mv \in Y_+, \lv \in \tY_+$. As we observed earlier, the support of such a function is a linear combination of $\tJ_{\zv}$ with $\zv \in Y_+$. For such a $\zv \in Y_+$, consider the multiplication map $ m_{\mv, \lv}: \bmu_n U \pi^{\mv} K \times_K K \pi^{\lv} K \rr \tG$ and 
let $(a, b) \in m_{\mv, \lv}^{-1}(\pi^{\zv})$ with $a  \in \bmu_n \, U \pi^{\mv} K$ and $b \in K \pi^{\lv} K$. Suppose we write $a = \omega u \pi^{\mv} k$ for $\omega \in \bmu_n, u \in U, k \in K$. 
Write $\omega = \mathbf{z}(x)$ as one sees easily that it depends only on $x$. On the other hand, $u$ is not well-defined as we can replace it by $u u_1$ with $\pi^{-\mv} u_1 \pi^{\mv} \in K$, i.e. $n_1 \in \pi^{\mv} U_{\O} \pi^{-\mv}.$ 
However, since $\mv \in Y_+$ and $\psi$ is trivial on $U_{\O}$, we have  $\psi( u u_1) = \psi(u) \psi(u_1) = \psi(n). $ 
In sum, if for any $x=(a, b ) \in m_{\mv, \lv}^{-1}(\pi^{\zv})$ where $a = \z(x) \, u \pi^{\mv} k$ we set $\bn(x)= u$, then $\psi(\bn(x))$ only depends on $x$. Using this notation, we define  
\be{} \label{def:conv:J:h}  
\tJ_{\mv} \star_K \t{h}_{\lv} := \sum_{\zv \in Y_+} q^{ \la \rho, \zv \ra } \left(  \sum_{ x \in m_{\mv, \lv} ^{-1}(\pi^{\zv}) } \epsilon(\z(x) ) \psi(\bn(x)) \right) \, \J_{\zv}, 
\ee 
where we note that $\zv \leq \lv + \mv$ in order for $b^{\lv}_{\mv}( \zv):= q^{ \la \rho, \zv \ra } \left(  \sum_{ x \in m_{\mv, \lv}^{-1}(\pi^{\zv}) } \psi(x) \epsilon(\z(x) )\right) \neq 0.$ 
\begin{nrem} Let $du$ be the usual Haar measure on $U$ which assigns $U(\O)$ volume $1$. Then we can also write 
	\be{b:int} b^{\lv}_{\mv}({\zv}) &=& q^{ \la \rho, \zv \ra } \, \int_{U} \psi(u) h_{\lv}(\pi^{-\mv + \zv} u ) du. \ee
\end{nrem}

\tpoint{Iwahori--Whittaker spaces } \label{subsub:whit-iwa} We  define $\whitiw$ by replacing $K$ with $\Iop$ in the definitions given in \S\ref{W:psi-1}. In light of the Iwahori-Matsumoto decomposition \eqref{g:im-iwa}, we may regard functions in $\whitiw$ as $(U, \psi)$-left invariant, right $\Iop$-invariant, $\epsilon$-genuine functions which are supported on a finite union of sets of the form $\bmu_n \, U \, \dot{x} \, \Iop$ with $x \in \affW:=\affW(I, \cdot, \mf{D})$. For each $x \in \affW$, let $\vbi_{\psi, x}$ denote the unique function (if it exists)  in $\whitiw$ which is supported on the subset $\bmu_n \, U \dx \Iop,$ and taking the value $1$ on $\dx$. If $x = \tt(\mv)$ with $\mv \in Y$, we shall just write  $\vbi_{\psi, \mv}$ in place $\vbi_{\psi, \tt(\mv)}$ from now on.  Note that $\vbi_{\psi, x}$ with $x \in \affW$ is only well-defined when $ \psi\mid_{ x \Iop x^{-1} \cap U} = 1.$  The vector space $\whitiw$ carries a right action, denoted again by $\star$, under convolution by $\thiw$. One can define its action in a similar manner to the one in which we defined the action of $\thsp$ on $\whitk$ (note that the modular character does not appear though).

\tpoint{Averaging operator $\Av^{\gen}_{U^-}$} Define the operator $\Avg_{U^-}: \whitiw \rr \C[Y]$ by linearly extending the following procedure.   For $x \in \affW$ , let  $f = \vbi_{\psi, x}$. Note that $f$ is only well-defined if $\psi_{x \i x^{-1} \cap U } = 1$.

 Using the map $ m=m_x: \bmu_n \, U \dx \Iop \times \Iop U^- \rr \tG,$ for $z \in \tG$ and $y= (a, b) \in m^{-1}(z)$ with $a= \z(y) \,  u x i$ for $\z(y ) \in \bmu_n, u \in U,$ and $i \in \Iop, $ we may set $\psi(y) = \psi(u).$  By our assumption on $x$,  $\psi(y)$ is well-defined. Set 
\be{avg:u-def} \Avg_{U^-}(\vbi_{\psi, x})   = \sum_{\mv \in Y}    \left( \sum_{y \in m_x^{-1}(\pi^{\mv})} \epsilon(\z(y)) \, \psi(y) \right) Y_{\mv} \, q^{ \la \rho, \mv \ra}. 
\ee 
For example, since $U \pi^{\lv} \Iop \cap U^- \pi^{\lv} \Iop = \pi^{\lv} \Iop$ we have   \be{} \Avg_{U^-}(q^{ - \la \rho, \lv \ra }\vbi_{\psi, \lv}) = Y_{\lv} \text{ for } \lv \in Y_+.  \ee  
	
\begin{nrem}	
In the usual integral notation, one writes 
\be{} \Avg_{U^-}(f) = \sum_{\zv \in Y} q^{ \la \rho, \zv \ra }\left( \int_{U^-} f(\pi^{\zv} u^- ) du^- \right) Y_{\zv}, 
\ee 
where $du^-$ is the Haar measure normalized so that $\bU^-(\O) = K \cap U^-$ has total volume $1$. 
\end{nrem} 

\spoint We shall need the following result which can be deduced using the same arguments as in \cite{chan:savin}.
	
\begin{nprop} \label{prop:averaging-isom} The map $\Avg_{U^-}$ induces an isomorphism of vector spaces 
\be{} \label{whitiw:Y-isom} \Avg_{U^-}: \whitiw \stackrel{\cong}{\longrightarrow} \C[Y], \ee 
which satisfies the following equivariance condition: 
\be{} \label{average:equivariance}  
\Avg_{U^-} ( w \star  Y_{\lv} ) = \Avg_{U^-}(w) Y_{\lv} \mbox { for } w \in \whitiw, \lv \in \tY, 
\ee 
where the action on the right hand side is just by multiplication in the group algebra. \end{nprop}

\noindent One can prove the Proposition in a different manner than that suggested in \emph{op. cit.}, proceeding as follows: \begin{itemize}
	\item 	We may write $\Avg_{U^-}(\vbi_{\psi, x}) = \sum_{\mv \in S_x} c_{\mv} Y_{\mv} $ with $c_{\mv}\in \C$ and  $S_x:= \{ \mv  \in Y \mid \bmu_n \, U \dx \Iop \cap \bmu_n U^- \pi^{\mv} \Iop \neq \emptyset\}.$ One may introduce a natural order on $Y$ such that  $\Avg_{U^-}$ is upper triangular with non-zero diagonal coefficients. This suffices to prove that the map is an isomorphism. 
	\item To prove the equivariance, one uses that  $Y_{\lv} \, \star \mathbf{1}_{\bmu_n I U^-} = \mathbf{1}_{\bmu_n \Iop \pi^{\lv} U^-}$ for $\lv \in \tY,$ where for any $\nv \in \tY$, we write $\mathbf{1}_{\bmu_n \Iop \, \pi^{\nv} U^-} \in \C_{\epsilon}(G)$ to be the unique function supported on $\bmu_n \Iop \, \pi^{\nv} U^-$ and taking value $1$ on $\pi^{\nv}.$   
  \end{itemize}

\subsection{Structure of the Iwahori--Whittaker space $\whitiw$} \label{sub:IwahoriAction} 

\renewcommand{\x}{\wt{\mathscr{Y}}}
\newcommand{\hv}{\wt{\mathscr{V}}}
\renewcommand{\T}{\wt{\mathscr{T}}}

Recall that we constructed an action of $\taffH$ on $\V = \Zvg[Y]$ in \S\ref{sub:met-poly}  by means of the  operators $\Tmw_{s_i}$ for $i \in I$, see \S\ref{subsub:metDL}. Using the $p$-adic specialization map $\mf{p}: \Zvg \rr \C$ defined as in \ref{notation:padic-spec} and the isomorphism $\C \otimes_{\zee} \thiw \cong \C \otimes_{\At} \taffH$,  we obtain an action of $\thiw$ on $\C[Y] = \C \otimes_{\Zvg} \Zvg[Y].$

\spoint \label{subsub:met:iw:whit:basis}  Let $\x_{\lv} \in \whitiw$ be defined uniquely through the relation \be{} \label{def:xlv} \Avg_{U^-}(\x_{\lv}) =  Y_{\lv} \mbox{ for any } \lv \in Y. \ee 

\begin{nprop} \label{thm:i-basis} For any simple root $a \in \Pi$ and $\lv \in Y,$ we have  \be{avg-x} \Avg_{U^-}( \x_{\lv} \star \T_a ) = Y_{\lv} \cdot \Tmw_a.\ee Hence we have an isomorphism of $\C$-vector spaces that we denote as  \be{} \label{Avg:whit-V-isom} \mf{p}: \C \otimes_{\Zvg} \V  \stackrel{\simeq}{\longrightarrow} \whitiw  \ee sending $Y_{\lv} \mapsto \x_{\lv} $ and intertwining the actions of $\taffH$ and $\thiw.$  
\end{nprop}

\tpoint{Proof of Proposition \ref{thm:i-basis}, part 1}  Suppose that $\lv \in Y_+$ and consider the element $\vbi_{\psi, \lv}$ constructed above which we showed satisfies  $\Avg_{U^-}(q^{ - \la \rho, \lv \ra }\vbi_{\lv}) = Y_{\lv}.$ Hence $\x_{\lv} =q^{ - \la \rho, \lv \ra }  \vbi_{\psi, \lv}$ for $\lv \in Y_+$. We claim \be{}\Avg_{U^-}( q^{ - \la \rho, \lv \ra }\vbi_{\lv}  \star \tT_{s_a} ) = Y_{\lv} \cdot  \Tmw_{a} \text{ when } \lv \in Y_+\ee
follows from the arguments in \cite[\S5.7]{PPAIM}. Since the setup in \emph{op. cit.} is a bit different than the present context (in particular what we write as $\Tmw_{s_i}$ here is $\Tmw_{s_i}^{-1}$ in \emph{op. cit.}), let us sketch the key ideas in the computation here.  First, to determine the support of  $\Avg_{U^-}( \vbi_{\lv} \star \T_{s_a} )$, we need to determine for which $\mv \in Y$  
\be{} 
\emptyset \neq \bmu_n U \pi^{\lv} I^- \tw_a I^- U^- \cap \, \bmu_n U \pi^{\mv} U^- = \bmu_n \, U \pi^{\lv} U_{\pi} \, U_{-a, \O}  \tw_a U^- \cap \bmu_n \, U \pi^{\mv} U^- , 
\ee 
where we have used the dominance of $\lv$ and the Iwahori--Matsumoto factorization to obtain the equality in the last line. Apply \cite[(5.24), (5.25)]{PPAIM} or using the relations described in \S \ref{subsub:gen-rel-met}, for $x_{-a}(s) \in U_{-a}(\O)$ with $s^{-1}= \pi^{-k} r, r \in \O^*$ and $k \geq 0$, we may write \be{rank1:LU-fact} x_{-a}(s) = x_a(s^{-1}) h_a(s^{-1}) \tw_a x_a(s^{-1}), \mbox{ where }  h_a(s^{-1}) = h_a(r) \pi^{ - k \av} (r, \pi)^{k \Qs(\av) }  . \ee Defining $\Psi(x_{-a}(s)) = \epsilon\left( (r, \pi)^{ k \Qs(\av)} \right)$, one finds then that $\Avg_{U^-}( q^{- \la \rho, \lv \ra }\vbi_{\lv} \star \T_{s_a} )$ is computed as the sum \be{}\label{sum-for-avg} 
\sum_{k \geq 0} q^{ \la \rho, \lv - \mv \ra } Y_{\lv - k \av } \, \int_{U_{-a}[k]} \psi(\pi^{\lv} u_{-a} \pi^{\lv}) \Psi(x_{-a}(s)) du_{-a}[k], \ee 
where $du_{-a}[k]$ is the restriction of the measure on $U_{-a}$ giving $U_{-a}(\O)$ total volume $1$. Note that non-zero summands in the above expression may arise only for  $Y_{\zv}$ when $\zv \in \{ \lv, \lv - \tav, \ldots, \lv - k_0 \tav, \lv \da w_a\}$, where $k_0$ is the largest positive integer such that $k_0 \tav \leq \la \lv, a \ra$.  We leave the rest of the computation to the reader as it follows from the same ideas at the end of \cite[\S5.7]{PPAIM}.

\tpoint{Proof of Proposition  \ref{thm:i-basis}, part 2} Recalling the definition of $\tnalc$ from \eqref{def:fund-alcove-closures}, we choose $\lv = \etav \da W$ for some $\etav \in \tnalc.$ As we observed earlier, this forces  \be{} \label{fin-W-cond}  -n(\av_i) < \la \lv + \rhov, a_i \ra < n(\av_i) \text{ for all } i \in I, \ee and one notes that verifying the Proposition for these $\lv$ is equivalent (see Lemma \ref{lem:explicitDLoperators}(2)) to verifying : 
\be{ } \label{rank1:x-rels} 
\x_{\lv} \cdot \tT_{s_{i}} = 
\begin{cases} 
\g_{\la \lv + \rhov, a_i \ra \Qs(\av_i)} \x_{\lv \da s_i } & \mbox{ if } -n(\av_i) < \la \lv + \rhov, a_i \ra < 0, \\  
\g_{\la \lv + \rhov, a_i \ra \Qs(\av_i)} \, \x_{\lv \da s_i } + (q - 1) \x_{\lv} & \mbox{ if } 0 < \la \lv + \rhov, a_i \ra < n(\av_i), \\ 
- \x_{\lv} & \mbox{ if } \la \lv + \rhov, a_i \ra = 0.
\end{cases}
\ee  
Indeed, we may just apply the isomorphism $\Avg_{U^-}$ to such relations and bear in mind that $\x_{\lv} \mapsto Y_{\lv}$ under $\Avg_{U^-}$.   Now the second case above follows since $\lv \in Y_+$, so we may apply part (1) and the general formulas for $\Tmw_{s_i}$ action.  Assume then that $\lv$ is as in the first case, and let us show the equivalent statement that $ \g_{\la \lv + \rhov, a_i \ra \Qs(\av_i)}^{-1} \x_{\lv} = \x_{\lv \da s_i } \cdot \tT_{s_i}^{-1}.$ Since  \be{} \Avg_{U^-}(\x_{\lv \da s_i } \cdot \tT_{s_i}^{-1})= q \, \Avg_{U^-}(\x_{\lv \da s_i } \cdot \tT_{s_i} ) +  (q - 1) \Avg_{U^-}(\x_{\lv \da s_i } ), \ee and $\la \lv \da s_i, a_i \ra \geq 0$ by our assumption on $\lv$,  the right hand side of the above can be computed explicitly using part (1) and the desired result follows. The third case is proven similarly. The Proposition is thus proven for all $\lv \in \tnalc \da W$.

\tpoint{Proof of Proposition \ref{thm:i-basis}, part 3} Observe that for $\zv \in \tY$, Proposition \ref{prop:averaging-isom} gives 
\be{} \x_{\lv} Y_{\zv} = \x_{\lv+ \zv} \text{ for } \lv \in \tnalc \da W , \zv \in \tY. \ee
As every $\mv \in Y$ is of the form $\mv = \lv + \zv$ with $\zv \in \tY$ and $\lv \in \tnalc \da W$, it sufices to show that

\be{} \label{x:ext-1} \Avg_{U^-} ( \x_{\lv + \zv} \star \tT_a ) = Y_{\lv + \zv} \cdot \Tmw_a \mbox{ for } a \in \Pi, \lv \in \tnalc \da W, \zv \in \tY. \ee 
Using Proposition \ref{prop:averaging-isom} we may compute
{\small \be{} 
\Avg_{U^-} ( \x_{\lv + \zv} \star \tT_a ) &=& \Avg_{U^-} ( \x_{\lv} \star Y_{\zv}  \cdot \tT_a ) =  \Avg_{U^-} ( \x_{\lv} \star \tT_a Y_{w_a \zv}  ) +  (q-1) \, \Avg_{U^-} ( \x_{\lv} ) \star \frac{Y_{w_a \zv} - Y_{\zv}}{1 - Y_{-\av}}  \\ 
&=& \Avg_{U^-} ( \x_{\lv} \star \tT_a ) Y_{w_a \zv}   +  (q- 1) \, \Avg_{U^-} ( \x_{\lv}) \star \frac{Y_{w_a \zv} - Y_{\zv}}{1 - Y_{-\av}}.
\ee } 
The terms on the right hand side are all computed from part (2), so \eqref{x:ext-1} can be easily verified.

\tpoint{A decomposition of $\whitiw$}
For $\etav \in \tnalc$, define the $\thiw$-module 
\be{} \label{whit:etav} \whitiw(\etav):=  \mathrm{Span}_{\C}\{ \ \x_{\etav} \star h \mid h \in \ \thiw \}. \ee

\noindent From the previous Proposition and the results of Section \ref{sec:gKL}, we obtain the following.

\begin{nprop} \label{prop:VtoP} For $\etav \in \tnalc$, the map $\mf{p}$ from Proposition \ref{thm:i-basis} restricts to an isomorphism  
\be{} \C \otimes_{\Zvg} \V(\etav)  \stackrel{\simeq}{\longrightarrow}  \whitiw(\etav)\ee 
and hence by Proposition~\ref{prop:Wdav} we have a decomposition into $\thiw$-submodules 
\be{} 
\whitiw:= \oplus_{\etav \in \tnalc} \, \whitiw(\etav). 
\ee 
\end{nprop} 
\newcommand{\Vh}{\wt{\mathscr{V}}}
\newcommand{\Ih}{\wt{\mathscr{I}}}
\newcommand{\Xh}{\wt{\mathscr{X}}}

\subsection{Structure of the spherical Whittaker space $\whitk$}\label{sub:p-adic:spherical}

\tpoint{Action of $\thsp$ on $\whitk$}\label{subsub:action:thsph:on:twhitk}

Let $\Vsp$ be $\tspaff$-module introduced in \S \ref{subsub:Vsp-g}. The map $\mf{p}: \C \otimes_{\Zvg} \V  \longrightarrow  \whitiw$ from the Proposition \ref{thm:i-basis} descends to a map that we denote as 
\be{} \mf{p}_K: \C \otimes_{\Zvg} \Vsp \longrightarrow \whitk.  \ee  Recall the diagram in \eqref{algebraic-hecke-1} and the remarks at the end of \S \ref{subsub:metIwahorHecke} which explain the relation between the element $\ek \in \thiw$ and $\epsilon  \in \taffH$. The natural right convolution maps $\cdot \star \ek: \whitiw \rr \whitk.$ 
\begin{nprop}\label{prop:diagrama:met:sph}
The following diagram 
\begin{center} 
\begin{tikzcd}[column sep = 4em, row sep = 2em]
	\whitiw \arrow[r, "\cdot  \, \star \ek"]  & \whitk   \\ 
	\C \otimes_{\Zvg} \V \arrow[u, "\cong"', "\mf{p}"] \arrow[r, " \cdot \, \epsilon" ] & \arrow[u, "\cong","\mf{p}_K"'] \C \otimes_{\Zvg} \Vsp  		
\end{tikzcd} 
which sends
\begin{tikzcd}[column sep = 4em, row sep = 2em]
	\x_{\lv} \arrow[r ] & \tJ_{\lv}   \\ 
	Y_{\lv} \arrow[r]  \arrow[u] & \ket{\lv} \arrow[u]
\end{tikzcd} 
for $\lv \in Y$
\end{center}
is commutative. The vertical maps in the left diagram are isomorphisms that intertwine the actions of the corresponding Hecke algebras: $\thiw \simeq \taffH$ on the left and $\thsp \simeq \tsphH$ on the right.  
\end{nprop}

\noindent We may define the $\thsp$-submodules 
\be{}  \label{whitk:etav} \whitk(\etav):= \whitiw(\etav) \star \ek ,\ee 
for $\etav \in \tnalc$ and by the Proposition above we have that $\mf{p}_K$ restricts to an isomorphism 
\be{} \mf{p}_K: \C \otimes_{\Zvg} \Vsp(\etav)  \longrightarrow \whitk(\etav)  \ee which intertwines the $\thsp \simeq \tspaff$-actions.
We can reformulate this using~\eqref{eq:def:vsp:met} as follows

\begin{ncor}\label{cor:decomposition:whitk}
One has the following decomposition of  $\whitk$ into $\thsp$-modules:
\be{}
\whitk \simeq \oplus_{\etav \in \tnalc} \whitk(\etav).
\ee 	
\end{ncor}

\noindent For $\lv \in Y_+$, define the natural bases $\tL_{\lv}$ and $\til_{\lv}$ in $\whitk$ which correspond under the isomorphism $\mf{p}_K$ to the bases  $\gketm{\lv}$ and $\gket{\lv}$ of $\C \otimes_{\Zvg} \Vsp(\etav)$, respectively. 
We call these the \emph{canonical} bases of $\whitk$.

\tpoint{Metaplectic geometric Casselman--Shalika formulas} Using Proposition~\ref{prop:diagrama:met:sph} and the remarks in the previous paragraph,  we can import the structures studied in Part~\ref{part:combinatorial_models} and especially in \S\ref{sec:gKL} to the study of  $\whitk$. 
The results in \S\ref{subsub:gtwistedLR} and \S\ref{subsub:gtwisted-Steinberg-Lusztig} produce the following geometric Casselman--Shalika formulas for the `$p$-adic basis' $\tJ_{\lv}$ and the canonical bases $\tL_{\lv}$ and $\til_{\lv}$ of $\whitk$, respectively.
 
\begin{nthm}  \label{thm:met-geom-CS-LR} 
\begin{enumerate}
\item Suppose $\mv \in Y_+$ and $\lv \in \tY_+$. Writing $ \leftidx^{\g} Q_{\mv, \lv}^{\zv}:= \mf{p}(\leftidx^{\gf} Q_{\mv, \lv}^{\zv} )$ for the $p$-adic specializations of the $\gf$-twisted Littlewood-Richardson coefficients defined in~\eqref{def:gLR}, we have 
\be{} \label{padic:Lys} 
\tJ_{\mv} \star \tckl_{\lv} = \sum_{\zv \in Y_+} \leftidx^{\g} Q_{\mv, \lv}^{\zv} \, \yket{\zv}.
\ee 

 \item \label{thm:met-geom-CS-SL} If $\lv_0 \in \Box_{(\Qs, n)}$ and $\zv \in \tY_+$, then we have 
\be{} \label{padic:SL} \tL_{\lv_0} \star \tckl_{\zv} = \tL_{\lv_0 + \zv}. \ee 

\item If $\lv_0 \in \Box_{(\Qs, n)}$ and $\zv \in \tY_+$, then writing $\lvbar_0:= \lv_0 \cdot w_0 + 2 (\t{\rho}^{\vee} - \rhov)$ we have  
\be{}\label{padic:tilt}\til_{\lvbar_0} \star \tckl_{\zv} = \til_{\lvbar_0 + \zv}. \ee
\end{enumerate}
\end{nthm}

\begin{nrem} Note that the second and third parts, together with the well-known structure coefficients for $\tckl_{\mv} \in \thsp$ (see~\eqref{eq:Shimura:LR:coefficients}), determine the action of $\tckl_{\mv}$ on the basis $ \tL_{\lv}$ and $\tT_{\lv}$ as explained in \eqref{eq:geomCSatmu} of \S\ref{intro:geom-cs}. \end{nrem}

\renewcommand{\csc}{\mathbf{c}}

\newcommand{\Repone}{\Rep(U_1(\mf{D}))}

\newcommand{\Repq}{\Rep(U_{\zeta}(\mf{D}_l))}

\part{Interpretation via quantum groups at roots of unity and applications}\label{part:qg} 

\section{Quantum groups at a root of unity} \label{sec:qg_at_rou}

After introducing in~\S\ref{sub:algebras:root:datum}-\ref{subsub:rou} quantum groups and their specializations (following Lusztig~\cites{lus:qg, lus:coord}), we collect several results from the representation theory of these objects. 
We are working with Lusztig's \emph{dotted} version of the quantum group and its root of unity and quasi-classical specializations, though we could have as well worked with the undotted versions of these objects as their representations theories are the same. 
The main results needed to connect the quantum setting to the $p$-adic ones are contained in \S\ref{sub:extension}.

Fix throughout $(I, \cdot, \mf{D})$ a root datum written as $\mf{D}= (Y, \{ \av_i\}, X, \{ a_i \})$ with Cartan matix $\As=(a_{ij})$ attached to $(I, \cdot)$. Let $\qv$ be a formal variable and, for each $i \in I$, set $\qv_i:= \qv^{\frac{i \cdot i}{2}}$. These elements lie in the ring $\zeev:= \zee[\qv, \qv^{-1}]$. Throughout this section we shall consider commutative, unital rings $A$ equipped with a ring homomorphisms $\phi: \zeev \rr A.$
We shall need to impose stricter conditions on $\mf{D}$ in \S\ref{sub:extension} as we will explain at the start of that subsection.

\subsection{Algebras attached to root datum}\label{sub:algebras:root:datum} 

\tpoint{The algebra $\falg$} Attached to the Cartan datum $(I, \cdot)$ with Cartan matrix $\As= (a_{ij})$, we construct $\falg:= \falg(I, \cdot)$ the associative $\ratv$-algebra with unit defined as in \cite[\S1.2.5]{lus:qg}. It is a deformation of the enveloping algebra of the `positive' half of $\mf{g}$, and defined as the unital algebra generated by elements $\theta_i \, (i \in I)$ subject to the quantum Serre relations: for $i, j \in I$, with $i\neq j$   
\be{} \label{quantum:serre}  \sum_{p, p' \in \mathbb{N}; p + p'= 1 - a_{ij} } (-1)^{p'} \left( \theta_i^p/[p]_i^! \right) \theta_j \left( \theta_i^{p'}/[p']_i^! \right) = 0  \mbox{ where } [p]^!_i = \prod_{s=1}^p \frac{u_i^s - u_i^{-s}}{u_i - u_i^{-1}}. \ee 
The algebra $\falg$ has a decomposition $\falg = \oplus_{\mu \in \zee[I]} \falg_{\mu}$ where $\falg_{\mu}$ is the vector space spanned by words in $\theta_i$ in which the number of appearances of $\theta_i$ is given by the coefficient of $i$ in $\mu$.

\tpoint{$\zeev$-forms for $\falg$} For each $i \in I, n \in \zee$, we define the divided power $\divp{\theta_i}{n}:= \theta_i^n/{[n]^!_i}$ if $n \geq 0$ and $\divp{\theta_i}{n}=0$ otherwise. Defining $\falg_{\zeev}$ to be the unital subalgebra of $\falg$ generated by  $\divp{\theta_i}{n}$ for all $i \in I, n \in \zee$, we have  
\be{} \falg_{\zeev} = \oplus_{\mu \in \zee[I]} \falg_{\mu, \zeev} \mbox{ where } \falg_{\mu, \zeev} = \falg_{\mu} \cap \falg_{\zeev}. \ee 
For any $\zeev \rr A$ as above, we  set $\falg_A:= A \otimes_{\zeev} \falg_{\zeev}$ and note again that \be{} \falg_{A} = \oplus_{\mu \in \zee[I]} \falg_{\mu, A} \mbox{ where } \falg_{\mu, A} =A \otimes_{\zeev}  \falg_{\mu, \zeev}. \ee

\tpoint{The quantum group $\dU_A(\mf{D})$} \label{subsub:defLusQG} Let $\dU_A(\mf{D})$ be the $A$-algebra generated by symbols $x^+ 1_{\zeta} y^-, x^- 1_{\zeta} y$ where $x \in \falg_{\mu, A}, y \in \falg_{\mu', A}$ for $\mu, \mu' \in \zee[I]$ and $\zeta \in X$ and subject to the relations described in \cite[\S31.1.3]{lus:qg}. Note that these relations depend on the root datum $\mf{D}$, not just on the Cartan datum $(I, \cdot)$. We also record here that the $A$-linear maps \be{} \falg_A \otimes_A \At[X] \otimes_A \falg_A \rr \dU_A(\mf{D} )&,& x \otimes X_{\lambda} \otimes y \mapsto x^+ 1_{\lambda} y^- \\ \falg_A \otimes_A \At[X] \otimes_A \falg_A \rr \dU_A(\mf{D})&,& x \otimes X_{\lambda} \otimes y \mapsto x^- 1_{\lambda} y^+ \ee are isomorphisms of $A$-modules and hence we have $\dU_A(\mf{D})= A \otimes_{\zeev} \dU_{\zeev}$. Note that $\dU_A$ does not have a unit, but instead a family of elements $(1_{\lambda})_{\lambda \in X}$ such that $1_{\lambda}1_{\lambda'}= \delta_{\lambda, \lambda'}.$ These give a decomposition $\dU_A = \sum_{\lambda, \lambda'} 1_{\lambda} \dU_A 1_{\lambda'}$. If $A=\ratv$, we often drop it from our notation and just write $\dU$ in this case; the natural map $\zeev \subset \ratv$ allows us to view $\dU_A \subset \dU$ as an $\zeev$-subalgebra, or, what we might call a $\zeev$-lattice.

\tpoint{Module categories} Let $\mf{C}_A(\mf{D})$ denote the category whose objects  $\dU_A(\mf{D})$-modules $M$ which are 
\begin{itemize}
	\item 	\emph{unital} in the sense that  for any $z \in M$, we have $1_{\lambda}z=0$ for almost all $\lambda \in X$ and $\sum_{\lambda \in X} 1_{\lambda}z = z$,
	\item finitely generated when regarded as an $A$-module. 
\end{itemize} 

\noindent Objects $M \in \mf{C}_A(\mf{D})$ can be decomposed as $M = \oplus_{\lambda \in X} M_{\lambda}$ where $M_{\lambda} = 1_{\lambda}M.$   
One can equip $\mf{C}_A(\mf{D})$ with the structure of a monoidal tensor category (\cf \cite[\S1.6]{lus:coord}) with  product denoted $\otimes_A$. 
If $A= \ratv$, we drop it from our notation and just write $\mf{C}(\mf{D})$ in this case.
An object $M \in \mf{C}_A(\mf{D})$ is said to be a \emph{highest weight module} with highest weight $\lambda \in X$, \cf \cite[\S31.3]{lus:qg} if there exists a vector $m \in M^{\lambda}$  such that \begin{itemize} 
	\item setting $\divp{E_i}{n}:= \sum_{\zeta \in X} (\divp{\theta_i}{n})^+1_{\zeta}$,  we have $\divp{E_i}{n} m=0$ for all $i$ and $n > 0$; 
	\item $M = \{ x^- m \mid x \in \falg_A \} $;
	\item $M^{\lambda}$ is a free $A$-module of rank one with generator $m$.
\end{itemize}	
We define a module $M$ to be \emph{integrable} if for every $m \in M$ and $i \in I$, there exists an $n_0$ such that $E_i^{(n_0)}m=F_i^{(n_0)}m=0$ (\cf \cite[\S31.2.4]{lus:qg}). 
For a field $F$, let  $\Rep(\dU_F(\mf{D}))$ be the full subcategory of $\mf{C}_A(\mf{D})$ consisting of integrable highest weight (unital) modules $M$ whose weight spaces $M_{\lv}$ are finite-dimensional $F$-vector spaces. For such modules we can define their character 
\be{}\label{character:M} \chi_M:= \sum_{\lambda \in X} \left( \dim_F \, M_{\lambda} \right) \, X_{\lambda} \in F[X]. \ee

\tpoint{Standard or Weyl modules $\Lambda_{\lambda}$} \label{subsub:standard-def} For  $\lambda \in X_+$, define the $\falg$-modules $T_{\lambda} = \sum_i \falg \theta_i^{\la \av_i, \lambda \ra +1}$ and define $ \Lambda_{\lambda} = \falg / T_{\lambda}.$  Letting $\eta_{\lambda}= 1 + T_{\lambda},$ we may regard $\Lambda_{\lambda}$ as an element in $\mf{C}(\mf{D})$ \cf \cite[6.3.4, 23.1.4]{lus:qg} in which \be{} \begin{array}{lcr}  (\theta_i^+1_{\lambda'})\eta_{\lambda}=0, i \in I, \lambda' \in X & \text{and } & (x^- 1_{\lambda'}) \eta_{\lambda} = \delta_{\lambda, \lambda'} x + s, \text{ where } s \in T_{\lambda} \text{for } x \in \falg, \lambda' \in X. \end{array} \ee We also write, given any $\zeev$-algebra $\zeev \rr A$, 
\be{} \begin{array}{lcr} \Lambda_{\lambda, \zeev} := \dU_{\zeev}(\mf{D}) \eta_{\lambda} & \text{ and } &  \Lambda_{\lambda, A} := A \otimes_{\zeev} \Lambda_{\lambda, A} \end{array} .\ee One knows that $\Lambda_{\lambda, \zeev}$ is an $\zeev$-lattice in $\Lambda_{\lambda}$ and $  \Lambda_{\lambda, A} \in \mf{C}_{A}(\mf{D})$. These are called the \emph{standard} modules. Over a field $F$,  $\Lambda_{\lambda, k} \in \Rep(\dU_F(\mf{D}))$ and their characters can be computed using the Weyl character formula. 

\tpoint{Irreducible modules } Again we work over a field $F$. The structure of irreducible highest weight modules in $\Rep(\dU_F(\mf{D}))$ is similar to that for ordinary Lie algebras, \cf \cite[Prop. 31.3.2]{lus:qg}: for each $\lambda \in X_+$ there is a simple object in $\Rep(\dU_F(\mf{D}))$ which is a highest weight module with highest weight $\lambda$. We denote this obect as $L_{\lambda, F}$ (or just $L_{\lambda}$ if $F$ is implicitly understood). One know that $L_{\lambda, F}$ and $L_{\lambda', F}$ are not isomorphic if $\lambda, \lambda' \in X$ are distinct. Moreover, if $M$ is any highest weight module with highest weight module with highest weight $\lambda$, then $M$ has a unique maximal subobject whose corresponding quotient is isomorphic to $L_{\lambda, F}$.

\subsection{Classical and quasi-classical specializations}  \label{subsub:quasi-classical}

\tpoint{Classical and quasi-classical specializations} Let $\phi: \zeev \rr A$ be a $\zeev$- algebra such that $\phi(\qv_i)= \pm1$ for all $i \in I.$ In this case $\dU_A$ will be called  a \emph{quasi-classical} specialization of $\dU_{\zeev}$. If $\phi(\qv)=1$, then we say that $\dU_A$ is a \emph{classical} specialization of $\dU_{\zeev}.$ We would like to describe the categories $\mf{C}_A(\mf{D})$ when $A$ is taken to be one of these two specializations. As it turns out if $\phi: \zeev \rr A$ is a quasi-classical specialization, and we write $A_0=A$ for a copy of $A$ and define $\phi_0: \zeev \rr A_0$ to be the unique map such that $\phi(\qv)=1$, then from  \cite[Prop. 33.2.3]{lus:qg}, there is a unique isomorphism of $\zeev$-algebras \be{} \dU_{A_0}(\mf{D}) \stackrel{\cong}{\rr} \dU_A(\mf{D}), \mbox{ and hence }\mf{C}_{A_0}(\mf{D}) \cong \mf{C}_{A}(\mf{D}). \ee Not only do the representation categories under classical and quasi-classical specializations agree, but from the construction in \emph{op. cit.} it follows that if $M_0 \in  \mf{C}_{A_0}(\mf{D})$ corresponds to  $M \in \mf{C}_{A}(\mf{D})$ under this equivalence, then it also follows that $M_{0, \lambda} \cong M_\lambda$ for each $\lambda \in X.$ 

\tpoint{Classical enveloping algebras attached to root datum} Following \cite[\S5.1]{lus:coord}, define $\Ug_{\Q}(\mf{D})$ be the $\Q$-algebra with unit generated by symbols $x^+, x^-$ with $ x \in \falg_{\Q}$ and $\u{y}, y \in Y$ subject to the relations: \begin{itemize} 
	\item the maps $\falg_{\Q} \rr \Ug_{\Q}(\mf{D}), x \mapsto x^{\pm}$ are $\Q$-algebra homomorphisms preserving $1$,
	\item the map $Y \rr \Ug_{\Q}, y \mapsto \u{y}$ is $\zee$-linear,
	\item $\u{yy'} = \u{y}\u{y'}$ for $y, y' \in Y$,
	\item $\u{y} \theta_i^{\pm} = \theta_i^{\pm}( y\pm \la y, a_i\ra)$ for $y \in Y, i \in I$,
	\item $\theta_i^+ \theta_j^- \theta_j^-\theta_i^+ = \delta_{i, j} \underline{\av_i}$ for $i, j \in I$.	\end{itemize} 
This is a generalization of the usual universal enveloping algebra (attached to a semisimple algebra), and will be called the \emph{universal enveloping algebra attached to a root datum}. One may equip it with a Hopf algebra structure as in \emph{op. cit.}

\tpoint{Kostant type forms} Let $\Ug_{\zee}(\mf{D})$ be the subring of $\Ug_{\Q}(\mf{D})$ generated by the elements $$ \left( \frac{y}{k} \right) := \frac{1}{k!} \u{y} \,  \u{y-1} \, \cdots \, \u{y-k+1} \text{ for } y \in Y, k \in \mathbb{N} $$ as well as $\left(\divp{\theta_i}{m}\right)^{\pm}$ for $i \in I, m \in \mathbb{N}$.  It is a $\zee$-lattice in $\Ug_{\Q}(\mf{D})$ which is the analogue of Kostant's $\zee$-form  \cf \cite{kostant}, which it agrees with when $\mf{D}$ is of adjoint type (we eventually need to restrict to this case due to existing limitations in the literature).

\spoint Let $\mf{C}'_{\Q}(\mf{D})$ denote the category of unital $\Ug_{\Q}(\mf{D})$-modules $M$ with vector space decompositions \be{} M = \oplus_{\lambda \in X} M^{\lambda} \mbox { where we set }  M^{\lambda}= \{ z \in M \mid \u{y} z = \la y, \lambda \ra z \text{ for any } y \in Y \}. \ee The algebra form \S\ref{subsub:defLusQG} attached to the map $\phi: \zeev \rr \Q, u \mapsto 1$ will be denoted $\dU_{\Q}(\mf{D})$ and its category of unital modules written  $\mf{C}_{\Q}(\mf{D}).$ From \cite[\S5.4]{lus:coord}, one has an equivalence of categories \be{} \mf{C}_{\Q}(\mf{D}) \stackrel{\cong}{\rr} \mf{C}'_{\Q}(\mf{D}) .\ee 
These categories are semisimple with the simple objects $V_\lambda := \Lambda_{\lambda, \Q}, \lambda \in X_+$ introduced above. 
Letting $\bG_{\mf{D}}$ denote the algebraic group attached to the root datum $\mf{D}$, it then also follows from \cite{lus:coord} that there is an equivalence to the the category of finite-dimensional representations of $\bG_{\mf{D}}(\C)$, i.e.   \be{} \label{quasi-classical-to-group} \Rep(\dU_{\C}(\mf{D}))  \simeq \Rep(\bG_{\mf{D}}(\C)). \ee

\subsection{Quantum groups at roots of unity} \label{subsub:rou}

\newcommand{\zeevl}{\zee_l}

\tpoint{Roots of unity} \label{subsub:roots-l} Let $\ell$ be any positive integer. If $\ell$ is even, set $l = 2 \ell$ while if $\ell$ is odd, we define $l = \ell$ or $2 \ell$. Let $\zeevl = \zeev/ (\Phi_{l}(\qv))$ where $\Phi_d$ is the $d$-th cyclotomic polynomial. In this section, we only consider  $\phi: \zeev \rr A$ which factors through $\zeevl$, i.e. $\Phi_l(\phi(\qv))=0.$ For $\zeta \in \C$ a primitive $l$-th root of unity, we define the corresponding map $\phi_{\zeta}:  \zeev \rr \C$ sending $\qv \mapsto \zeta.$ With respect to this map, we then write  \be{} \dU_{\zeta}(\mf{D}):= \C  \otimes_{\zeev}  \dU_{\zeev}(\mf{D}). \ee  

\tpoint{On $\Rep(\dU_\zeta(\mf{D}))$} The category $\Rep(\dU_\zeta(\mf{D}))$ is isomorphic to the category of finite dimensional highest weight complex modules of type $\bf{1}$ of the (non-dotted) quantum group at a root of unity (as defined as in~\cite[Definition 1.1]{Andersen:linkage}, see also~\cite{Ko:ext}).  
For $\lambda \in X_{+}$, we denote by $\Delta_\lambda$ the specialization of the modules  $\Lambda_\lambda$ defined in~\S\ref{subsub:standard-def}.
We denote by $\nabla_\lambda$ the costandard object of highest weight $\lambda$ (for its construction, which is essentially `dual' to that of $\Delta_{\lambda}$, we refer to either ~\cite{Jantzen:book} where they are denoted by $H^0(\lambda)$ in or to ~\cite{ClineParshallScott:enriched} where they are denoted as $A(\lambda)$). The characters of both standard and costandard modules are given by the Weyl character formula.
Let $L_\lambda$ be the simple module with highest weight $\lambda \in X_+$ in $\Rep(\dU_\zeta(\mf{D}))$. Their characters are given by Lusztig's conjecture (see, for example~\cite[Appendix H.12]{Jantzen:book}).

\tpoint{Tilting modules} Tilting modules are a class of representations that naturally appear when studying $\Rep(\dU_\zeta(\mf{D}))$. They are objects admitting both a filtration by standard modules and a filtration by costandard modules. Andersen~\cite[Theorem 2.5]{andersen:CMP} showed (using ideas from modular representation theory) that there exists a unique indecomposable tilting module in $\Rep(\dU_\zeta(\mf{D}))$ with highest weight $\lambda \in X_+$. We shall denote such a tilting by $T_{\lambda}$. See~\cite[\S2]{andersen:nato} for an explicit construction of $T_\lambda$.

The characters of $T_\lambda$ are known due to the work of Soergel~\cite{soergel:combinatoric,soergel:tilting} in many cases.

\tpoint{On $\Rep(\dU(\mf{D}_l))$} If $(I, \cdot, \mf{D})$ is our given root datum, where $\mf{D}= (Y, \{ \av_i\}, X, \{ a_i\})$ is assumed of adjoint type,  we may construct  the $l$-twisted root datum $(I, \circ_l, \mf{D}_l)$ as in \S\ref{subsub:exampleltwisted}. Denoting $\mf{D}_l:= (Y_l, \{ \tav_i \}, X_l, \{ \ta_i\} ),$ we have that $\tav_i = l_i^{-1} \av_i$, $\ta_i = l_i a_i$ where $l_i$ were defined in \S\ref{subsub:exampleltwisted}.  With respect to $(I, \circ_l)$ we then have $\qv_i:= \qv^{l_i^2}$ (cf. \cite[35.2.1]{lus:qg}) so that if we are working with a map $\phi: \zeev \rr A$ as in the previous section, we are necessarily in the quasi-classical case. So, for example $\dU_{\C}(\mf{D}_l)$ for the map $\zeev \rr \C, v \mapsto \zeta$ is a  quasi-classical specialization of $\dU_{\zeev}(\mf{D}_l)$, and hence from \eqref{quasi-classical-to-group} 
\be{} \Rep_{\C}(\dU(\mf{D}_l)) \cong \Rep(\bG_l(\C)), \ee 
where $\bG_{\ell}$ is the algebraic group with root datum $\mf{D}_l$.

\tpoint{Quantum Frobenius} For any $\phi: \zeevl \rr A$ as in \S\ref{subsub:roots-l}, the quantum Frobenius morphism of Lusztig, \cf\cite[Ch. 35]{lus:qg}, is the unique $A$-algebra homomorphism $\Fr: \dU_A(\mf{D}) \rr \dU_A(\mf{D}_l)$ such that for all $i \in I, n \in \zee$ and $\zeta \in X$ we have \be{} \begin{array}{lcr}  \Fr : E_i^{(k)} \mapsto 
\begin{cases} E_i^{\left(\frac{k}{l_i}\right)} 1_{\zeta} & \mbox{ if } k \equiv 0 \mbox{ mod } l_i, \\
	0 & \mbox{ otherwise }
\end{cases} & \text{ and }
 \Fr : F_i^{(k)} \mapsto 
\begin{cases} F_i^{\left(\frac{k}{l_i}\right)} 1_{\zeta} & \mbox{ if } k \equiv 0 \mbox{ mod } l_i, \\
	0 & \mbox{ otherwise }.
\end{cases} \end{array}\ee

\noindent	Using this map of algebras, we can define   a fully faithful and exact embedding \be{def:fr} \Fr : \mf{C}_A(\mf{D}_l)  \to \mf{C}_A(\mf{D}).\ee For $\zeev \rr \C, v \mapsto \zeta$ as above, we obtain a functor  \be{} \label{quantumF:reps} \Fr: \Rep(\bG_{\mf{D}_l}(\C) ) \rr \Rep(\dU_\zeta(\mf{D})) . \ee
	
\noindent Replacing $\mf{D}$  by $\check{\mf{D}},$  writing  $\check{\mf{D}}_l = \left(\mf{D}_{(\Qs, l)} \right)^{\vee}$ for $(\Qs, l)$ as in \S\ref{subsub:exampleltwisted}, and setting $\check{\bG}_{\ell}:= \bG_{\check{\mf{D}}_l}$, we have
\be{} \label{quant-Frob}  \Fr: \Rep(\check{\bG}_{\ell}(\C) ) \rr \Rep(\dU_\zeta(\check{\mf{D}})). \ee 
The simples in $\Rep(\check{\bG}_{\ell}(\C) )$, which we'll denote from now on as $V_{\lv}$, are parametrized by $\lv \in Y_{l,+}$ which we regard as a subset of $Y_+$ (the latter is the set parametrizing simples $L_{\mv}$ in $\Rep(\dU_\zeta(\check{\mf{D}}))$.) As $\Fr$ is exact,  $\Fr(V_{\lv}) \in \Rep(\dU_\zeta(\check{\mf{D}}))$ for $\lv \in Y_l$ is again simple. Unlike the general simple in the latter, one can explicitly write down the character of these simples using a dilation of the Weyl character formula \cite[Prop. 35.3.2]{lus:qg}.

\renewcommand{\int}{\mathrm{int}}

\tpoint{Tensor product theorems} \label{subsub:quantum:tensor:products}
Recall the notion of restricted weights from \S\ref{subsub:boxes} and consider the box in the affine root system determined by $\check{\mf{D}}_{l}$. Concretely, 
\be{eq:qgrestrictedregion} \Box_l = \{ \lv \in Y \,\, | \,\, 0 \leq \langle \lv, a_i \rangle < l_{i}, \mbox{ for all }  i \in I \}, \ee and we note that each $\lv \in Y_+$ can be decomposed uniquely as $\lv = \lv_0 +  \etav$ where $\lv_0 \in \Box_l$ and $\etav \in Y_l.$

\begin{nthm} \label{thm:quantum:tensor:products}
Let $\lv_0 \in \Box_l$ and $\etav \in \tY_{l,+}$. 
	\begin{enumerate}	
	\item (Steinberg-Lusztig) \label{thm:SteinbergLusztigrQGrep}  We have an isomorphism of modules 
	\be{}  L_{\lv_0+ \etav}  \simeq L_{\lv_0} \otimes \Fr(V_{\etav}) . \ee
	
	\item (Tilting) Writing $\lvbar_0:= \lv_0 \cdot w_0 + 2 (\t{\rho}^{\vee}- \rhov)$, we have \be{} T_{\lvbar_0 + \etav} \simeq L_{\lvbar_0} \otimes \Fr(V_{\etav}) . \ee
\end{enumerate}
\end{nthm}

The first is the analogue Steinberg tensor product theorem from modular representation theory which has been proved for $l$ odd by Lusztig~\cite{lus:modular-quantum} (for general $l$, see \cite{AndersenParadowski:fusion}). The second is due to Andersen \cite[Corollary 5.10]{andersen:CMP} \label{thm:TensorProductTilting}, and follows closely from related results by Donkin in the mod $p$ literature.

\subsection{Background on enriched Grothendieck rings}  \label{sub:enriched}

\tpoint{Enriched Grothendieck rings of highest weight categories}\label{subsub:enrichedGr}

Given an abelian category $\mathcal{C}$, it is natural to consider the derived bounded category $D^b(\mathcal{C})$ (see~\cite[Chapter III]{gelfand:manin} for definitions and basic properties of derived categories). The Grothendieck ring does not distinguish between the $\mathcal{C}$ and $D^b(\mathcal{C})$, namely $K_0(\mathcal{C}) \simeq K_0(D^b(\mathcal{C})).$ For a natural class of categories arising in representation theory,  a finer invariant called the enriched Grothendieck group was introduced in \cite{ClineParshallScott:enriched} to partially remedy this deficiency. To introduce it, let $\tau$ be an indeterminate and consider the ring $\zee [\tau,\tau^{-1}]$ of Laurent polynomials\footnote{At the categorical level, $\tau^{-1}$ will correspond to the exact functor $t : D^b(\mathcal{C}) \to D^b(\mathcal{C})$ that maps $\tau^{-1}: X \mapsto X[-1]$ on objects and $t: f \mapsto f[-1]$ on morphisms.}. 
Let $\mathcal{C}$ be a highest weight categories as introduced in~\cite{CPS:crelle} with weight poset $\Lambda$ indexing simple objects in $\mc{C}.$ Define the full additive subcategories $\hat{\mathcal{E}}^L \subset D^b(\mathcal{C})$ and $\hat{\mathcal{E}}^R \subset D^b(\mathcal{C})$ as in~\cite[Section 2]{ClineParshallScott:enriched}. Then 
$K_0^L(\mathcal{C})$ is the free abelian group generated by symbols $[X], X \in \hat{\mathcal{E}}^L$ subject to the relation $[X]+[Z]=[Y]$ if there is a distinguished triangle $X \to Y \to Z$ in $\hat{\mathcal{E}}^L$. One may define $K_0^R(\mathcal{C})$ similarly. 

\tpoint{Inner products}  Let $\mc{C}$ be a highest weight category equipped with weight poset $\Lambda_+$. 
For $\lambda \in \Lambda_+$, we denote by $\Delta_\lambda$ the standard object of highest weight $\lambda$ (also known as the Weyl module) and by $\nabla_\lambda$ the costandard object of highest weight $\lambda$ (which is the induced module and denoted $H^0(\lambda)$ in~\cite{Jantzen:book}). Let $L_\lambda$ be the simple module with highest weight $\lambda$; it appears as the unique irreducible quotient of $\Delta_\lambda$. 

\begin{nprop}\cite[Proposition 2.3]{ClineParshallScott:enriched} \label{prop:enrichedG} 
The spaces $K_0^L(\mathcal{C})$ and $K_0^R(\mathcal{C})$ are free $\zee[\tau,\tau^{-1}]$-modules with basis given by $\{ [\Delta_\lambda], \lambda \in \Lambda_+ \}$ and $\{ [\nabla_\lambda], \lambda \in \Lambda_+\}$, respectively. Moreover, there is a natural non-degenerate, sesquilinear pairing 
\be{eq:def:graded:inner:prod} \begin{array}{lcr} \la \cdot  , \cdot \ra : K_0^L(\mathcal{C}) \times K_0^R(\mathcal{C}) \to \zee [\tau, \tau^{-1}]; \quad \la [V], [W] \ra = {\bf{R}}\Hom(V,W) := \sum_{i\geq 0} \tau^{-i} \dim(\Ext^i_{\mathcal{C}} (V,W)) . \end{array} \ee
Under this pairing $[\Delta_{\mu}], [\nabla_{\eta}]$ form dual bases, i.e., 
$ \la [\Delta_{\mu}], [\nabla_{\eta}] \ra = \delta_{\mu, \eta} $ for $\mu, \eta \in \Lambda_+$.
\end{nprop}
\noindent Using this pairing, for $V \in \mathcal{C} $ its image $V$ in $K_0^R(\mathcal{C})$ may be written as
\be{} \label{eq:basisexpansionenrichedGr}
[V] = \sum_{\mu} \la [\Delta_{\mu}], [V] \ra \, [\nabla_{\mu}]  .
\ee

\begin{nrem}
Note that in~\eqref{eq:def:graded:inner:prod} we are using $\tau^{-1}$ instead of the $t$ of~\cite{ClineParshallScott:enriched}. This is done to make the matching with the results in \S\ref{sub:HeckealgebrasandKLtheory} easier.  
\end{nrem}

\tpoint{A trivial example}\label{subsub:K_0G}  Let $(I, \cdot, \mf{D})$ be a root datum, and $\bG$ the corresponding algebraic group. Then $\Rep(\bG(\C)),$ the semi-simple category of finite-dimensional representations of $\bG(\C),$ is a highest weight category and there are natural $\zee[\tau, \tau^{-1}]$  isomorphisms 
\be{}
K_0^L (\Rep(\bG(\C))) \simeq K_0^R  (\Rep(\bG(\C)))  \simeq \zee[\tau,\tau^{-1}] \otimes_\zee K_0  (\Rep(\bG(\C))) .
\ee
\newcommand{\Kz}{K_0^\odot}
Let us denote by $\Kz(\Rep(\bG(\C)))$ the space $\zee[\tau,\tau^{-1}] \otimes_\zee K_0  (\Rep(\bG(\C)))$.

\subsection{Extension formulas in $\Rep(\dU_{\zeta}(\check{\mf{D}}))$} \label{sub:extension}

In this section we need to make more stringent requirements on $\mf{D}$ and $\ell$. Fix $\mf{D}= (Y, \{ \av_i\}, X, \{ a_i\})$, which implies $\check{\mf{D}}= (X, \{ a_i\}, Y, \{ \av_i\})$. We emphasize that all representations from now onwards will be over $\C$.
Our restrictions are: 
\begin{itemize}
	\item We assume $(I, \cdot, \mf{D})$ is simply-connected. This means $\check{\mf{D}}$ is of adjoint type and so admits a primitive twist $(\check{\Qs}, l)$ for any $l$. The quantum group $\dU_{\zeta}(\check{\mf{D}})$ is also isomorphic to the quantum group associated to a semi-simple Lie algebra $\dU_\zeta (\mf{g})$.
	\item We choose $l$ larger than the Coxeter number of the semi-simple Lie algebra $\mf{g}$ of Cartan type given by $(I, \cdot)$ and also that $l$ is KL-good (as defined in~\cite[\S7]{Tanisaki}). Under these assumptions, one can invoke the Kazhdan--Lusztig equivalence \cite{KL:IMRN}.
\end{itemize}

Under the first assumption, it is known that $\Rep(\dU_{\zeta}(\check{\mf{D}}))$ is a highest weight category \cf\cite{CPS:homologicaldual}, see also \cite[\S2]{Ko:ext}. It has standard objects $\Delta_{\lambda}:= \Lambda_{\lambda,\C}$, costandard objects which we denote as $\nabla_{\lambda}$ and irreducibles $L_{\lambda}$ for $\lambda \in Y_+$.
The second assumption is needed because we use a result of Ko (Proposition~\ref{prop:ext:Ko}).

\spoint \label{subsub:K_0K^Rmodule} Under the assumptions above, the spaces $K_0^L(\Rep (\dU_\zeta(\check{\mf{D}})))$ and $K_0^R(\Rep (\dU_\zeta(\check{\mf{D}})))$ are free $\zee[\tau,\tau^{-1}]$-modules with (dual) bases given by the classes $[\Delta_{\lv}]$ and $[\nabla_{\lv}]$, respectively, for $\lambda \in Y_+$. 
We will mostly be working with $K_0^R(\Rep (\dU_\zeta(\check{\mf{D}})))$. 
Recall that $\Kz(\Rep(\check{\bG}_{\ell}(\C)))$ defined in~\S\ref{subsub:K_0G} has basis $[V_{\etav}], \etav \in Y_{l,+}$.
The quantum Frobenius map \eqref{quant-Frob} equips $K_0^R(\Rep (\dU_\zeta(\check{\mf{D}})))$ with the structure of a $\Kz(\Rep(\check{\bG}_{\ell}(\C)))$-module defined as follows: for $V \in \Rep(\check{\bG}_{\ell}(\C))$ and $W \in \Rep (\dU_\zeta(\check{\mf{D}}))$, then 
\be{}\label{def:odot}  [W] \odot [V]  = [W \otimes \Fr(V)] \in K_0^R(\Rep (\dU_\zeta(\check{\mf{D}})) .\ee 
The $\zee[\tau, \tau^{-1}]$-module $K_0^R(\Rep (\dU_\zeta(\check{\mf{D}}))$ also has another basis coming from simple modules $[ L_{\lv}], \lv \in Y$, and Theorem \ref{thm:quantum:tensor:products}(1) implies that for $\lv_0 \in \Box_l$ and $\etav \in Y_{l, +}$, one has 
\be{} [L_{\lv_0}] \odot [V_{\etav} ]= [ L_{\lv_0 + \etav}] \in K_0^R(\Rep (\dU_\zeta(\check{\mf{D}})). \ee

\noindent This relation determines the structure of $K_0^R(\Rep (\dU_\zeta(\check{\mf{D}}))$ as a $\Kz(\Rep(\check{\bG}_{\ell}(\C)))$-module in the following sense: if $\lv \in Y_+$ and $\etav \in Y_{l, +}$, we write $\lv = \lv_0 + \etav'$ with $\lv_0 \in \Box_l$ and $\etav' \in Y_{l, +}$, and then we can compute 
\be{} [L_{\lv}] \odot [V_{\etav}] = \left( [L_{\lv_0}] \odot [V_{\etav'}] \right)\odot [V_{\etav}] = [L_{\lv_0}] \odot ( [V_{\etav'}] \cdot [V_{\etav}]), \ee 
where $[V_{\etav'}] \cdot [V_{\etav}]$ is computed in $\Kz(\Rep(\check{\bG}_{l}(\C))$ using the usual Littlewood--Richardson coefficients.  

Alternatively, the $\zee[\tau, \tau^{-1}]$-module $K_0^R(\Rep (\dU_\zeta(\check{\mf{D}}))$ also has another basis consisting of indecomposable tilting modules $[ T_{\lv}], \lv \in Y$ and if $\lv = \lv_0+\etav$ with $\lv_0 \in \Box_l$ and $\etav \in Y_{l, +}$, Theorem~\ref{thm:TensorProductTilting} implies that 
\be{} [T_{\lvbar_0}] \odot [V_{\etav} ]= [ T_{\lvbar}] \in K_0^R(\Rep (\dU_\zeta(\check{\mf{D}})). \ee

\spoint Recall the Hecke algebra $\tspaff = \tspaff(\mf{D}) \simeq H_{\sp} (\mf{D}_{(\Qs, l)})$ from  \S\ref{subsub:actionHsponVspquantum}. 
It is a standard fact that 
\be{eq:SatakeGcheckl}
\Kz(\Rep(\check{\bG}_{\ell}(\C))) \simeq \zee[Y_l]^W \simeq \tspaff,
\ee
where in \S\ref{subsub:actionHsponVspquantum} we used $\tY$ for $Y_l$. 
The space $\tVs$ from \S\ref{subsub:qVspherical} is a $\tspaff$-module, where we take the twist $(\Qs,l)$ to be a multiple of a primitive twist $(\Qs_{\prim},l)$. Using Proposition~\ref{prop:TensorProd} and Theorem~\ref{thm:quantum:tensor:products}(1), we deduce: 

\begin{nprop} \label{cor:KR:tVs} The map $\Phi: K_0^R(\Rep (\dU_\zeta(\check{\mf{D}})) \rr \tVs $ sending $[L_{\lv}] \mapsto \qlketm{\lv} $ for $\lv \in Y_+$ is an isomorphism of $\zee[\tau, \tau^{-1}]$-modules which intertwines the $\Kz(\Rep(\check{\bG}_{\ell}(\C))) \simeq \tspaff$ actions. In particular,  
\be{} \begin{array}{lcr} \Phi (v ) \star \ckl_{\muv} = v \odot [V_{\muv}]  & \text{ for } & v \in  K_0^R(\Rep (\dU_\zeta(\check{\mf{D}})), \muv \in Y_{l, +}. \end{array} \ee  \end{nprop}

\spoint 
Let $\taffW:= \affW(I, \circ_l, \mf{D}_l^{\vee})$. Let us also recall the definition of $\tnalc$ from \S\ref{subsub:twisted-alcoves}, which is again defined for root datum $(I, \circ_l, \mf{D}_l^\vee)$. 
The following result of Ko will be used in our matching. 
\begin{nprop}\label{prop:ext:Ko} \cite[Thm. 4.10]{Ko:ext} 
Fix $\etav \in \tnalc$ and $J \subset I_{\aff}$ as in \eqref{def:J} and let $x, y \in \affW$ such that $\etav \da x, \etav \da y \in Y_+$. Then 	
\be{eq:inprod:delta:l} \la [\Delta_{\etav \da y}], [ L_{\etav \da x }] \ra = \sum_{z \in \leftidx_J(\t{W}_l)} (-\tau^{-1})^{\ell(z)} m^-_{zy,x}.\ee 	
\end{nprop}		
\begin{proof}
This is essentially the result of~\cite[Thm. 4.10]{Ko:ext}; we will now explain how to translate it into our setting.  
Unlike \emph{op. cit}, we use a right action of $\t{W}_l$ on the weight lattice as opposed to a left action. As such,  $x \da \lambda$ in \emph{op. cit} corresponds to our $\lambda \da x^{-1}$.  
Now the $t$ in \emph{op. cit} corresponds to our $\tau^{-1}$ and $t^{\ell(x)-\ell(y)} P_{y,x}(t^{-1})$ from  (2.3.1) in \emph{op. cit} (which, we note, matches with the $m_{y,x}(t)$ in~\cite[II.C.2]{Jantzen:book} needs to be replaced with $m^{-}_{y^{-1},x^{-1}}$ in our conventions, since  $M_J$ in~\eqref{pos:parabolic-mod} is a left module for $H_{\sW}$ while~\cite{Jantzen:book} works with right modules and Jantzen's $m$ is our $m^{-}$). 	
Finally, we note that the expression  $$t^{\ell(x)-\ell(y)}P^J_{y,x} (t^{-1}) := t^{\ell(x)-\ell(y)} \sum_{z\in (\taffW)_J} (-1)^{l(z)} P_{yz,x} (t^{-1}) = \sum_{z\in (\taffW)_J} (-t)^{l(z)} t^{\ell(x)-\ell(yz)}P_{yz,x} (t^{-1})$$ in \cite[\S4.2]{Ko:ext} matches to Jantzen's~\cite{Jantzen:book} $\sum_{z \in (\t{W}_l)_J} (-t^{-1})^{\ell(z)} m_{yz,x}(t)$ which in turn  matches to our 
\[\sum_{z \in \leftidx_J(\t{W}_l) } (-\tau^{-1})^{\ell(z)} m^-_{zy,x}.\]
\end{proof}

\spoint The following result follows immediately by combining Propositions~\ref{cor:KR:tVs} and~\ref{prop:ext:Ko},~\eqref{eq:SatakeGcheckl} and \eqref{eq:lvectovec} and the definition of $Q_{\etav, \mv}^{\lv}$ in~\eqref{LR-poly}.

\begin{nprop}\label{prop:LtoN:Ko} For any $\mv \in Y_+$ and $\lv \in Y_{l, +}$ we have  \be{} [ \nabla_{\mv} ] \odot [V_{\lv}] = \sum_{\etav \in Y} Q_{\etav, \mv}^{\lv} [\nabla_{\etav}]. \ee
\end{nprop}	
	

\begin{proof} It suffices to show that for any $\mv \in Y_+$, we may write
\be{eq:L:to:nabla} [L_{ \mv}] = [ \nabla_{\mv} ] + \sum_{\xv < \mv } o^{-}_{\mv,\xv} [\nabla_{ \xv}]. \ee 
We know that $o_{\mv, \xv}^{-} = o^{-}_{y, x}$, so the result above follows by \eqref{eq:basisexpansionenrichedGr},~\eqref{eq:lvectovec:neg} and \eqref{eq:inprod:delta:l}.
\end{proof}

\begin{nrem}\label{rem:KLgoodnotneeded} The proof presented for this result is not optimal as it assumes (via the dependence on \cite{Ko:ext}) the graded version of Lusztig's conjecture. The restrictions on $l$ which we place (especially $l$ being KL-good) stem from the application of Kazhdan--Lusztig equivalence that is used to prove this conjecture. We believe\footnote{We thank Hankyung Ko and Catharina Stroppel for discussions and advice on this point.} there should be a proof that bypasses this work and just uses the fact $\nabla_{\mv} \otimes \Fr(V_{\zv})$ has a $\nabla$ filtration where the successive quotients can be expressed combinatorially.  \end{nrem}

\tpoint{On tiltings and dualities}\label{subsub:assumption-tilting} 
Using the $\At$-module $K_0^L(\Rep (\dU_\zeta(\check{\mf{D}}))$ allows us to naturally connect (at least conjecturally) the tilting modules with the polynomials $o_{\lv, \mv}$. Recall that $o_{\lv, \mv} \in \At^+$, so we do not expect it to arise from an inner-product of the type introduced in Proposition \ref{prop:enrichedG}, which would take values in $\tau^{-1}$. On the other hand, one may just take $\overline{o}_{\mv, \lv} \in \At^-$ and posit that \be{}  \label{assumption:tilting-o} \la T_{\lv} , \nabla_{\muv} \ra = \overline{o}_{\muv, \lv}. \ee If this were true, then in $K_0^L(\Rep (\dU_\zeta(\check{\mf{D}}))$, we would have 
\be{} \label{Tilting:nabla} 
[T_{\lv}] &=&  [\Delta_{\lv}]+ \sum_{\muv < \lv } \la T_{\lv} , \nabla_{\muv} \ra \, [\Delta_{\muv}]  =  [\Delta_{\lv}]+ \sum_{\muv < \lv } \overline{o}_{\mv, \lv} \, [\Delta_{\muv}]
\ee
We could not find a reference for \eqref{assumption:tilting-o} in the literature, though it does seem to hold at the level of specialized characters from the work of Soergel \cite{soergel:tilting}. Actually, \eqref{assumption:tilting-o} seems to be the content of certain Koszul type dualities at the categorical level (see~\cite[\S1.3]{AMRW1}); it is also consistent with the known multiplicity of $\nabla_{\muv}$ in a $\nabla$-flag of $T_{\lv}$ (the later of which produces character formulas for many $T_{\lv}$, see~\cite[\S7]{soergel:combinatoric} and \cite{soergel:tilting}).

\section{Main result and applications} \label{sec:main-result}

In this section we formulate the main connection between the results in Part \ref{part:combinatorial_models}, \ref{part:padic}, and \S\ref{sec:qg_at_rou}. In \S \ref{sub:main}, the main result is formulated and the consequecnes for Lysenko's Lysenko's conjecture in the `quantum' Geometric Langlands program are draw. In \S \ref{sub:local:shimura:correspondence}, an extension of Savin's local Shimura correspondence to the Whittaker level is formulated. We use this to highlight certain classical behaviors within metaplectic Whittaker spaces connecting to recent work of Gao--Shahidi--Szpruch~\cite{gao:shahidi:szpruch}. In \S \ref{sub:questions}, we describe some  combinatorial properties and a formulate some questions concerning the different basis of $\whitk$.

\subsection{The main result} Let us review the different settings which our main result will bring together.  \label{sub:main}

\tpoint{The generic case}  Fix $(I, \cdot, \mf{D})$ a root datu, written $\wt{\mf{D}}= (Y, \{ \av_i\}, X, \{ a_i \}),$ a $(\Qs, n)$ a twist on $\mf{D}$, and associated twisted root datum $(I, \circ_{(\Qs, n)}, \wt{\mf{D}})$ where $\mf{D}= (\tY, \{ \tav_i\}, X, \{ \ta_i \})$.  Assume that $\wt{\mf{D}}$ is of simply-connected type and construct the Hecke algebra $\taffH:= \affH(\wt{\mf{D}})$ and its module $\V$ as in \S\ref{sub:met-poly}. $\V$ has a decomposition $\V = \oplus_{\etav \in \tnalc} \V(\etav)$ into $\taffH$-submodules and bases $\{ \vv_{\lv} \}$ and $\{Y_{\lv}\}$ for $\lv \in Y$. 

At the spherical level, we have $\tspaff:= \epsilon \taffH \epsilon$ and its module $\Vsp= \{ [w]  \mid w \in \V \},$ where $[w]:= w \epsilon$ for any $w \in \V$. The algebra $\tspaff$ is equipped a basis $\tckl_{\lv}$ while the module $\Vsp$ is equipped with bases $[\vv_{\lv}]$, $\lvecket{\lv}$, and $\lvecketm{\lv}$ (all indexed by $\lv \in Y_+$), as well as their rescaled versions $[Y_{\lv}]$, $\gket{\lv}$, $\gketm{\lv}$ for $\lv \in Y_+$, where \be{} \begin{array}{lccr} [\vv_{\lv}] = \kappa(\lv) \yket{\lv}, & \lvecketm{\lv}= \kappa(\lv) \gketm{\lv}, & \text{and } \lvecket{\lv}= \kappa(\lv) \gket{\lv}, \end{array} \ee and $\kappa(\lv)$ is a product involving the Gauss sum parameters as in \eqref{eq:vv:kappaY}. 

\tpoint{The quantum case} Assume that $\mf{D}$ is of simply-connected type (i.e.,  $\check{\mf{D}}$ is of adjoint type.) For any positive integer $\ell$, define the integer $l$ as in \S\ref{subsub:roots-l} and construct the corresponding quantum group at a primitive $l$-th root of unity $\dU_{\zeta}(\check{\mf{D}})$. 
Write $(\Qs, l)$ for the primitive twist on $\mf{D}$ (see \S\ref{subsub:exampleltwisted}) so that $ (\mf{D}_{(\Qs, l)})^{\vee}$ is equal to $ \check{\mf{D}}_l$, the $l$-twist from the theory of quantum groups at roots of unity.
In \S\ref{subsub:K_0K^Rmodule} we constructed the module $K_0^R(\Rep(\dU_{\zeta}(\check{\mf{D}})))$ for the algebra $\Kz(\Rep(\check{\bG}_{\ell}(\C)))$. The module is equipped with bases of simple modules $\{ [L_{\lv}]\}$, indecomposable tilting modules $[T_{\lv}]$, and costandard modules $\{ [\nabla_{\lv} ] \}$ for $\lv \in Y_+$. 
On the other hand, the algebra $\Kz(\Rep(\check{\bG}_{\ell}(\C)))$ comes equipped with bases $[V_{\etav}], \etav \in Y_{l,+}$, and its (right) action on $K_0^R(\Rep(\dU_{\zeta}(\check{\mf{D}})))$ is denoted by $\odot$,\textit{ i.e.} $[W] \odot [V] = [ W \otimes \Fr(V)]$ for $W \in \Rep(\dU_{\zeta}(\check{\mf{D}}))$ and $V \in \Rep(\check{\bG}_{\ell}(\C))$.

\tpoint{The $p$-adic case} Fix $\psi$ an additive character of conductor $0$ of the local field $\K$ of residue characteristic $q$, an integer $n$ such that $q \equiv 1 \mod 2n,$ and $\epsilon: \bmu_n \rr \C^*$ as in \S\ref{subsub:epsilongenuine}. We assume now that both $(I, \cdot, \mf{D})$ and $(I, \circ_{\qsn}, \wt{\mf{D}})$ are of simply-connected type. In Part \ref{part:padic}, we have constructed a metaplectic group $\tG$ with associated spherical and Iwahori--Hecke algebras of genuine functions $\hec(\tG, \Iop)$ and $\hec(\tG, K)$ as well as their Whittaker modules $\whitiw$ and $\whitk$, respectively. 
The right actions of the Hecke algebras were denoted by $\star$. We recall that $\whitk$ has a natural $p$-adic basis which we write as $\tJ_{\mv}, \mv \in Y_+$ as well as corresponding canonical bases that were denoted as $\tL_{\mv}$ and $\til_{\mv}$. The algebra  $\hec(\tG, K)$ has a natural basis $\tckl_{\lv}$ for $\lv \in \tY_+$ which, under the Satake isomorphism, is mapped to the character of the associated irreducible, highest weight representation $V_{\lv}$ of the group $\bG_{(\wt{\mf{D}})^{\vee}}(\C).$

\spoint Connecting the above constructions, we have the following.

\begin{nthm}\label{thm:main-thm} 
Let $\ell$ be a positive integer satisfying the assumptions in~\S\ref{sub:extension} and $l$ defined as in \S\ref{subsub:roots-l}. Fix $(I, \cdot, \mf{D})$ of simply connected type with primitive twist $(\Qs, l)$ such that $ (\mf{D}_{(\Qs, l)})^{\vee} = \check{\mf{D}}_l$ is also of simply connected type. The maps $\mf{p}: \Zvg \rr \C$ (see \S\ref{notation:padic-spec}) and $\mf{q}:\Zvg \rr \At$ (see \S\ref{notation:quantum-spec}) extend to isomorphisms of $\C$ and $\At$-modules respectively that we continue to denote by the same name, 
\be{}
\mf{p}: \C \otimes_{\Zvg} \Vsp \stackrel{\simeq}{\longrightarrow} \whitk& &\textrm{ sending } [Y_{\mv}] \mapsto \J_{\mv} \,\,\, \textrm{ for } \mv \in Y_+, \text{ and } \\
\mf{q}: \zee[\tau, \tau^{-1}]  \otimes_{\Zvg} \Vsp \stackrel{\simeq}{\longrightarrow} K_0^R(\Rep(\dU_{\zeta}(\check{\bG}))) & &\textrm{ sending } \yket{\mv} \mapsto [\nabla_{\mv}] \,\,\,\mbox{ for } \mv \in Y_+.
\ee 
The map $\mf{p}$ intertwines the $\tspaff$ and $\thsp$ actions, while $\mf{q}$ intertwines the $\tspaff$ and $\Kz(\Rep(\check{\bG}_{\ell}(\C)))$ actions.  
Moreover, $\mf{p}$ maps $\gketm{\mv}$ to $\tL_{\mv}$, $\gket{\mv}$ to $\til_{\mv}$. On the other hand, $\mf{q}$ sends $\gketm{\mv}$ (or $\lvecketm{\mv}$) to $[L_{\mv}].$ 

\end{nthm}

\begin{nrem} In \S \ref{subsub:tiltings:dualities}, we comment on the image of $\gket{\mv}$ under the 'dual' of $\mf{q}$. Let us also mention here that we may work on the $p$-adic side with a twist that is a multiple of a primitive twist, then the main theorem will still hold after choosing $l$ accordingly, see Remark~\ref{rem:nonprimitivetwist}. \end{nrem}

\noindent As for the proof of this theorem: first on the quantum side, the theorem follows from the main results of \S\ref{sec:qg_at_rou}: Theorem~\ref{thm:quantum:tensor:products} and Proposition~\ref{prop:LtoN:Ko}; on  the $p$-adic side, the theorem follows from Proposition \ref{prop:diagrama:met:sph}.

\tpoint{Lysenko's Conjecture}  Within the quantum Langlands program (see ~\cite{ga:twisted}), Lysenko~\cite[Conjecture 11.2.4]{lys} formulated the Casselman--Shalika problem as giving an interpretation of a certain quantity (see~\cite[equation (62)]{lys})-- essentially, the function-sheaf equivalent of our $\tJ_{\mv} \star \tckl_{\lv}$--  in terms of quantum groups at a root of unity. He conjectured a precise relation in Conjecture 11.2.4 of \emph{loc. cit.} and the following result seems to answer the version of his conjecture that can be seen at the level of functions.

\begin{ncor} \label{lys:cor} For $\lv \in Y_{+,l}$, $\mv, \etav \in Y_+$ there exist  $\leftidx^{\gf}Q^{\zv}_{\mv, \lv} \in \Zvg$ (defined uniqutely via ~\eqref{def:gLR}),  such that
\be{} \begin{array}{lcr} 
\tJ_{\mv} \star \tckl_{\lv} = \sum_{\etav} \mf{p} (\leftidx^{\gf}Q^{\lv}_{\etav, \mv})  \tJ_{\etav} 
& \text{ and } & 
\left[ \nabla_{\mv} \right] \odot \left[ V_{\lv} \right] = \sum_{\etav} \mf{q} (\leftidx^{\gf} Q_{\mv, \lv}^{\etav})  \left[ \nabla_{\etav} \right]. \end{array} \ee 
\end{ncor} 

The $p$-adic statement is Theorem \ref{thm:met-geom-CS-LR}, whereas the quantum statement is Proposition \ref{prop:LtoN:Ko}.

\tpoint{Tiltings, dualities, and inner-products} \label{subsub:tiltings:dualities} 
Denote the map $\mf{q}: \At \otimes_{\Zvg} \Vsp \stackrel{\simeq}{\rr} K_0^R(\Rep (\dU_\zeta(\check{\mf{D}}))$ which sends $\vket{\lv} \mapsto \nabla_{\lv}$ by $\mf{q}^L$ now, and let us define a corresponding `left'-variant \be{q:L} \begin{array}{lcr}  \mf{q}^L: \At \otimes_{\Zvg} \Vsp \stackrel{\simeq}{\longrightarrow} K_0^L(\Rep (\dU_\zeta(\check{\mf{D}})) & \text{ by sending }  \overline{\vket{\lv}} \mapsto \Delta_{\lv}, \end{array} \ee where $\overline{\vket{\lv}}$ is the image of $\vket{\lv}$ under the involution $d$. Dualizing  \eqref{eq:lvectovec} one obtains  
\be{} \label{eq:lvectovec:bar}
\qlket{\lv}  &=&  \overline{[\vec_{\lv}]} + \sum_{\muv < \lv } \overline{o}_{\muv, \lv} \, \overline{[\vec_{\muv}]}, \text{ where } \overline{o}_{\muv, \lv}  \in \At^-. 
\ee Applying $\mf{q}^L$ to the right hand side, we obtain 
$[\Delta_{\lv}]+ \sum_{\muv < \lv } \overline{o}_{\muv, \lv} \, [\Delta_{\muv}]$, which under assumption \eqref{assumption:tilting-o} is just equal to $[T_{\lv}]$,\textit{ i.e.} $\mf{q}^L(\gket{\lv}) = [T_{\lv}]$ under this assumption. The map $\mf{q}^L$ may be understoof as follows. The pairing of Proposition \ref{prop:enrichedG}  \be{} K_0^L(\Rep (\dU_\zeta(\check{\mf{D}})) \times K_0^R(\Rep (\dU_\zeta(\check{\mf{D}})) \rr \At \ee is uniquely defined such that $\Delta_{\lv}$ and $\nabla_{\lv}$ are dual bases. Using the map $\mf{q}^L$ and the diagram below, one obtains a unique pairing on $ \At \otimes_{\Zvg} \Vsp$ 
\be{} 	\begin{tikzcd}
		 \At \otimes_{\Zvg} \Vsp \ar[d, "\mf{q}^L"] & \times &  \At \otimes_{\Zvg} \Vsp  \ar[r] \ar[d, "\mf{q}^R"] & \At \\ 
		K_0^L(\Rep (\dU_\zeta(\check{\mf{D}})) & \times &  K_0^R(\Rep (\dU_\zeta(\check{\mf{D}})) \ar[r] & \At 
 \end{tikzcd}  \ee 
such that $\vket{\lv}$ and $\overline{\vket{\lv}}$ are dual bases.

\begin{nrem} 
Presumably there is an inner product on $\whitk$ given in terms of $p$-adic integrals, defined independently of our main result, but consistent with it. Such an inner product would allows us to give a consistent $p$-adic interpretation of all quantum objects $L_{\lv}, \Delta_{\lv}, \nabla_{\lv}$ and $T_{\lv}$ simultaneously. 
\end{nrem}

\subsection{On the local Shimura correspondence}\label{sub:local:shimura:correspondence}

\spoint Let $(I, \cdot, \mf{D})$ be any root datum equipped with a twist $(\Qs, n)$ and corresponding twisted root datum  written now as  $\mf{D}_{\qsn}= (Y_{\qsn}, \{ \av_{\qsn, i} \}, X_{\qsn}, \{a_{\qsn, i} \} ).$ Denote the algebraic group attached to this twisted root datum as $\bG_{\qsn}.$ Write $G_{\qsn}:= \bG_{\qsn}(\K),$ and let $\Iop_{\qsn}$ and $K_{\qsn}$ be Iwahori subgroups and compact subgroups defined in analogy with $\Iop$ and $K$, respectively, where we note that a choice of dominant chamber for the root system defined by $\mf{D}$ also picks one out in $\mf{D}_{\qsn}.$ In this notation, the local Shimura correspondence of Savin \cite{savin:localshimura}, \cf  \eqref{met-iwahori} and \eqref{met-satake}, states  
\be{}\begin{array}{lcr}  \thiw \simeq \hec(G_{\qsn}, I_{\qsn}) & \text{ and } & \thsp \simeq \hec(G_{\qsn}, K_{\qsn}). \end{array} \ee 

\newcommand{\vbin}{\mathbf{v}^{\Iop_{\qsn}}}

\spoint \label{subsub:description:antispherical} Let $\psi$ be an additive character on $\K$ of conductor $0$ extended to the unipotent subgroup $U_{\qsn}$ of $\bG_{\qsn}$ as in \S\ref{subsub:charactersUnipotents}. We may then define  $\whit_{\psi}(G_{\qsn}, I_{\qsn})$ and $\whit_{\psi}(G_{\qsn}, K_{\qsn})$ as the Iwahori and spherical Whittaker spaces for the \emph{linear} group $G_{\qsn}$. The structure of these modules over $\hec(G_{\qsn}, I_{\qsn})$ and $\hec(G_{\qsn}, K_{\qsn})$ was partially reviewed in \S\ref{intro:geom-cs}. Let us also mention here some additional facts. 
\begin{itemize}

\item It is known (see \cite{chan:savin}) that $\whit_{\psi}(G_{\qsn}, I_{\qsn})$ is the anti-spherical module for $\hec(G_{\qsn}, I_{\qsn})$ with  anti-spherical vector $\vbin_{\psi, -\t{\rho}^{\vee}} \in \whit_{\psi}(G_{\qsn}, I_{\qsn}),$ \textit{i.e.} $\whit_{\psi}(G_{\qsn}, I_{\qsn}) \cong \C[Y_{\qsn}]$ as vector spaces and as a module over $\hec(G_{\qsn}, I_{\qsn}) \cong H_W \otimes \C[Y_{\qsn}]$, the structure is specified as follows: $ \C[Y_{\qsn}]$ acts by translation and for each $i \in I$, $H_{s_i}$ acts on $\vbin_{\psi,-\t{\rho}^{\vee}}$ via the scalar $- \tau^{-1}$. 

\item The space $\whit_{\psi}(G_{\qsn}, K_{\qsn})$ has a basis $\J_{\qsn, \mv}$ for $\mv \in Y_{\qsn, +}$ and $\hec(G_{\qsn}, K_{\qsn})$ has basis $\ckl_{\lv}$ for $\lv \in Y_{\qsn, +}$ (see  \eqref{geom:cs-2}). One has $\J_{\qsn, 0} \star \ckl_{\lv} = \J_{\qsn, \lv}$.  

\end{itemize} 

\spoint \label{subsub:shimura:iwahori} We may decompose $\whitiw \cong \C[Y]$ into $\thiw$ submodules $\oplus_{\etav \in \tnalc} \whitiw(\etav)$ as in  Proposition \ref{prop:VtoP}. Using Theorem \ref{thm:i-basis} together with the observation that $-\rhov \da w = - \rhov$ for all $w \in W$, we conclude that $\whitiw(-\rhov)$ has rank one as a module over $\C[\tY] = \C[Y_{\qsn}] \subset \thiw$ and the $H_{s_i}, i \in I$ act on $\x_{-\rhov}$ via the scalar $- \tau^{-1}$. The following  may be seen as an Iwahori--Whittaker variant of Savin's local Shimura correspondence. It follows by using the description of $\whit_{\psi}(G_{\qsn}, I_{\qsn})$ in the previous paragraph.
\begin{nprop} \label{prop:classicalIwahoriCS} The map sending $\x_{-\rhov} \mapsto \vbin_{\psi, -\t{\rho}^{\vee}}$ extends to a unique  isomorphism of $\thiw$-modules \be{}  \whitk(-\rhov) \stackrel{\simeq}{\longrightarrow} \whit_{\psi}(G_{\qsn}, \Iop_{\qsn}). \ee  \end{nprop}

\begin{nrem}Note that this isomorphism must send $\x_{-\rhov + \t{\rho}^{\vee}}= \x_{-\rhov} Y_{\t{\rho}^{\vee}}$ to  $ \vbin_{\psi,-\t{\rho}^{\vee}} \mathscr{Y}_{\t{\rho}^{\vee}}  = \vbin_{\psi,0}$. This computation explains the appearance of the coweight $-\rhov +\t{\rho}^{\vee}$ in the next paragraph.
\end{nrem}

\spoint \label{subsub:classicalCS} From Corollary~\ref{cor:decomposition:whitk}, we have a decomposition $\whitk = \oplus_{\etav \in \tnalc } \whitk(\etav)$ into $\thsp$-modules. One observes that $\whitk(-\rhov):= \whitiw({-\rhov}) \star \ek$ is a free rank one $\thsp$-module. 
Within it, consider the elements $\tJ_{-\rhov + \t{\rho}^{\vee} + \mv} = \x_{-\rhov + \t{\rho}^{\vee} + \mv} \star \ek$ for $\mv \in \tY_+ = Y_{\qsn, +}$, where we note that $\mv \in \tY_+$ implies $-\rhov + \t{\rho}^{\vee} + \mv \in Y_+$. The spherical extension of the previous result can now be stated: 

\begin{nprop}\label{cor:CSisospherical}
There is an isomorphism of $\thsp$-modules  
\be{} \whitk(-\rhov) \stackrel{\simeq}{\longrightarrow} \whit_{\psi}(G_{\qsn}, K_{\qsn}) \ee  
which is defined by sending $\tJ_{\t{\rho}^\vee-\rhov} \mapsto \J_{\qsn, 0}$. For $\muv \in \tY_+$, one has the following analogue of \eqref{geom:cs}: 
\be{eq:McGertyspherical}
\tJ_{\t{\rho}^\vee-\rhov} \star \tckl_{\muv} = \tJ_{\t{\rho}^\vee-\rhov + \muv},
\ee
\end{nprop}
\noindent Equation~\eqref{eq:McGertyspherical} together with Theorem~\ref{thm:met-geom-CS-LR} imply  
\begin{ncor}
Let $\muv \in \tY_+$. Then $\tJ_{\t{\rho}^\vee-\rhov + \muv} = \tL_{\t{\rho}^\vee-\rhov+\muv} = \til_{\t{\rho}^\vee-\rhov+\muv}$ .
\end{ncor}

\tpoint{Quantum group interpretation}

On the quantum group side, Corollary~\ref{cor:CSisospherical} and equation~\eqref{eq:McGertyspherical} correspond to the fact that in $\Rep(\dU_\zeta (\check{\mf{D}}))$ the representation $L_{\t{\rho}^\vee-\rhov+\muv}$ for $\muv \in \tY_+$ is equal to the costandard and indecomposable tilting corresponding to the same weights (\cf~\cite[Corollary 6.8]{mcgerty:cmp} and~\cite{andersen:CMP}). In particular, equation~\eqref{eq:McGertyspherical} `corresponds' on the quantum side to~\cite[Lemma 5.2]{mcgerty:cmp}. The space $L_{\t{\rho}^\vee-\rhov}$ is called the Steinberg representation and plays an important role on the quantum side (see~\cite{andersen:CMP}). 

\tpoint{Relation to the $p$-adic literature}
An asymptotic or usual Casselman--Shalika formula (in the terminology of our introduction) was found in ~\cite{gao:shahidi:szpruch} for the same coweight $\t{\rho}^\vee-\rhov$ at the level of unramified Whittaker functions on the metaplectic group. 
In ~\cite[Section 5]{bbbg:iwahori:metaplectic:duality} a similar result is shown using integrable systems techniques for certain metaplectic covers of $\opn{GL}_r$. Actually in \emph{op. cit.} multiple coweights where similar phenonemon occur are shown. This stems from the fact that $\mf{gl}_r$ is not semisimple. In the semi-simple case, one expects that $\t{\rho}^\vee-\rhov$ will be the only point at which such phenomenon can occur.

We also note here a proof of the classical (non-metaplectic) Casselman--Shalika formula in terms of the Steinberg-Lusztig theorem was given in~\cite[Theorem 2.3]{laniniram}.

\subsection{Whittaker functions and $\gf$-twisted  combinatorics}\label{sub:questions}

In this section we discuss in more detail the  $\gf$-coefficients which arise in our work and mention a few questions that can be solved using the connections to the theory of quantum groups. We fix the same notation and hypotheses as in \S\ref{sub:local:shimura:correspondence}.

\tpoint{The $\mv$-large asymptotic Casselman--Shalika formula}

Let $\mv \in Y_+$ and $\lv \in \tY_+$. As in \eqref{padic:Lys}, write
\be{} \begin{array}{lcr}
\tJ_{\mv} \star \tckl_{\lv} = \sum_{\zv \in Y_+} \leftidx^{\g} Q^{\zv}_{\mv, \lv} \tJ_{\zv} & \text{or equivalently} & \yket{\mv} \star \tckl_{\lv} = \yket{\mv} \diamondsuit \chi_{\lv} = [ Y_{\mv} \cdot \chi_{\lv}(Y)] = \sum_{\zv \in Y_+} \leftidx^{\gf} Q^{\zv}_{\mv, \lv} \yket{\zv}. 
\end{array} \ee 
In general, straightening rules may be needed to rewrite the product $[Y_{\mv} \cdot \chi_{\lv}(Y)]$ in terms of the basis $\yket{\lv}$ with $\lv \in Y_+$. However, if $\mv$ is large compared to $\lv$ in the sense that $\mv + \zv \in Y_+$ for any $\zv$ appearing as a weight of the representation $V_{\lv}$, then no such straightening relations are necessary. We shall  say that $\mv$ is \emph{$\lv$-stable} if this condition is satisfied. Equations~\eqref{y:c:diamond} and~\eqref{eq:qQ:gfQ} can be used to show:

\begin{nprop}
Let $\lv \in \tY_+$ and let $\mv \in Y_+$ be $\lv$-stable. Writing $\chi_{\lv}(Y) = \sum_{\zv\in \tY} a_{\lv,\mv} Y_{\mv}$, for $ a_{\lv, \mv} \in \zee$, 
\be{}\begin{array}{lcr}
 \yket{\mv} \star \ckl_{\lv} = \sum_{\zv \in \tY} \leftidx^\gf a_{\lv,\zv}  \yket{\mv + \zv} & 
\mbox{ where } & \leftidx^\gf a_{\zv,\muv} := \kappa(\muv)\kappa(\lv)^{-1} a_{\zv,\muv}. 
\end{array}\ee  
\end{nprop}

There are well-known  combinatorial interpretations of the coefficients $a_{\lv, \zv}.$  For example, in type $A$ the Weyl character is written as a sum over Gelfand--Tsetlin patterns and the coefficients are the number of certain Young Tableaux (there are generalizations of this result to orthogonal and symplectic characters). This result and Proposition~\ref{prop:diagrama:met:sph} produce simple combinatorial formulas for $\tJ_{\lv} \star \t{c}_{\zv}$ in the stable range:
\be{}
\tJ_{\mv} \star \t{c}_{\lv} = \sum_{\zv \in \tY} \leftidx^\g a_{\lv,\zv}  \tJ_{\mv + \zv}.
\ee 
Similar formulas for $\tJ_{\lv} \star \t{h}_{\zv}$ have been used  in the study of Weyl group multiple Dirichlet series~\cites{bbf:wgmds:stable,bbf:wgmds:stable:applications}.

\tpoint{On $\gf$-twisted LLT polynomials  }
The quantum specializations $Q_{\muv, \lv}^{\etav}$ are known to have non-negative coefficients by the work of Grojnowksi and Haiman~\cite[Theorem 5.9]{haiman:grojnowski}. 
We may use $\leftidx^{\gf}Q^{\etav}_{\mv, \lv}$ to define $\gf$-twisted LLT polynomials by using the formula after \cite[Eq. (13)]{haiman:grojnowski}. Among the connections between LLT polynomials and Macdonald polynomials, let us note that it is proven in \emph{op. cit.}  that the Schur polynomial expansion of transformed Macdonald polynomials are positive by using an expression of transformed Macdonald polynomials in terms of LLT polynomials. This suggests there might exist  connections between $\gf$-twisted LLT polynomials and the metaplectic Macdonald polynomials of~\cite{ssv1, ssv2}.

\tpoint{On the inverse $\gf$-twisted Kazhdan--Lusztig polynomials} \label{subsub:strong:linkage}
Consider the basis $\tJ_{\lv}$ and $\tL_{\lv}$ of the Whittaker space $\whitk$ introduced in \S\ref{subsub:action:thsph:on:twhitk}. 
If we define $ \leftidx^\g d_{\lv, \muv}  \in \C$ as the expansion coefficients in
\be{} \tJ_{\lv} = \sum_{\muv \in Y_+} \leftidx^\g d_{\lv, \muv} \tL_{\muv},
\ee 
then it is a natural question to understand when these coefficients are non-zero. 

As in the representation theory of quantum groups, we say that $\lv \in Y_+$ is \emph{strongly linked} to  $\mv \in Y_+$ (\cf\cite[\S3]{Andersen:linkage}) if there exists a chain 
$\muv = \muv_1, \muv_2,\cdots, \muv_r=\lv$ 
of elements in $Y_+$ such that 
\be{} 
\begin{array}{lcr} 
\muv_{j-1} = \muv_i \da  s_{\betav_j} + m_j \, l_{\betav_j} \betav_j \leq \muv_i & \text { for } & \betav_j \in \check{\rts}_+, m_i \in \zee, i = 2,\cdots, r 
\end{array}. 
\ee 
In the equation above $l_{\alpha_j}$ is the integer $l_j$ defined in \S\ref{subsub:exampleltwisted} when $\alpha_j$ is a simple positive coroot. Otherwise, for any coroot $\betav$, there exists a $w\in W$ such that $\betav = w(\alphav_j), j \in I$ and we set  $l_{\betav}=l_{\alphav_j}$. 
The following is a  $p$-adic analogue of the strong linkage principle from the theory of quantum groups \cf\cite{AndersenPoloWen, Andersen:linkage}.

\begin{nprop}
	If the coefficient $\leftidx^\g d_{\lv, \muv}\neq0$, then $\muv$ is strongly linked to $\lv$. 
\end{nprop}

\begin{proof}
	By the discussion in \S\ref{subsub:Vsp-g}, the coefficients $\leftidx^\g d_{\lv, \muv}\neq0$ will be non-zero if and only if their quantum version $d_{\lv, \muv}$ is non-zero (they will differ by a quotient of $\kappa$ factors which is product of Gauss sums). The quantum coefficients are non-zero if a filtration of $\nabla_{\lv}$ by irreducibles contains $L_{\muv}$. The result follows by the strong linkage principle \cf\cite[Theorem 8.1]{AndersenPoloWen}\cite[Theorem 3.13]{Andersen:linkage}.  
\end{proof}

\begin{nrem}  
	It might be interesting to see if there is a purely $p$-adic proof of the above result.  We also remark that a similar result may be stated for the relation between $\tJ_{\lv}$ and $\tT_{\lv}$, whose proof would again rest of the validity of \eqref{assumption:tilting-o}. 
\end{nrem}

\tpoint{On the principal submodule $\whitk(0)$} 
In \S\ref{subsub:classicalCS} we related the subspace $\whitk(-\rhov)$ with the non-metaplectic version of the spherical Gelfand--Graev representation. 
Let us now look the \emph{principal} subspace $\whitk(0) $.
Let $\lv \in \tY_+ \subset 0 \da \taffW \cap Y_+$. 

Recall that $[Y_{\lv}] := Y_{\lv} \epsilon$ can be thought of as an element in $\Zvg[Y]$ obtained by acting with $\epsilon$ on $Y_{\lv}$, where $\epsilon$ is defined in~\eqref{def:ep} and the action of the $H_{s_i}$ on $\Zvg[Y]$ is the one from \S\ref{subsub:met-poly}. On the other hand, in \eqref{action:parabolic}, we introduced an action of $H_W$ on $\At[\tY]$ (this is \textit{not }the quantum action considered above). Using it, to each $\mv \in \tY$, we define an  element $[\tY_{\mv}]:= \tY_{\mv} \epsilon$ where $\tY_{\mv}$ is a typical element in the group algebra $\At[\tY]$.

Denote by  $\Pi_{(\Qs,n)} : \Zvg[Y] \to \Zvg[\tY]$ the natural projection sending $Y_{\mv} \mapsto 0$ if $\mv \notin \tY$ and $Y_{\mv} \mapsto Y_{\mv}$ if $\mv \in \tY$. Now, one expects (as in \cites{fh:dual, mcgerty:cmp}) that  $\Pi_{(\Qs,n)} ([Y_{\lv}]) \in \C[\tY]$ can be expanded to an identity in $\Zvg[\tY]$.
\be{}\label{eq:projection:algebraicFHM} 
\Pi_{(\Qs,n)} ([Y_{\lv}]) = \sum_{\muv} \leftidx^{\gf}m_{\lv}^{\muv}\, [\tY_{\mv}]. 
\ee 

\noindent Specializing $\tau \mapsto -1$ in the identity above, which on the quantum group side means working with the regular Grothendieck group instead of the enriched one, the corresponding specialized coefficients, denoted $\leftidx^{\g}m_{\lv}^{\muv},$ will be equal to (up to signs) the Littlewood--Richardson coefficients $c_{\t{\rho}^\vee-\rhov, \lv}^{\muv + \t{\rho}^\vee-\rhov}$ for the group $\check{G}_{\qsn}$. This essentially follows by comparing the work of Frenkel--Hernandez~\cite{fh:dual} and McGerty~\cite{mcgerty:cmp} (especially Proposition 4.4 and Theorem 5.5 in \emph{loc. cit.}) to our setting.  We pose the problem to study the coefficients $\leftidx^{\gf}m_{\lv}^{\muv}$ for generic $\tau$ and give them a $p$-adic interpretation.

\bibliographystyle{abbrv}
\bibliography{variants}

\end{document}